\PassOptionsToPackage{a4paper,margin=25mm}{geometry}

\RequirePackage{iftex}
\ifpTeX
  \PassOptionsToPackage{dvipdfmx}{graphicx}
  \PassOptionsToPackage{dvipdfmx}{hyperref}
\fi

\documentclass[12pt]{article}
\usepackage{da-common}
\usepackage{paper2-h4}
\usetikzlibrary{arrows.meta}

\title{The Fourth Geometry I:\\
Difference--Angle Geometry Beyond Euclid, Hyperbolic, and Elliptic}

\author{%
Masanori Nakazato\\
Mita International School of Science, Tokyo, Japan\\
\texttt{masa727axio.math@gmail.com}\\[6pt]
\small Formerly at Graduate School of Science, Tohoku University\\
\small (Master's Program Completed)
}

\date{November 3, 2025}

\begin{document}

\maketitle

\begin{center}
\small
\textit{Preprint (arXiv:2511.01024 [math.MG]). Part~I of the series ``The Fourth Geometry''.\\
Version 2: December~2025.}
\end{center}


         \begin{abstract}
In this work, we introduce a new geometry based on the difference angle,
an angle defined as the difference of slopes of two lines, together with an
axiomatic system for angles.  
This framework provides a constructive approach to the fundamental
question ``What is an angle?'', and shows that an angle can be defined
independently of circles or rotations, as a primary geometric notion.

Within this geometry, one can define difference-angle triangles, norms,
bisectors, perpendiculars, and inner products.  
Several characteristic properties not seen in existing geometries emerge,
together with behaviors analogous to those in Euclidean geometry:
the triangle inequality always holds with equality, the sum of the interior
angles of any triangle is $0$, and a Miquel point exists even for parabolas.
In particular, the concurrency of the parabolic Miquel configuration was
suggested by Weiss and Odehnal, and our main theorem
provides the first explicit and rigorous confirmation of this assertion.
We also point out that many classical Euclidean configurations
(including Brocard-type configurations) naturally reappear in the setting
of difference-angle geometry.

These results indicate that difference-angle geometry is a promising
candidate for a ``Fourth Geometry'' following the Euclidean, hyperbolic,
and elliptic geometries.
\end{abstract}

\begin{notation}
We refer to the Euclidean concepts translated into the framework of difference--angle geometry as
“difference--angle \textcircled{ }\textcircled{ }.” When no confusion arises, we simply write “\textcircled{ }\textcircled{ }” omitting the qualifier
“difference--angle.”

Throughout this paper we use the following notation:

\begin{itemize}
  \item $\R$: the set of real numbers.
  \item $\PiC$: the projection of a curve $\mathcal{C}$ (the exact definition will be given at the relevant point).
  \item $\Slp{XY}$: the slope of the line through points $X$ and $Y$. It is also used for the slope of a line $L$, written as $\Slp{L}$.
  \item $\triangle_{\mathcal{P}} ABC$: a \emph{difference--angle triangle} with vertices $A,B,C$ on the parabola $\mathcal{P}$, whose edges are not parallel to the projective direction.
  \item $\mathrm{Int}(\triangle_{\mathcal{P}} ABC)$, $\mathrm{Ext}(\triangle_{\mathcal{P}} ABC)$: the \emph{interior} and \emph{exterior} of the triangle $\triangle_{\mathcal{P}} ABC$, defined in \cref{int-ext}.
  \item $\measuredangle_{\mathcal{P}} XYZ$: the difference angle at vertex $Y$ of the triangle $\triangle_{\mathcal{P}} XYZ$.
  \item $\measuredangle_{\mathcal{P}} A$, $\measuredangle_{\mathcal{P}} B$, $\measuredangle_{\mathcal{P}} C$: abbreviations of $\measuredangle_{\mathcal{P}} BAC$, $\measuredangle_{\mathcal{P}} CBA$, and $\measuredangle_{\mathcal{P}} ACB$, respectively.
 \item $\|AB\|_{\mathcal P}$ : the \emph{difference–angle norm} of the segment $AB$ (see \Cref{def:diff-angle-norm}).
  \item $\theta_A$, $\theta_B$, $\theta_C$: variables used to denote the above three interior angles concisely in calculations and proofs.
\end{itemize}

\end{notation}


\section{Introduction}
\subsection{Motivation}

In this paper, we reconsider the notion of an angle in geometry.  
While modern mathematics has rigorously axiomatized quantities such as distance and area,  
no explicit axiomatic foundation has ever been formulated for the concept of an angle itself.  

The author's initial motivation for confronting this question dates back roughly thirty years,  
to a well-known property of parabolas learned in school:  
for the parabola $y = \kappa x^2$ and a line $\ell$ passing through a point $P(p, q)$  
that intersects the parabola at $x = a$ and $x = b$,  
the product $(p - a)(p - b)$ remains constant regardless of the choice of $\ell$.  

This elegant trick is familiar to many Japanese students as a standard technique in high-school entrance examinations.  
Because the statement closely resembles the power-of-a-point theorem for circles,  
we shall refer to it here as the parabolic power.  

At first, this fact was merely a source of problems for mathematical exercises.  
However, the realization that the construction on a parabola can be expressed entirely  
in terms of differences of $x$-coordinates became a turning point.  
While studying two tangents to the parabola $y = \kappa x^2$  
and the line passing through their points of tangency,  
the author noticed an intriguing 2:1 ratio  
between the difference of the $x$-coordinates and the difference of the tangent slopes.  
This relation seemed analogous to that between a central angle and its corresponding inscribed angle in a circle.  
It is also well known that when the two tangents are perpendicular,  
their intersection point lies on the directrix.  
Yet the author was left with a simple question:  
where does this orthogonality reside as the tangent points recede infinitely far away?  

Around the same period, the author became deeply interested in the parallelogram theorem.  
Although it appears unrelated to circles at first glance,  
its derivation from Ptolemy's theorem reveals a surprising kinship between parallelograms and circles.  
Moreover, the structure connecting the inner product  
$\vec{OA}\cdot\vec{OB}$ of vectors from a common origin $O$ to points $A,B$  
with the power of $O$ with respect to a circle suggested a trinity  
linking the notions of inner product, power, and the parallelogram theorem.  

In 2013, the author conjectured that a similar power structure  
might also exist for other conic sections, particularly for hyperbolas.  
This led to the definition of the hyperbolic power and the discovery of its fundamental properties.  
By that time, the author's focus had shifted to a broader question:  
within a geometry where a power theorem holds,  
might it be possible to define a new kind of inner product?  
Although the idea was difficult to communicate to others,  
the author published a short article about the hyperbolic power in 2014  
in a magazine for Japanese middle-school students, but the work did not advance further.  

Since the author's background was in number theory rather than geometry,  
and after many years away from academic mathematics,  
it became clear that addressing such foundational questions would be difficult.  
The author continued asking professional geometers questions such as:  
What is an angle?  
What would an axiomatic system of angles look like?  
Is the notion of a right angle essential to the concept of angle itself?  
The only answers received were “we do not know” or “there is none.”  

This question remained dormant for years, persisting like an unresolved echo.  

In July 2025, through a renewed comparative study of various geometric frameworks,  
the author once again confronted the problem—  
discovering that in Hilbert's axiomatic system,  
angles are defined only as figures formed by two segments,  
yet no axioms governing angles themselves are postulated.  

Revisiting an observation first noticed in 2014,  
the author explored whether angles could instead be constructed  
from differences of slopes:

\begin{observation}
Let $A(a, 0)$ and $B(b, 0)$ be two points on the $x$-axis,  
and let $P(x, y)$ with $x\neq a,b$.  
If the difference angle $\measuredangle_{\mathcal P} APB$—defined as the difference of slopes—remains constant,  
then the point $P$ lies on a parabola.
\end{observation}

Starting from this realization, we construct a new geometry  
whose fundamental quantity is an angle defined by the difference of slopes,  
and from this foundation we seek to axiomatize the notion of angle itself.  
We call this geometry the Difference–Angle Geometry (abbreviated as DA geometry),  
and aim to establish its constructive principles and theoretical framework  
encompassing angles, powers, norms, and inner products.

\subsection{Structure of the Main Theorems}

The main theorems developed in this paper proceed in three major stages:

\begin{enumerate}
  \item By \Cref{mthm:diff-angle-bisector}, we obtain a theorem on angle bisectors  
        as one of the first consequences of the difference angle  
        and of the generally non-symmetric difference–angle norm.
  \item Building on this, we derive the \Cref{mthm:DABCT}  
        (Difference–Angle Bisector Collinearity Theorem),  
        from which the concept of isogonality is naturally reconstructed within DA geometry.
  \item Introducing a hierarchy of similarity and congruence,  
        we show that the strongest form of congruence implies collinearity,  
        which leads to the final statement \Cref{mthm:parabolic-final-theorem}.
\end{enumerate}

Among these developments, two achievements are particularly noteworthy.

\begin{itemize}
  \item Theoretical progress within DA geometry.  
    The progression from \Cref{thm:ceva_diff_angle} (DA version of Ceva's theorem)  
    to \Cref{mthm:DABCT} (Difference–Angle Bisector Collinearity Theorem)  
    constitutes the central development grounded in the axioms of DA geometry.
  \item External consequences and exportability.  
    Some theorems in DA geometry, such as \cref{lem:bis-midpoint-safe},  
    have no Euclidean counterpart,  
    while others can be exported as new Euclidean theorems—  
    for instance, the path from \Cref{mthm:DABCT}  
    to \Cref{thm:Angle-Bisector-Collinearity-Theorem}.  
    This demonstrates that DA geometry is not a mere reformulation of Euclidean geometry  
    but an independent and self-contained geometric system.
\end{itemize}

\subsection{Philosophical Foundations of DA Geometry}

The purpose of this subsection is to clarify the conceptual foundations that characterize DA geometry.  
Its framework rests on two fundamental pillars.

\medskip
The first pillar originates from Hilbert's Foundations of Geometry\cite{Hilbert1899}, particularly Axiom~III$_5$,  
which states that if two sides and the included angle of one triangle are congruent respectively to those of another triangle,  
then the remaining corresponding angles are also congruent.  
Together with Axiom~III$_4$ (the transfer of angles),  
Hilbert derived the SAS congruence theorem.  
The underlying philosophy of this construction is that congruence should not be regarded merely as a metric equivalence,  
but as a relation determined through constructive operations.  
Hilbert emphasized that congruence must be uniquely fixed by such constructive processes (“superposition”).  
Within this framework, an angle can always be transferred in coordination with its adjacent sides,  
and as long as its external form coincides, it is regarded as identical.  
This strong requirement, however, renders the angle a derived quantity—  
one dependent on other geometric entities rather than primitive in itself.

\medskip
DA geometry relaxes this restriction under the parabolic limit (in the sense of Cayley--Klein degeneration),  
introducing a stratified hierarchy of congruence according to geometric conditions,  
and thereby reconstructing the angle as a primary geometric quantity.

From this standpoint, a new structure emerges that subdivides Euclidean geometry itself.  
One of its most symbolic manifestations is the following phenomenon of collinearity.  
(Here $x_T$ denotes the $x$-coordinate of a point $T$.)

\begin{observation}
Let $C\colon y=\kappa x^2\ (\kappa>0)$ be a parabola, and let $D$ be another parabola obtained by translating $C$.  
Take three points $A,B,C$ on $C$ with $x_A<x_B<x_C$, and three points $C',B',A'$ on $D$ with $x_{C'}<x_{B'}<x_{A'}$, satisfying
\[
x_B-x_A = x_{A'}-x_{B'}, \qquad x_C-x_B = x_{B'}-x_{C'}.
\]
Let $H_A$ be the point on line $BC$ whose $x$-coordinate equals $x_{A'}$,  
$H_B$ the point on line $CA$ whose $x$-coordinate equals $x_{B'}$, and  
$H_C$ the point on line $AB$ whose $x$-coordinate equals $x_{C'}$.  
Then the three points $H_A,H_B,H_C$ are collinear.
\end{observation}

This collinearity follows naturally from the projective invariance inherent in DA geometry  
and represents one of the most fundamental phenomena revealed under the parabolic perspective developed in this paper.  
In the final section, we shall provide a geometric proof of this result without any algebraic computation.

\medskip
The second pillar concerns the treatment of the line at infinity latent in affine geometry.  
Whereas classical geometry has regarded the line at infinity as a homogeneous, monolithic entity,  
DA geometry divides it into distinct components.  
Specifically, by distinguishing the point at infinity corresponding to the projective reference direction ($\infty_G$)  
from that corresponding to the singular direction ($\infty_S$),  
singular behavior emerges even within the finite region.  
This bifurcation gives rise to phenomena absent in conventional geometries,  
such as the constant equality in the triangle inequality and the multiplicity of zero vectors.  
This viewpoint of a \textit{fracture at infinity}\footnote{That is, a conceptual splitting of the ideal line in the affine–projective sense, which will be further developed in the forthcoming paper on Hilbert's Fourth Problem(H$_4$)}  
distinguishes DA geometry as an autonomous framework,  
rather than a mere deformation of Euclidean or hyperbolic geometry.

\medskip
These two pillars—the hierarchical stratification of congruence and the fracture of the line at infinity—  
may seem independent at first glance, yet both embody the same underlying principle:  
to subdivide geometry itself.  
By making explicit the latent, unresolved structures within Euclidean and affine geometry,  
DA geometry opens a new branch of geometric thought.

\medskip
Moreover, DA geometry represents an attempt to axiomatize the angle as a primary geometric quantity,  
positioning itself as an instance of ``geometry based on invariants'' in the sense of Klein's Erlangen Program.

The new structure presented in this study not only encompasses many familiar properties of Euclidean geometry  
but also reveals novel phenomena that do not appear within it.  
This suggests the possibility of reexamining the definition of angle across other geometric systems.

Finally, by focusing on the structure in which straight lines can be regarded as geodesics arising from the norm secondarily defined from the DA,  
DA geometry can be connected with Hilbert's Fourth Problem—  
the classification of geometries in which straight lines are geodesics—  
from the viewpoint of a pre-Finsler structure, that is,  
a Finsler-type geometry not necessarily symmetric nor homogeneous.  
While this problem has been solved analytically under certain regularity assumptions,  
the present study proposes a new connection through the axiomatization of angle.  
Further discussion of this aspect is deferred to a separate paper (the H$_4$ manuscript).

\subsection{Projective Geometry and the Problem of Angle}

In pure projective geometry, angles are not defined,  
since the framework preserves only incidence relations and not metric quantities such as distance or angle.  
However, as demonstrated in the Cayley--Klein construction,  
both distance and angle can be introduced by incorporating an external reference.

Although DA geometry is not a metric geometry in the Cayley--Klein sense,  
it belongs to the same lineage:  
by introducing a projective reference line and a projective direction,  
it defines an angle as the difference of the slopes of two half-lines.  
Thus, DA geometry inherits from the Cayley--Klein framework  
the essential idea of constructing metric notions by adding an external reference to projective geometry.

Moreover, in DA geometry,  
the ideal point at infinity bifurcates into two distinct ideal points—  
one corresponding to the projective reference direction and the other to the singular direction.  
This bifurcation gives rise to a new kind of structural fracture  
absent in classical geometries, emphasizing the distinctiveness of the present framework.  
A more detailed comparison will be provided in a separate paper.  

In this way, DA geometry can be regarded as a parabolic extension  
within the lineage of projective geometry,  
exhibiting a structure in clear contrast to that of the Cayley--Klein geometries.


\section{Undefined Terms and Primitive Notions}

In this paper, we work on the basic Euclidean plane $\mathbb{R}^2$,  
fixing a projective direction $d$ and a projective reference line $\ell$.  
These are referred to as the projective direction and the projective reference line, respectively.

\begin{itemize}
  \item The pair $(\ell, d)$ is called the projective reference structure.
  \item In this paper, the underlying space of our geometry is the real two-dimensional affine plane $\mathbb{R}^2$. All geometric objects, such as points, lines, and angle quantities,
are defined within this plane, and the projective reference structure
$(\ell, d)$ is fixed inside $\mathbb{R}^2$.

  \item Using finitely many points $A,B,C,\dots$ on the plane, we shall construct line segments and difference angles.
  \item For a point $P$, let $P'$ denote the intersection of the line through $P$ along $d$ with $\ell$;
        the point $P'$ is called the projection of $P$ onto $\ell$.
  \item At this stage, we only prepare the notion of ``difference of slopes,''  
        while the concept of an ``angle'' itself is not yet defined.
\end{itemize}

\begin{figure}[H]\label{fig:proj-structure}
\centering
  \includegraphics[scale=0.8]{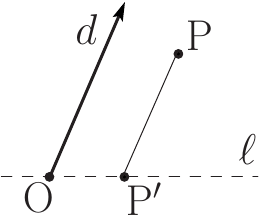}
  \caption{Projective gauge $(\ell,d)$ with projection.}
\end{figure}
\medskip

Within the projective reference structure, the following conventions are implicitly assumed:

\begin{enumerate}
  \item The projective direction $d$ does not coincide with the direction of the reference line $\ell$;  
        that is, $d \neq \mathrm{dir}(\ell)$.  
        The direction $\mathrm{dir}(\ell)$ is called the \textbf{singular direction},  
        and it is excluded from the domain of definition.
  \item Accordingly, when $d = \mathrm{dir}(\ell)$,  
        neither the line ``along $d$'' nor the projective slope $\pslp{OX}{XY}$ is defined.  
        The treatment of this situation will later be specified in the boundary policy,  
        either as an absorptive or divergent type.
\end{enumerate}

Under these assumptions, the notions introduced later—  
the projective slope \\
$\pslp{OX}{XY}$, the angular quantity $\angle(A,P,B)$,  
and the difference angle $\measuredangle_{\mathcal{P}}$—  
will all be uniquely determined within the domain $d \neq \mathrm{dir}(\ell)$.


\section{Axiomatic System}
\bigskip

\begin{quote}
First, as one of the geometric axiom systems, we construct and propose an axiomatic framework that treats the angle as a primary quantity,
based on a comparative study of the properties of angles across various geometries including the Euclidean one.
In particular, Hilbert adopted a strong form of the congruence axiom (III$_5$: the SAS axiom).%
Instead of assuming it directly, we separate the congruence of sides and angles and formulate them independently.
As a result, the accompanying axiom III$_4$ (the transfer of an angle) is eliminated in the present system.

By first introducing an axiom system for primary quantities,
then defining a secondary quantity in the form of a norm,
and finally presenting the axiom of congruence,
the logical flow of this framework formally parallels
Hilbert's construction in the \textit{Foundations of Geometry}.
Within this standpoint, we define the difference angle
$\measuredangle_{\mathcal{P}}$,
and reconstruct the Euclidean congruence conditions
as angle–primary theorems in the parabolic limit.

\end{quote}
\begin{axiom}[I1]\label{ax:I1}
For any two distinct points, there exists at least one line passing through them.
\end{axiom}

\begin{axiom}[I2]\label{ax:I2}
Through any two distinct points, there passes exactly one line.
\end{axiom}

\begin{axiom}[I3]\label{ax:I3}
Every line contains at least two distinct points.
\end{axiom}

\begin{axiom}[I4]\label{ax:I4}
If three points are not on the same line, there exists a plane containing them.
\end{axiom}

\begin{axiom}[I5]\label{ax:I5}
Every plane contains at least one line.
\end{axiom}

\begin{axiom}[I6]\label{ax:I6}
If two lines share a point and lie in the same plane, then they intersect.
\end{axiom}

\begin{axiom}[I7]\label{ax:I7}
There exist at least four points in space that do not lie in the same plane.
\end{axiom}

\begin{axiom}[O1]\label{ax:O1}
If $A$–$B$–$C$, then the three points are collinear, and $B$ is uniquely determined
as the point lying between $A$ and $C$.
\end{axiom}

\begin{axiom}[O2]\label{ax:O2}
For any two points $A$ and $C$, there exists at least one point $B$
lying between them.
\end{axiom}

\begin{axiom}[O3]\label{ax:O3}
If one of three points lies between the other two,
then the three points are collinear.
\end{axiom}

\begin{axiom}[O4]\label{ax:O4}
Given three non-collinear points $A,B,C$ forming the ordered triple $(A,B,C)$,
let $P$ be a point between $A$ and $B$, and let $\ell$ be a line through $P$
not passing through $C$.  
Then $\ell$ meets either the segment $AC$ or the segment $BC$.
\end{axiom}

\begin{axiom}[P1]\label{ax:P1}
On every plane, at least one point $O$ can be chosen as a reference point.
\end{axiom}

\begin{axiom}[P2]\label{ax:P2}
Every line on the plane possesses a notion of direction,
defined by an ordered pair of distinct points $(A,B)$ on it,
interpreted as “the direction from $A$ to $B$.”
\end{axiom}

\begin{axiom}[P3]\label{ax:P3}
Through the reference point $O$,
there exist at least two distinct directions on the plane.
\end{axiom}

\begin{axiom}[P4]\label{ax:P4}
Among all lines through the point $O$,
one may be chosen and designated as the \emph{projective base line} $\ell$.
\end{axiom}

\begin{axiom}[P5]\label{ax:P5}
For any point $X$ on the plane and any direction $d$
different from that of the base line $\ell$,
there exists exactly one line through $X$ parallel to direction $d$.
\end{axiom}

\begin{axiom}[P6]\label{ax:P6}
For any point $X$ not lying on the projective reference line $\ell$, 
and for any direction $d$ distinct from $\ell$, 
let $Y$ be the intersection of $\ell$ with the line through $X$ in direction $d$.
Then, the ratio determined by the pair $(OX;\ XY)$, formed by the segments $OX$ and $XY$, 
defines a quantity called the \emph{projective slope}.
\end{axiom}

\begin{axiom}[PAR1]\label{ax:PAR1}
For any point $P$ and any line $\ell$ not passing through $P$,
there exists exactly one line through $P$ parallel to $\ell$
—that is, having the same or opposite direction as $\ell$.
\end{axiom}


\begin{figure}[H]
  \centering
  \includegraphics[scale=0.8]{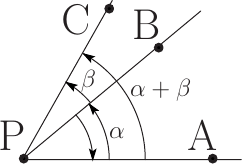}
  \caption{Axiom A2 (additivity) and A1 (antisymmetry).}
  \label{fig:A2-and-A1}
\end{figure}

\begin{axiom}[A1: Opposite Angles (Order Reversal)]\label{ax:A1}
For any points $A,P,B$,
\[
\angle(A,P,B)=-\,\angle(B,P,A)\,.
\]
\end{axiom}

\begin{axiom}[A2: Additivity]\label{ax:A2}
Let $A,B,C$ be collinear with $B$ between $A$ and $C$,
and let $P$ be a point not on this line such that
the angles can be continuously formed.
Then
\[
\angle(A,P,B)+\angle(B,P,C)=\angle(A,P,C)\,.
\]
\end{axiom}

\begin{axiom}[A3: Vanishing (Definition of a Straight Angle)]\label{ax:A3}
If $P,A,B$ are collinear in the order $P,A,B$ or $A,B,P$,
then $\angle APB=0$; conversely, if $\angle APB=0$,
then $P,A,B$ are collinear.
\end{axiom}

\begin{axiom}[A4: Scaling Invariance]\label{ax:A4}
If $A'\in \overrightarrow{PA}$ and $B'\in \overrightarrow{PB}$,
then $\angle APB=\angle A'PB'$.
\end{axiom}

\begin{axiom}[A5: Continuous Divisibility]\label{ax:A5}
Fix a point $P$.
Let $\mathcal R_P$ denote the topological space of all oriented rays from $P$,
and let $S_P\subset\mathcal R_P$ be a singular subset.
Define $D_P:=\mathcal R_P\setminus S_P$.
For each connected component $C$ of $D_P$,
there exists an angular mapping $\angle:C\times C\to\mathbb{R}$ satisfying:
\begin{enumerate}
\item[(i)] \textbf{Bisection.}
For any $r,s\in C$, there exists $t\in C$ such that
\[
\angle(r,t)=\angle(t,s)=\tfrac12\,\angle(r,s).
\]
\item[(ii)] \textbf{Continuity.}
For each fixed $r\in C$, the map $s\mapsto\angle(r,s)$
is continuous on $C$.
\end{enumerate}
\end{axiom}

\begin{remark}[Reason for adopting A5(ii)]
The condition A5(i) alone (bisection) does not in general imply continuity
(because Jensen–type pathological examples exist).
Therefore, A5(ii) is explicitly included as an independent axiom.
In continuous models (including those defined in Calib later),
A5(ii) is automatically satisfied by construction,
but declaring it at the axiomatic level excludes pathological models
and clarifies logical independence.
\end{remark}

\begin{convention}[Boundary Policy with Covering Group]\label{conv:boundary}
The axioms are defined on $D_P=\mathcal R_P\setminus S_P$;
the behavior on $S_P$ itself is not axiomatized.
For each geometry, one may specify a value space $V$
and a covering map $p:\mathbb{R}\to V$
with a deck transformation group (periodic lattice)
$\Lambda\le(\mathbb{R},+)$ (discrete),
and adopt one or more of the following conventions:
\begin{itemize}
  \item \textbf{Lift type:}
  For each connected component $C\subset D_P$,
  take a continuous lift $\widetilde{\Angle}:C\times C\to\mathbb{R}$,
  and define $\Angle = p\circ \widetilde{\Angle}$.
  The lift is $\Lambda$-periodic
  ($\widetilde{\Angle}\sim \widetilde{\Angle}+\lambda,\ \lambda\in\Lambda$).
  When $\Lambda=\{0\}$, non-periodic real-valued angles are included.
  \item \textbf{Absorbing type:}
  As $s\to S_P$, the angle satisfies $\Angle(r,s)\to 0$;
  singular directions are absorbed (collapsed).
  The range may be $V=\mathbb{R}$ or $\overline{\mathbb{R}}$.
  \item \textbf{Divergent type:}
  As $s\to S_P$, $\Angle(r,s)\to \pm\infty$,
  and no continuous extension is made.
  The range is $V=\overline{\mathbb{R}}$.
\end{itemize}
Note: depending on the connected components or singularities,
these policies may be applied separately or jointly.
\end{convention}

\begin{remark}[Choice of Boundary Policy]
When extending an angle continuously,
three general boundary policies may be considered:
absorbing, lift, and divergent.
In the DA geometry treated here,
the divergent type is incompatible with the angle axioms
and does not yield a well-defined angle quantity;
thus only the absorbing and lift types are admissible.
\end{remark}

In the present framework, the angle is treated as a primary quantity,
while length is introduced as a secondary one.
Accordingly, Hilbert's system of congruence axioms
is modified as follows.
Hilbert's CONG4 served as an existence axiom
for defining angles as derived quantities;
here it becomes unnecessary,
since the angle is already established independently by A1–A5.
Therefore, the system of congruence is closed
by CONG1–3 together with the distributive axiom CONG5.

\begin{axiom}[CONG1: Transfer of Segments (Existence)]\label{ax:CONG1}
For any two points $A,B$ and any point $A'$ on a line,
there exists a point $B'$ such that
the segment $\overline{AB}$ is congruent to $\overline{A'B'}$.
\end{axiom}

\begin{axiom}[CONG2: Transitivity of Segment Congruence]\label{ax:CONG2}
If $\overline{AB}\cong\overline{A'B'}$
and $\overline{AB}\cong\overline{A''B''}$,
then $\overline{A'B'}\cong\overline{A''B''}$.
\end{axiom}

\begin{axiom}[CONG3: Identity]\label{ax:CONG3}
Every segment is congruent to itself.
\end{axiom}

\begin{definition}[Secondary Length Structure]\label{def:length-structure}
For each line $L$, assign to every segment on $L$
a positive real number through a map $|\cdot|_L$
satisfying the following:
\begin{enumerate}
\item[(L1)] \textbf{Homogeneity:}
For a similarity $h_\lambda:A\mapsto P+\lambda(A-P)$
with center $P$ and ratio $\lambda>0$,
\[
|h_\lambda(A)h_\lambda(B)|_{h_\lambda(L)}=\lambda\,|AB|_L.
\]
\item[(L2)] \textbf{Compatibility with the Angular Structure:}
$|XY|_L$ depends continuously on the direction $\overrightarrow{XY}$,
and remains invariant under automorphisms preserving
the angle axioms A1–A5.
\item[(L3)] \textbf{Symmetry:}
$|XY|_L=|YX|_L$.
\end{enumerate}
\end{definition}

\begin{axiom}[CONG5: Distributive Axiom of ASA Type]\label{ax:CONG5}
Let two triangles $\triangle ABC$ and $\triangle A'B'C'$ satisfy
\[
\angle ABC=\angle A'B'C',\quad
\angle ACB=\angle A'C'B',\quad
|BC|_{L}=|B'C'|_{L'}\,.
\]
Then
\[
|AB|_{L}=|A'B'|_{L'},\quad
|AC|_{L}=|A'C'|_{L'}\,.
\]
In other words, the length structure is distributively compatible
with the angular structure (ASA congruence).
Here $L$ and $L'$ are length structures identified
via an automorphism preserving the angular axioms A1–A5.
\end{axiom}

\begin{remark}[Relation to Hilbert's System]
Hilbert's III$_4$ (transfer of angles)
was necessary in his framework
because angles were defined as derived quantities
from the distance structure.
In the present framework, since the angle is a primary quantity,
its transfer follows automatically from A1–A5.
Hence CONG4 is not required as an independent axiom,
and CONG5 serves as the first distributive relation
linking the angular and length structures.
(The number is reserved for consistency across versions.)
\end{remark}

\begin{proposition}[Independence of (P1--P6)]\label{prop:p1-p6-indep}
On the Euclidean plane $\mathbb{R}^2$, if segment congruence is taken to be the usual one
induced by the Euclidean distance and angle congruence is defined as “equality of
Euclidean angle measure,” then
\Cref{ax:CONG1,ax:CONG2,ax:CONG3} hold, whereas
\Cref{ax:CONG5} does not hold in general (in particular, for angles with
$0<\theta<\pi$ there are two symmetric transfers).
\end{proposition}

\begin{proof}
We assume I1--I5 as background.\footnote{%
The I- and P-families presuppose distinct structures; neither family
is derivable from the other. Hence no additional verification of “mutual
independence” between them is required.}
Define a base model $\mathcal M_0$ as follows:
the plane $\mathbb R^2$, reference point $O=(0,0)$, reference line
$l=\{(x,0)\}$ (the $x$-axis); the set of directions consists of all
orientations (the unit circle). For any $X$ and any direction $d\neq l$,
there is a unique line through $X$ in direction $d$.
Fixing a plane with $O$ and the baseline $\ell$, for $X\notin l$ and
$d\neq l$ let $Y=l\cap \ell(X,d)$. We call the ratio-type quantity
\[
  \pslp{OX}{XY}
\]
the \emph{projective slope} of $OX$ with respect to $XY$.

We regard $\pslp(OX;XY)$ as well-defined by the usual $x$-coordinate ratio.
Below, in each case we keep the structure of $\mathcal M_0$ except for the single
axiom we intend to violate.

\begin{enumerate}[label=\text{Case P\arabic*:}, leftmargin=*, widest=6]

\item (P1: existence of a reference point $O$) \\
Model $\mathcal M_{\neg P1}$:
remove the constant symbol $O$ from the language; everything else as in $\mathcal M_0$.
Directions (P2), choice of $\ell$ (P4), uniqueness of the line through $X$ with
direction $d\neq l$ (P5), and the definition of $\PSLP$ (P6) still hold, but P1 is false
because $O$ does not exist. Hence P1 is independent.

\item (P2: line direction given by “from $A$ to $B$”) \\
Model $\mathcal M_{\neg P2}$:
give a global direction set $D$ (e.g.\ the unit circle) and do not define
the direction of a line via an ordered pair of points $(A,B)$.
Keep $O,l$ as in $\mathcal M_0$; satisfy P4 and P5 using $D$.
For $X\notin l$ and $d\neq l$, the intersection $Y$ exists and
$\pslp{OX}{XY}$ is well-defined with respect to $D$.
Thus P1 and P3–P6 are true while only P2 is false. Hence P2 is independent.

\item (P3: at least two directions from $O$) \\
Model $\mathcal M_{\neg P3}$:
restrict the set of admissible directions from $O$ to $D_O=\{d_0\}$
(a single direction). Directions from other points (or the global $D$) remain
as in $\mathcal M_0$.
Then P1, P2, P4, P5, P6 remain valid, while “at least two directions from $O$”
fails. Hence P3 is independent.

\item (P4: specification of the baseline $\ell$) \\
Model $\mathcal M_{\neg P4}$:
treat two distinct lines $\ell_1\ne \ell_2$ through $O$ both as “baselines,”
thus breaking uniqueness. With all else as in $\mathcal M_0$, P1, P2, P3, P5, P6 hold and only P4 fails.
Hence P4 is independent.
(\emph{Remark:} If P4 were stated as a pure “existence” axiom, one could also make it false by not specifying any baseline at all.)

\item (P5: uniqueness of the line through $X$ in direction $d\neq l$) \\
Model $\mathcal M_{\neg P5}$:
fix one direction $d^\ast\neq \mathrm{dir}(l)$ and provide, for each point $X$,
\emph{two} distinct lines through $X$ in direction $d^\ast$ (duplicate the same orientation).
Keep uniqueness for all other directions.
Then P1–P4 and P6 continue to hold, while uniqueness in P5 fails.
Hence P5 is independent.

\item (P6: well-defined projective slope $\PSLP$) \\
Model $\mathcal M_{\neg P6}$:
for $X\notin l$, $d\neq l$, the intersection $Y=l\cap \ell(X,d)$ exists, but define
$\pslp{OX}{XY}$ to be \emph{multi-valued and orientation-dependent}:
for instance, assign different values when $XY$ is oriented $X\to Y$ versus $Y\to X$
(or depending on which side along $\ell$ the point $Y$ is approached).
Then P1–P5 hold, but $\PSLP$ is not uniquely defined; P6 is false.
Hence P6 is independent.
\end{enumerate}
\qedhere
\end{proof}

\begin{proposition}[Independence of PAR and A]\label{prop:par-ang-indep}
\Cref{ax:PAR1} and \cref{ax:A1,ax:A2,ax:A3,ax:A4,ax:A5} are independent.
\end{proposition}

\begin{proof}
Following Hilbert, we show that none of these axioms can be derived from the others  
by constructing counterexample models\footnote{Geometric structures in which only the target axiom fails while all others hold.}.

\medskip
As the standard reference model $\mathcal{M}_0$,  
we take the Euclidean plane $\mathbb{R}^2$ with the circumparabola $y=x^2$,  
the projection direction along the $x$-axis, and define
\[
\measuredangle_{\mathcal P}(\ell,m)=\mathrm{slp}(\ell)-\mathrm{slp}(m),
\qquad
|XY|_{\mathcal P}=|x_Y-x_X|
\]
(the singular lines being those parallel to the $x$-axis).  
In the following, we modify only the angular measure as needed,  
keeping all other structures (points, lines, order, etc.) fixed as in $\mathcal{M}_0$.

\begin{enumerate}[label=(\roman*), align=left, leftmargin=3em]

\item[PAR1 (Uniqueness of parallels)]  
Starting from $\mathcal{M}_0$, allow \emph{two} distinct “lines in direction $d$”  
through each point and each given direction $d$.  
All other I- and P-axioms remain valid, but PAR1 fails.  
Hence PAR1 is independent.

\item[A1 (Antisymmetry)]  
Define the angular measure by the absolute value
\[
\angle^{\mathrm{abs}}(\ell,m)=\big|\mathrm{slp}(\ell)-\mathrm{slp}(m)\big|.
\]
Additivity (for unoriented angles), degeneracy, and scaling invariance are preserved,  
but $\angle(\ell,m)=-\angle(m,\ell)$ does not hold; thus A1 fails.  
Therefore A1 is independent.

\item[A2 (Additivity)]  
Fix a threshold $s_0$ and define the piecewise-linear measure
\[
\angle^{\mathrm{pc}}(\ell,m)=
\begin{cases}
\alpha\big(\mathrm{slp}(\ell)-\mathrm{slp}(m)\big) & \text{if } \mathrm{slp}(\ell),\mathrm{slp}(m)\le s_0,\\[2pt]
\beta \big(\mathrm{slp}(\ell)-\mathrm{slp}(m)\big) & \text{if } \mathrm{slp}(\ell),\mathrm{slp}(m)> s_0,
\end{cases}
\qquad (\alpha\neq\beta).
\]
Antisymmetry (A1), degeneracy (A3), and scaling invariance (A4) remain valid,  
but for configurations where $B$’s slope crosses the threshold,  
$\angle(\ell,m)+\angle(m,n)\neq\angle(\ell,n)$,  
so A2 fails. Hence A2 is independent.

\item[A3 (Vanishing)]  
(Assume the domain of A1 is restricted to $\ell\neq m$.)  
Define
\[
\angle^{(\varepsilon)}(\ell,\ell)=\varepsilon,\qquad
\angle^{(\varepsilon)}(\ell,m)=\mathrm{slp}(\ell)-\mathrm{slp}(m)\ \ (\ell\neq m),
\]
where $\varepsilon>0$ is a fixed constant.  
Then antisymmetry (for $\ell\neq m$), additivity (for distinct $\ell,m,n$),  
and scaling invariance hold,  
but A3 fails. Thus A3 is independent.

\item[A4 (Scaling invariance)]  
Let the angular measure depend on a \emph{length scale} $\lambda>0$ as
\[
\angle^{\lambda}(A,P,B)=\lambda\cdot\big(\mathrm{slp}(PA)-\mathrm{slp}(PB)\big).
\]
Under geometric scaling transformations,  
the angle measure is no longer invariant, so A4 fails,  
while A1–A3 remain valid. Hence A4 is independent.

\item[A5 (Continuous divisibility)]  
Let $\theta$ be a continuous angular coordinate on $D_P$,  
and set the singular set $S_P:=\theta^{-1}(\mathbb{Z})$.  
Define $\Angle(r,s):=\lfloor\theta(s)\rfloor-\lfloor\theta(r)\rfloor$.  
Then A1–A4 hold, but since the value set of $\Angle$ is $\mathbb{Z}$,  
when $\Angle(r,s)=1$ there exists no midpoint $t$ satisfying $\Angle(r,t)=\tfrac12$,  
contradicting A5(i).  
Moreover, $\Angle(r,\cdot)$ is constant on each connected component,  
hence continuous; thus A5(ii) holds.
\end{enumerate}
\qedhere
\end{proof}

\begin{proposition}[Independence of CONG5]\label{prop:cong5-indep}
The axiom \textup{CONG5} (the distributive axiom of ASA type) is not derivable
from all the other axiom groups (I, O, P, A, and CONG1--3).
\end{proposition}

\begin{proof}
As in Hilbert's system where III$_5$ (the SAS axiom) is independent, in the present
framework the angular structure (A1--A5) and the length structure
(\Cref{def:length-structure}) are constructed independently.
Hence removing \textup{CONG5} entails no contradiction among the remaining axioms.
Indeed, a \emph{restricted model} that retains only the angular structure and does not
introduce any length structure $|\cdot|_L$ (corresponding to the foundational
DA geometry) satisfies A1--A5 and CONG1--3 but not CONG5.
Therefore \textup{CONG5} is an independent axiom.
\end{proof}

\begin{maintheorem}[Independence of the Axiom System]\label{mthm:axiom-independence}
The axiom system presented above is independent.
\end{maintheorem}

\begin{proof}
\Cref{ax:I1}--\Cref{ax:I7}, \Cref{ax:O1}--\Cref{ax:O4}, and
\Cref{ax:CONG1}--\Cref{ax:CONG3} are unaffected by the projective baseline
structure; hence Hilbert's independence arguments---which guarantee that the
groups of distance, order, and incidence axioms do not depend on each other---
extend to our setting without contradiction.
Moreover, \Cref{prop:p1-p6-indep}, \Cref{prop:cong5-indep}, and
\Cref{prop:par-ang-indep} establish the remaining cases.
\end{proof}


\section{Definitions}

In this chapter, we define the difference angle, the central notion of DA geometry.  
Unlike the Euclidean or Hilbertian angle, it is not derived from distance or rotation,  
but is a purely projective quantity of first order, defined independently of any metric.

Given a projective reference structure,  
let the direction along the projective reference line be denoted by $\ell$,  
and another direction not parallel to $\ell$ be called the projective direction $d$.  

Starting from this chapter, we regard the underlying space as the real
two-dimensional affine plane~$\mathbb{R}^2$, and we again employ the
projective reference structure~$(\ell,d)$ introduced earlier.
Within this setting, we define the difference angle.
Consequently, the difference angle becomes a concrete model satisfying
the axiomatic system of angles established in Chapter~3.

Through this pair $(\ell,d)$, the projective reference structure carries  
a vector space structure isomorphic to $\mathbb{R}^2$,  
so that all discussions can be made on the $xy$-plane.  
However, the difference angle and the difference norm (defined later)  
are independent of the choice of basis in this vector space.

\begin{figure}[htbp]
  \centering
  \includegraphics[scale=0.8]{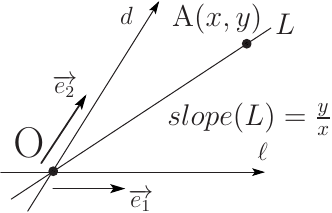}
  \caption{Slope of a line and the definition of the difference angle.}
  \label{fig:slope-of-line-L}
\end{figure}

\subsection{Preparation: The nature of angle — difference of slopes (coordinate-free definition)}

Fix a projective reference structure $(\ell,d)$.  
Let $\varphi,\psi\in(\mathbb{R}^2)^{\ast}$ be covectors satisfying
\[
\varphi(d)=0,\quad \varphi|_{\ell}\neq 0,\qquad
\psi|_{\ell}=0,\quad \psi(d)\neq0.
\]
For a direction vector $v$ of a line $\ell$, define its slope by
\[
\Slp{\ell}:=\frac{\psi(v)}{\varphi(v)}.
\]
Then, for two lines $\ell_1,\ell_2$ passing through a vertex $B$,  
the difference angle between them is defined by
\[
\measuredangle_{\mathcal P}(\ell_1,B,\ell_2):=\lambda\bigl(\Slp{\ell_2}-\Slp{\ell_1}\bigr),
\]
where the sign is determined by the direction of rotation not crossing $d$.  
Here $\lambda>0$ is a constant that fixes the unit of angular measure.  
For convenience one may set $\lambda=1$, which corresponds to taking  
the difference of slopes itself as the primitive angular quantity.

\paragraph{Invariance and coordinate representation.}
Under a coordinate transformation preserving $(\ell,d)$,
\[
(x,y)\mapsto(\alpha x+\beta,\ \gamma y+\delta), \qquad (\alpha,\gamma>0),
\]
the slope transforms as $s\mapsto(\gamma/\alpha)s$,  
so the difference angle is uniquely defined once $\lambda$ is fixed.  
Hence, for convenience, we may take $\ell$ as the $x$-axis and $d$ as the $y$-direction,  
in which case for a line $\ell_1$ we have
\[
\Slp{\ell_1}=\frac{\Delta y}{\Delta x},
\]
so that the slope coincides with the usual one in Cartesian coordinates.

\begin{remark}
Let the basis vectors be $\vec e_1=(1,0)$ and $\vec e_2=(0,1)$.  
Then the line $\ell_1=x\vec e_1+y\vec e_2$ ($x>0$) has slope $y/x$.  
Moreover, for two points $A=(a_1,a_2)$ and $B=(b_1,b_2)$ with $a_1<b_1$, one has
\[
\Slp{AB} = \frac{b_2-a_2}{b_1-a_1}.
\]
\end{remark}

\begin{remark}
The difference angle is determined linearly by the ratio of $(\varphi,\psi)$,
and therefore it does not depend on any units of length or area.  
Secondary structures, such as inner products and areas, will be discussed in Base~2 and later papers.
\end{remark}

\begin{remark}[Assumption on coordinate systems]
Throughout this paper, for the reader's convenience,  
all constructions, theorems, and proofs in DA geometry  
are presented within the orthogonal Cartesian coordinate system $(x,y)$.  
Nevertheless, the constructions themselves are defined for any projective reference structure  
and can be extended to general affine or projective geometries.  
When changing to another coordinate system (for example, an oblique coordinate frame),  
the corresponding transformation-matrix correction  
(such as the determinant factor for area) can be applied to obtain equivalent results.  
Hence, all statements remain invariant under affine transformations that preserve the reference structure $(\ell,d)$.
\end{remark}

\subsection{System of Definitions for the Difference Angle}

\begin{definition}[Segment and Ray]\label{def:segment-ray}
Let $A,B$ be two distinct points on a line.
\begin{itemize}
  \item The set consisting of $A,B$ and all points between them is called a \textbf{segment}, denoted by $AB$.  
        Unless otherwise stated, both endpoints $A,B$ are included.
  \item When a line is divided at a point $A$,  
        the set consisting of $A$ and all points on one side containing a given point $B$  
        is called a \textbf{ray}, denoted by $AB$.  
        Unless explicitly stated, the endpoint $A$ is included.
\end{itemize}
\end{definition}

\begin{figure}[htbp]
  \centering
  \begin{subfigure}{0.45\textwidth}
    \centering
    \includegraphics[scale=0.8]{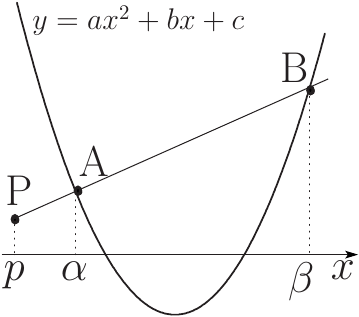}
    \caption{Parabolic power theorem.}
    \label{fig:Parabolic-power-theorem}
  \end{subfigure}
  \hfill
  \begin{subfigure}{0.45\textwidth}
    \centering
    \includegraphics[scale=0.8]{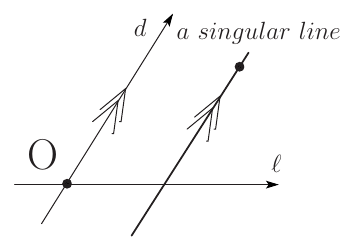}
    \caption{Singular line parallel to $d$.}
    \label{fig:singular-line}
  \end{subfigure}
  \caption{Fundamental notions related to the parabolic power.}
\end{figure}

\begin{theorem}[Parabolic Power Theorem]\label{thm:Parabolic-power-theorem}
Let $C: y=ax^2+bx+c$ be a parabola in the $xy$-plane,  
and let $P=(p,q)$ be a point in the plane.  
Let $\ell$ be any line passing through $P$ and intersecting $C$ at two points  
with $x$-coordinates $\alpha$ and $\beta$.  
Then the product $(\alpha-p)(\beta-p)$ is constant,  
independent of the choice of $\ell$.
\end{theorem}

\begin{proof}
Let $\ell$ be given by $y=mx+n$.  
Since the intersection points satisfy $ax^2+bx+c=mx+n$,  
the two roots $\alpha,\beta$ yield
\begin{align*}
	ax^2+bx+c-mx-n &= a(x-\alpha)(x-\beta), \\[3pt]
	q &= mp+n, \\[3pt]
	a(\alpha-p)(\beta-p) &= ap^2+bp+c-mp-n \\[3pt]
	&= ap^2+bp+c-q \quad \text{(constant)}.
\end{align*}
\end{proof}

\begin{definition}[Singular Line]\label{def:singular-line}
A line parallel to the projective direction $d$ is called a \textbf{singular line}.
\end{definition}

A singular line plays a special role in the computation of difference angles.  
Since its properties differ from those of ordinary lines,  
unless explicitly stated otherwise, singular lines are excluded from the notion of “line” in what follows.

\begin{figure}[htbp]
  \centering
  \includegraphics[scale=1.0]{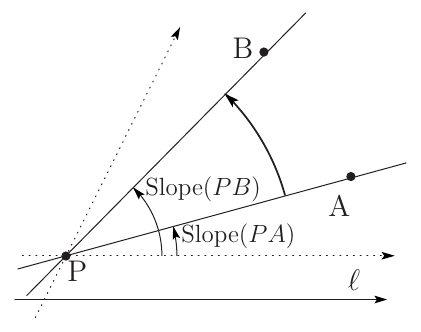}
  \caption{Definition of the difference angle.}
  \label{fig:def-diff-angle}
\end{figure}

\begin{definition}[Difference Angle]\label{def:differece_angle}
The \textbf{difference angle} $\measuredangle_{\mathcal{P}} APB$ is defined by
\[
\measuredangle_{\mathcal{P}} APB := \Slp{PB} - \Slp{PA}.
\]
It is an oriented (signed) quantity,  
and may, when appropriate, be considered modulo $\pi/2$.  
The following boundary policy is adopted:\footnote{
Although it also admits a lift-type interpretation with $\Lambda=\frac{\pi}{2}\mathbb{Z}$.}
\[
\textbf{Absorptive type:} \quad s\to S_P \ \text{implies}\ \Angle(r,s)\to 0.
\]
\end{definition}

\subsection{Well-definedness of the Difference Angle and Isoptic Curve}

For any three points $A,B,C$,  
the difference angle (difference of slopes) centered at $B$ is defined as
\[
\measuredangle_{\mathcal{P}} ABC
= \Slp{BA} - \Slp{BC}
= \frac{y_A - y_B}{x_A - x_B} - \frac{y_C - y_B}{x_C - x_B}.
\]
This definition is independent of the projective direction $d$,  
although the sign and orientation of the angle depend on the chosen reference structure $(\ell,d)$,  
which becomes important in subsequent geometric constructions.

The well-definedness of the difference angle as an angular measure follows from the following properties:

\begin{enumerate}
  \item \textbf{Order Reversal (A1):}
  \(\measuredangle_{\mathcal P}CBA = -\measuredangle_{\mathcal P}ABC.\)
  \item \textbf{Additivity (A2):}
  For any point $D$ on the segment $BC$,
  \[
  \measuredangle_{\mathcal P}ABC
  =\measuredangle_{\mathcal P}ABD+\measuredangle_{\mathcal P}DBC.
  \]
  \item \textbf{Collinearity (A3):}
  If $A,P,B$ are collinear, then \(\measuredangle_{\mathcal P}APB=0.\)
  \item \textbf{Scaling Invariance (A4):}
  Under isotropic scaling $(x,y)\mapsto(kx,ky)$,  
  the slopes are invariant, hence so is the difference angle.
\end{enumerate}

\begin{proposition}[Subtractive Additivity]\label{prop:subtractive_additivity}
If $C$ lies on the segment $AB$ and the corresponding angles are defined continuously, then
\[
\measuredangle_{\mathcal{P}}APC
=\measuredangle_{\mathcal{P}}APB-\measuredangle_{\mathcal{P}}CPB.
\]
\end{proposition}

\begin{proof}
By definition,
\begin{align}
\measuredangle_{\mathcal P}APB &= \Slp{PB}-\Slp{PA},\\
\measuredangle_{\mathcal P}APC &= \Slp{PC}-\Slp{PA},\\
\measuredangle_{\mathcal P}CPB &= \Slp{PB}-\Slp{PC}.
\end{align}
Hence,
\begin{align}
\measuredangle_{\mathcal P}APB-\measuredangle_{\mathcal P}CPB
&= \bigl(\Slp{PB}-\Slp{PA}\bigr)-\bigl(\Slp{PB}-\Slp{PC}\bigr)\\
&= \Slp{PC}-\Slp{PA}
= \measuredangle_{\mathcal P}APC.\qedhere
\end{align}
\end{proof}

\begin{remark}[Choice of Coordinate System]
The definition of the difference angle is coordinate-free  
and valid under any projective reference structure $(\ell,d)$.  
However, to maintain consistency with the figures and later computations,  
we shall, unless otherwise stated, fix $\ell$ as the $x$-axis  
and $d$ as the positive $y$-direction.  
This causes no loss of generality,  
since any other case can be transformed into this one by a suitable linear transformation.
\end{remark}

\begin{figure}[htbp]
  \centering
  \begin{subfigure}{0.45\textwidth}
    \centering
    \includegraphics[scale=0.8]{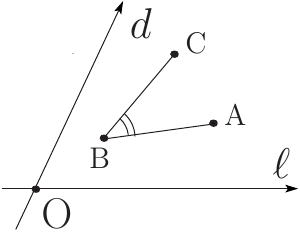}
    \label{fig:gauge-projective}
  \end{subfigure}
  \hfill
  \begin{subfigure}{0.45\textwidth}
    \centering
    \includegraphics[scale=0.8]{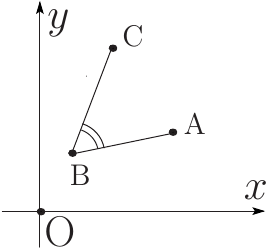}
    \label{fig:gauge-normalization}
  \end{subfigure}
  \caption{Normalization of the projective gauge.}
\end{figure}

\begin{figure}[htbp]
  \centering
  \begin{subfigure}{0.45\textwidth}
    \centering
    \includegraphics[scale=0.8]{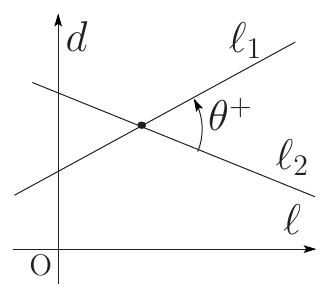}
    \caption{Positive difference angle.}
    \label{fig:positive-diff-angle}
  \end{subfigure}
  \hfill
  \begin{subfigure}{0.45\textwidth}
    \centering
    \includegraphics[scale=0.8]{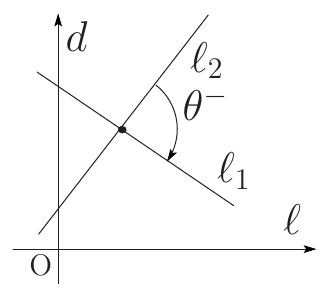}
    \caption{Negative difference angle.}
    \label{fig:negative-diff-angle}
  \end{subfigure}
  \caption{Positive and negative difference angles.}
\end{figure}

\begin{definition}[Positive and Negative Difference Angles]\label{def:positive-and-negative-diff-angle}
Fix the projective direction $d$. For two lines $\ell,m$ with slopes $\Slp{\ell}$ and $\Slp{m}$, set:
\begin{itemize}
\item If $\Slp{\ell} > \Slp{m}$, then the difference $\Slp{\ell} - \Slp{m}$ is called a \textbf{positive difference angle}, denoted by $\theta^{+}$.
\item If $\Slp{m} > \Slp{\ell}$, then the difference $\Slp{\ell} - \Slp{m}$ is called a \textbf{negative difference angle}, denoted by $\theta^{-}$.
\end{itemize}
A difference angle takes two values of opposite sign; the sign is determined by the rotation that does not cross the projective direction $d$.
\end{definition}

\begin{proposition}\label{prop:positive-diff-angle-equivalence}
Let $\ell,m$ be two lines with $\Slp{\ell}\Slp{m}\neq 0$ and $\Slp{\ell}\neq\Slp{m}$. Then the following are equivalent:
\begin{enumerate}[label=\alph*)]
\item The opening for measuring the difference angle of $(\ell,m)$ is a positive difference angle.
\item The opening for measuring the difference angle of $(\ell,m)$ lies entirely on one side with respect to the projective direction $d$.
\end{enumerate}
\end{proposition}

\begin{proof}
This follows immediately from \cref{def:positive-and-negative-diff-angle}.
\end{proof}

\begin{definition}[Absolute Difference-Angle Value]\label{def:angle_value}
Given a difference angle $\measuredangle_{\mathcal{P}}APB$, its nonnegative magnitude
is called the \textbf{difference-angle value} of $\measuredangle_{\mathcal{P}}APB$ and is denoted by $|\measuredangle_{\mathcal{P}}APB|$.
\end{definition}

With these definitions, the isoptic curve in DA geometry for a segment $AB$ parallel to the projective reference line is a parabola. This corresponds to the classical Euclidean fact that the isoptic of a segment $AB$ is an arc of a circle (or two circles).

\begin{figure}[htbp]
  \centering
  \includegraphics[scale=0.8]{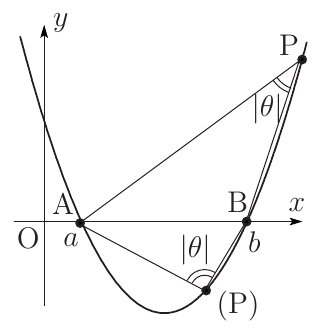}
  \caption{Proof of the parabola locus theorem.}
  \label{fig:isoangle-result}
\end{figure}

\begin{maintheorem}[Iso-angle Locus]\label{mthm:parabola-locus}
Let $A=(a,0)$ and $B=(b,0)$ with $a<b$. If a point $P$ satisfies the condition that the difference angle $\measuredangle_{\mathcal{P}}APB=\theta>0$ is constant, then its locus (including the endpoints $A,B$) is the parabola
\[
y \;=\; \frac{\theta}{\,b-a\,}\,(x-a)(x-b).
\]
\end{maintheorem}

\begin{remark}[Isoptic Curve]
The locus of points from which the chord $AB$ is seen under a fixed angle is classically called an \textbf{isoptic curve}. The parabola in \cref{mthm:parabola-locus} is the isoptic in the setting of DA geometry.
\end{remark}

\begin{figure}[htbp]
  \centering
  \begin{subfigure}{0.45\textwidth}
    \centering
    \includegraphics[scale=0.8]{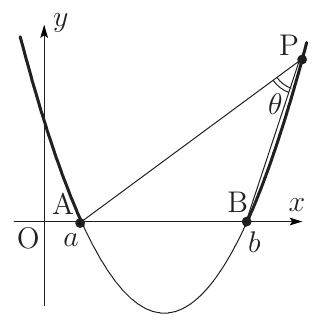}
    \caption{Iso-angle set above the $x$-axis.}
    \label{fig:isoangle-pos}
  \end{subfigure}
  \hfill
  \begin{subfigure}{0.45\textwidth}
    \centering
    \includegraphics[scale=0.8]{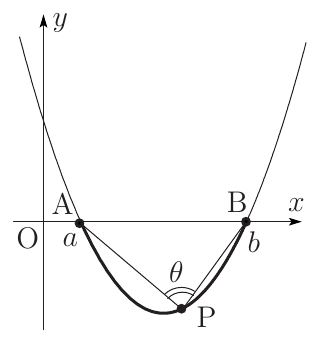}
    \caption{Iso-angle set below the $x$-axis.}
    \label{fig:isoangle-neg}
  \end{subfigure}
  \caption{Further proof of the parabola locus theorem.}
\end{figure}

\begin{proof}
Let $P=(x,y)$. Then
\[
\Slp{AP}=\frac{y}{x-a},\qquad
\Slp{BP}=\frac{y}{x-b},
\]
and hence
\[
\measuredangle_{\mathcal{P}}APB
=\Slp{PB}-\Slp{PA}
=\frac{(b-a)y}{(x-a)(x-b)}.
\]
Imposing $\measuredangle_{\mathcal{P}}APB=\theta$ gives
\[
y=\frac{\theta}{\,b-a\,}(x-a)(x-b),
\]
which is the required parabola.
\end{proof}

\begin{remark}
For the case $\measuredangle_{\mathcal{P}}APB=-\theta$, one obtains the corresponding parabola analogously.
\end{remark}

\paragraph{Difference-Angle Norm.}
\begin{definition}[Difference-Angle Norm]\label{def:diff-angle-norm}
Fix a projective reference structure $(\ell,d)$ and take a covector $\varphi\in(\mathbb{R}^2)^*$ such that
\[
\varphi|_\ell \neq 0,\qquad \varphi(d)=0.
\]
For two points $A,B$, define
\[
|AB|_{\mathcal P} := |\varphi(B-A)|
\]
and call this the difference-angle norm.
\end{definition}


\section{DA Triangles}

In this section, we define another fundamental figure in the DA geometry,
namely a triangle inscribed in a parabola,
and investigate its basic properties.

\medskip
Unlike in Euclidean geometry, an arbitrary triple of points
does not necessarily determine a triangle in this geometry.
This is due to the existence of singular lines,
that is, lines parallel to the projective direction $d$.
Any configuration including a singular line as one of its sides
degenerates along the projective direction
because of the property of the DA norm.
Therefore, such configurations are excluded at the first stage.

\medskip
Two features of the DA geometry are worth emphasizing:
(1) the triangle inequality for the DA norm always holds
with equality, and
(2) an equilateral triangle cannot be defined.
We begin by introducing the definition of a DA triangle
and the notion of interior angles,
which together will characterize this geometry.

\begin{figure}[htbp]
  \centering
  \begin{minipage}{0.32\textwidth}
    \centering
    \includegraphics[width=\linewidth]{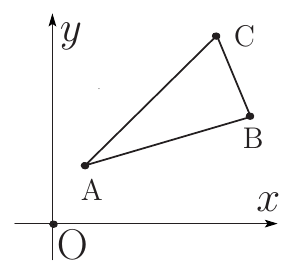}
    \subcaption{A DA triangle.}
    \label{fig:diff-ang-triangle}
  \end{minipage}
  \begin{minipage}{0.32\textwidth}
    \centering
    \includegraphics[width=\linewidth]{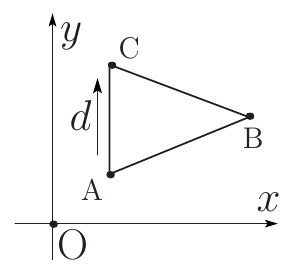}
    \subcaption{A degenerate configuration (containing a singular line).}
    \label{fig:not-diff-ang-triangle}
  \end{minipage}
  \begin{minipage}{0.32\textwidth}
    \centering
    \includegraphics[width=\linewidth]{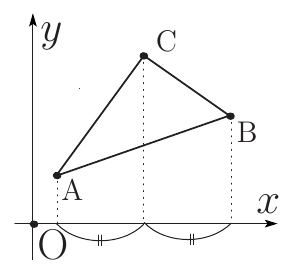}
    \subcaption{A DA isosceles triangle.}
    \label{fig:diff-ang-isotriangle}
  \end{minipage}
  \caption{DA triangles and excluded cases.}
\end{figure}

\begin{definition}[DA triangle]\label{def:diff-angle-triangle}
Let $A,B,C$ be three distinct noncollinear points.
If none of the lines passing through any two of them is singular,
the figure determined by these points is called a
\textbf{DA triangle},
denoted by $\triangle_{\mathcal P}ABC$.
In particular, if two of the three sides have equal
DA norms, the triangle is called
a \textbf{DA isosceles triangle}.
\end{definition}

\begin{figure}[htbp]
  \centering
  \begin{minipage}{0.32\textwidth}
    \centering
    \includegraphics[width=\linewidth]{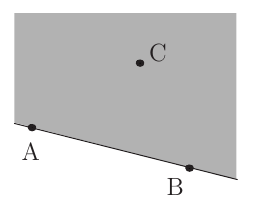}
    \subcaption{Half--plane determined by $AB$ containing $C$.}
    \label{fig:area-C-side-AB}
  \end{minipage}
  \begin{minipage}{0.32\textwidth}
    \centering
    \includegraphics[width=\linewidth]{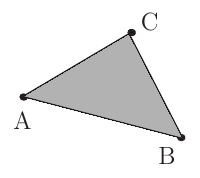}
    \subcaption{Interior region $\mathrm{Int}(\triangle ABC)$.}
    \label{fig:d-a-axi-fig-int-and-ext}
  \end{minipage}
  \begin{minipage}{0.32\textwidth}
    \centering
    \includegraphics[width=\linewidth]{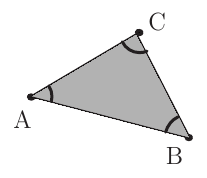}
    \subcaption{Interior angles of $\triangle ABC$.}
    \label{fig:interior-angle}
  \end{minipage}
  \caption{Interior of a DA triangle and its angles.}
\end{figure}

\begin{definition}[Interior and exterior regions]\label{int-ext}
For each line $AB$, $BC$, and $CA$,
take the open half--plane on the same side
as the opposite vertex, respectively denoted by
\[
H_C(AB),\quad H_A(BC),\quad H_B(CA).
\]
Then the \textbf{interior} of $\triangle_{\mathcal P}ABC$ is defined as
\[
\mathrm{Int}(\triangle_{\mathcal P}ABC)
 := H_C(AB)\cap H_A(BC)\cap H_B(CA),
\]
and the \textbf{exterior} as
\[
\mathrm{Ext}(\triangle_{\mathcal P}ABC)
 := \mathbb R^2\setminus \overline{\mathrm{Int}(\triangle_{\mathcal P}ABC)}.
\]
These notions are defined analogously
under any projective gauge structure.
\end{definition}

\begin{definition}[Interior angle]\label{interior-angle}
For a DA triangle $\triangle_{\mathcal P}ABC$,
the interior angle at vertex $A$,
denoted by $\measuredangle_{\mathcal P}A$,
is the DA between the rays $AB$ and $AC$
that lies within $\mathrm{Int}(\triangle_{\mathcal P}ABC)$.
The angles $\measuredangle_{\mathcal P}B$ and
$\measuredangle_{\mathcal P}C$ are defined similarly.
\end{definition}

Henceforth, the angles
$\measuredangle_{\mathcal P}BAC$, $\measuredangle_{\mathcal P}CBA$,
and $\measuredangle_{\mathcal P}ACB$
will be abbreviated as
$\measuredangle_{\mathcal P}A$,
$\measuredangle_{\mathcal P}B$,
and $\measuredangle_{\mathcal P}C$, respectively.

\begin{figure}[htbp]
  \centering
  \begin{minipage}{0.32\textwidth}
    \centering
    \includegraphics[width=\linewidth]{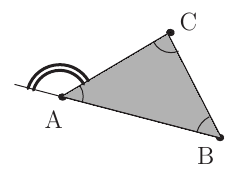}
    \subcaption{Exterior angle.}
    \label{fig:exterior-angle}
  \end{minipage}
  \hspace{0.5cm}
  \begin{minipage}{0.32\textwidth}
    \centering
    \includegraphics[width=\linewidth]{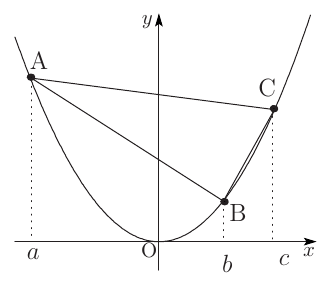}
    \subcaption{A DA triangle $ABC$ and its circumparabola.}
    \label{fig:unique-circum-parabola}
  \end{minipage}
  \caption{Two fundamental notions:
  (a) exterior angle, (b) circumparabola.}
\end{figure}

\begin{definition}[DA exterior angle]\label{def:exterior-angle}
The \textbf{exterior angle} at vertex $A$
is defined as the DA
formed in $\mathrm{Ext}(\triangle_{\mathcal P}ABC)$
between one of the rays obtained by extending either $AB$ or $AC$
beyond $A$, and the other ray.
The exterior angles at vertices $B$ and $C$ are defined in the same way.
\end{definition}

\begin{theorem}[Uniqueness of the Circumparabola]\label{thm:unique-circum-parabola}
For a DA triangle $\triangle_{\mathcal{P}}ABC$, there exists a unique parabola passing through the three vertices $A,B,C$ whose axis is parallel to the projective direction $d$.
\end{theorem}

\begin{proof}
This lemma underpins the reduction that moves any DA triangle onto the standard parabola $y=x^2$, thereby simplifying all subsequent arguments. In essence, it suffices to confirm the uniqueness of the associated coordinate normalization.

Let $A(a,d)$, $B(b,e)$, $C(c,f)$ be three points. A parabola with axis parallel to the $y$–axis passing through $A,B,C$ is given, by Lagrange interpolation, by
\[
y = d \frac{(x - b)(x - c)}{(a - b)(a - c)}
+ e \frac{(x - a)(x - c)}{(b - a)(b - c)}
+ f \frac{(x - a)(x - b)}{(c - a)(c - b)} 
\]
\begin{align*}
&\iff
y = \frac{
d(c - b)(x - b)(x - c)
- e(c - a)(x - a)(x - c)
+ f(b - a)(x - a)(x - b)
}{
(b - a)(c - b)(c - a)
} \\[2pt]
&\iff
y = \frac{
\{d(c - b) + e(a - c) + f(b - a)\}x^2
- \{d(c^2 - b^2) + e(a^2 - c^2) + f(b^2 - a^2)\}x
}{
(b - a)(c - b)(c - a)
} \\[2pt]
&\quad\ 
+ \frac{
bcd(c - b) + ace(a - c) + abf(b - a)
}{
(b - a)(c - b)(c - a)
}.
\end{align*}
Hence the parabola is uniquely determined.
\end{proof}

\begin{figure}[htbp]
  \centering
  \begin{minipage}{0.32\textwidth}
    \centering
    \includegraphics[width=\linewidth]{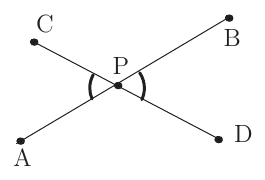}
    \subcaption{Vertical angles formed by two intersecting lines.}
    \label{fig:vertical-angles}
  \end{minipage}
  \begin{minipage}{0.32\textwidth}
    \centering
    \includegraphics[width=\linewidth]{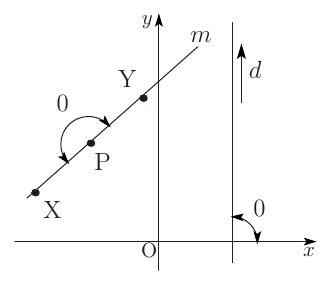}
    \subcaption{Convention: a straight angle is $0$.}
    \label{fig:straight-angle-zero}
  \end{minipage}
  \begin{minipage}{0.28\textwidth}
    \centering
    \includegraphics[width=\linewidth]{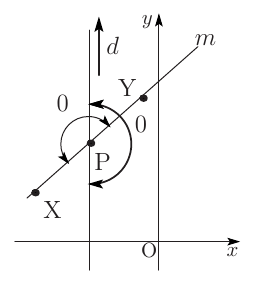}
    \subcaption{Proof of the convention (straight angle $=0$).}
    \label{fig:proof-straight-angle-zero}
  \end{minipage}
  \caption{Vertical angles and the straight-angle convention.}
\end{figure}

\begin{proposition}[Equality of DA Vertical Angles]\label{prop:vertical-angle}
If lines $AB$ and $CD$ intersect at a point $P$, then
\[
\measuredangle_{\mathcal P} APC = \measuredangle_{\mathcal P} DPB.
\]
\end{proposition}

\begin{proof}
At first glance this seems immediate from the sign-reversal rule for DA, but under a boundary policy it is not entirely trivial. A detailed discussion is deferred to Appendix~\Cref{app:vertical-angle-boundary}; here we compute directly from the definition.

By the slope definition,
\[
\measuredangle_{\mathcal P} UVW := \Slp{VW}-\Slp{VU},
\]
where $\Slp{\cdot}$ denotes the slope of a line (the same value is used for opposite rays on that line). At the intersection $P$, the rays $PA,PB$ lie on the same line $AB$, so
\(\Slp{PA}=\Slp{PB}=\Slp{AB}\); similarly $PC,PD$ lie on $CD$ so
\(\Slp{PC}=\Slp{PD}=\Slp{CD}\). Therefore,
\begin{align*}
\measuredangle_{\mathcal P} APC
  &= \Slp{PC}-\Slp{PA} \\
  &= \Slp{CD}-\Slp{AB} \\
  &= \Slp{PD}-\Slp{PB} \\
  &= \measuredangle_{\mathcal P} DPB.
\end{align*}
If $AB$ or $CD$ coincides with a singular line, then by Axiom~A3 (straight angle $=0$) together with the absorptive boundary policy, both sides are absorbed to $0$, and the equality still holds.
\end{proof}

\begin{remark}[Role of the Boundary Policy]
Although \Cref{prop:vertical-angle} is presented as a straightforward slope computation, the definition of the DA can become ambiguous when the rotation at $P$ crosses the singular direction. Because we adopt the absorptive policy, the boundary tags are identified, and this ambiguity disappears. Thus the equality of vertical angles is well-defined even in the presence of boundary effects.
\end{remark}

\begin{lemma}[Straight Angle Equals $0$]\label{lem:straight-angle-zero}
Let $P$ be a point and $\ell$ a line. For two opposite rays $PX,PY$ on $\ell$, the DA they form is $0$.
\end{lemma}

\begin{proof}
Let $D$ be the singular line through $P$. By \Cref{prop:vertical-angle}, the DA between the $d$–direction of $D$ and $PX$ equals the DA between the $-d$–direction and $PY$. Under the absorptive boundary policy, the sum of the DAs of $d$ and $-d$ is absorbed to $0$. Hence the DA between $PX$ and $PY$ is $0$.
\end{proof}

\begin{remark}
In Euclidean geometry, one shows that straight angles are constant (conventionally $\pi$) via the uniqueness of supplements, and then deduces the vertical-angle property. In our framework, the boundary policy yields the vertical-angle property directly, from which the constancy of supplements follows. Through this lemma, the convention “straight angle $=0$” arises naturally; the logical order is thus reversed compared to the Euclidean case.
\end{remark}

\begin{remark}[Dependence on the Boundary Policy]
The conclusion of \Cref{lem:straight-angle-zero} depends on choosing the absorptive policy. Under a lift-type policy, a straight angle is not $0$ but a nontrivial element of the period lattice; under a divergent policy, it extends to $\pm\infty$. Thus “straight angle $=0$” is not intrinsic to DA geometry itself, but a consistent outcome of the absorptive choice within the calibration framework.
\end{remark}

\begin{lemma}[Parallel Lines and Corresponding Angles]
Let $\ell, m$ be parallel lines and let a transversal $t$ intersect them. Then the corresponding angles formed by $t$ with $\ell$ and $m$ are equal in the DA sense.
\end{lemma}

\begin{proof}
Let $A,B$ be the intersections of $t$ with $\ell$ and $m$, respectively. Take $C$ on $\ell$ and $D$ on $m$ on the side opposite to $t$. By definition,
\[
\measuredangle_{\mathcal P}CAB
= \mathrm{slope}(AB)-\mathrm{slope}(AC),\qquad
\measuredangle_{\mathcal P}ABD
= \mathrm{slope}(BD)-\mathrm{slope}(AB).
\]
Since $\ell\parallel m$, we have $\mathrm{slope}(AC)=\mathrm{slope}(BD)$, hence
\[
\measuredangle_{\mathcal P}CAB=\measuredangle_{\mathcal P}ABD.
\]
\end{proof}

\begin{lemma}[Parallel Lines and Alternate Interior Angles]
Let $\ell, m$ be parallel lines and let a transversal $t$ intersect them. Then the alternate interior angles formed by $t$ with $\ell$ and $m$ are equal in the DA sense.
\end{lemma}

\begin{proof}
Let $A,B$ be the intersections of $t$ with $\ell$ and $m$, respectively, and take $C$ on $\ell$ and $D$ on $m$. We must show
\[
\measuredangle_{\mathcal P}CAB
=\measuredangle_{\mathcal P}DBA.
\]
Indeed,
\[
\measuredangle_{\mathcal P}CAB
=\mathrm{slope}(AB)-\mathrm{slope}(AC),\qquad
\measuredangle_{\mathcal P}DBA
=\mathrm{slope}(AB)-\mathrm{slope}(DB).
\]
Since $\ell\parallel m$, we have $\mathrm{slope}(AC)=\mathrm{slope}(DB)$, proving the claim.
\end{proof}

\begin{figure}[htbp]
  \centering
  \begin{minipage}{0.32\textwidth}
    \centering
    \includegraphics[width=\linewidth]{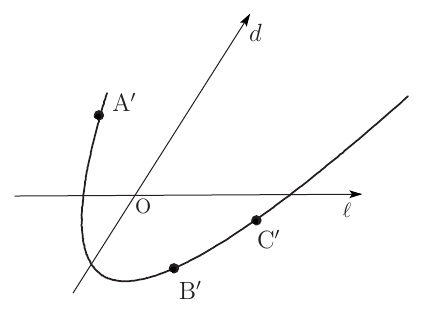}
    \subcaption{Before affine transformation.}
    \label{fig:normalize-affine-before}
  \end{minipage}
  \begin{minipage}{0.28\textwidth}
    \centering
    \includegraphics[width=\linewidth]{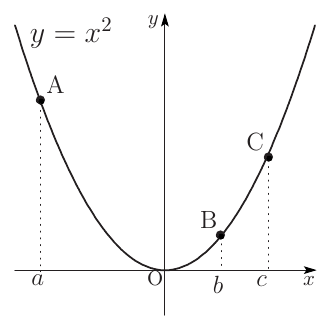}
    \subcaption{After affine transformation.}
    \label{fig:normalize-affine-after}
  \end{minipage}
  \begin{minipage}{0.32\textwidth}
    \centering
    \includegraphics[width=\linewidth]{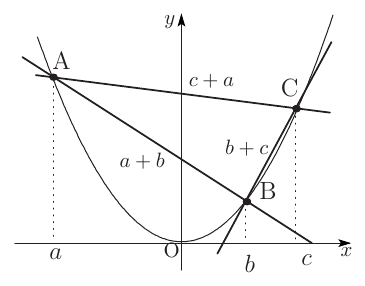}
    \subcaption{Normalized form on $y=x^2$.}
    \label{fig:normalize-slope}
  \end{minipage}
  \caption{Illustration for \Cref{lem:normalize} (Normalization Lemma).}
  \label{fig:normalize-triangle}
\end{figure}

\begin{lemma}[Normalization Lemma]\label{lem:normalize}
For any DA triangle $\triangle_{\mathcal P}ABC$, the circumparabola of \cref{thm:unique-circum-parabola} exists uniquely. Moreover, by an affine transformation preserving the reference structure $(\ell,d)$ together with a vertical scaling, it can be sent to the standard parabola $y=x^2$. Consequently, one may assume
\[
A=(a,a^2),\quad B=(b,b^2),\quad C=(c,c^2)\qquad (a<b<c).
\]
\end{lemma}

\begin{proof}
By \cref{thm:unique-circum-parabola}, there is a unique parabola $\Gamma$ through $A,B,C$ whose axis is parallel to $d$. First apply an affine transformation sending $\ell$ to the $x$–axis and $d$ to the $y$–axis. Then $\Gamma$ has the form
\[
y = \alpha x^2 + \beta x + \gamma.
\]
A horizontal translation removes the linear term $\beta x$, and a vertical translation removes the constant term $\gamma$. Finally, a scaling in the $y$–direction normalizes $\alpha$ to $1$. Hence $\Gamma$ is sent to $y=x^2$, and we may take
\[
A=(a,a^2),\quad B=(b,b^2),\quad C=(c,c^2),\quad (a<b<c).
\]
\end{proof}

In many arguments below, explicit algebraic calculations are clearer. We therefore record basic facts about the circumparabola. Most of the time we work on the normalized plane with the standard parabola $y=x^2$; when dependence on the quadratic coefficient matters, we also consider $y=\kappa x^2$.

\begin{remark}[Concrete Form on the Standard Parabola]
By the normalization lemma, a DA triangle can be written as
$A=(a,a^2)$, $B=(b,b^2)$, $C=(c,c^2)$ with $a<b<c$.
Then the slopes of the sides are
\[
\Slp{AB}=a+b,\qquad \Slp{BC}=b+c,\qquad \Slp{CA}=c+a.
\]
Subsequent computations in propositions and theorems will use this normalized form.
\end{remark}

\begin{theorem}[Sum of Interior Angles is Zero]\label{thm:sum-of-angles-zero}
Let the interior angles of $\triangle_{\mathcal{P}} ABC$ be
\[
\theta_A := \measuredangle_{\mathcal{P}} BAC,\quad
\theta_B := \measuredangle_{\mathcal{P}} CBA,\quad
\theta_C := \measuredangle_{\mathcal{P}} ACB.
\]
Then
\[
\theta_A+\theta_B+\theta_C = 0.
\]
\end{theorem}

\begin{proof}
In Euclidean geometry the sum of the interior angles is $\pi$, whereas in DA geometry,
as a consequence of the boundary policy together with the vertical–angle property,
a straight angle is $0$. Accordingly, the interior–angle sum becomes $0$ in DA geometry.

Using \Cref{lem:normalize}, write
$A=(a,a^2),\ B=(b,b^2),\ C=(c,c^2)$ with $a<b<c$.
Since a DA is the difference of slopes, we have
\begin{align*}
  \theta_A &= \Slp{AC} - \Slp{AB} = (a+c)-(a+b) = c-b,\\
  \theta_C &= \Slp{CB} - \Slp{CA} = (c+b)-(c+a) = b-a.
\end{align*}
At $B$, the interior angle is taken with the exterior orientation, so by
\Cref{lem:straight-angle-zero} the sign flips:
\[
\theta_B = -(\Slp{BC} - \Slp{BA}) = -(c+b)+(a+b)=a-c.
\]
Hence
\[
\theta_A+\theta_B+\theta_C=(c-b)+(a-c)+(b-a)=0.
\]
\end{proof}

\begin{figure}[htbp]
  \centering
  \begin{minipage}{0.32\textwidth}
    \centering
    \includegraphics[width=\linewidth]{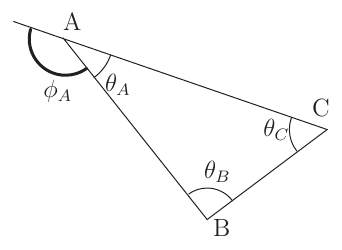}
    \subcaption{Exterior angle theorem.}
    \label{fig:exterior-theorem}
  \end{minipage}
  \begin{minipage}{0.32\textwidth}
    \centering
    \includegraphics[width=\linewidth]{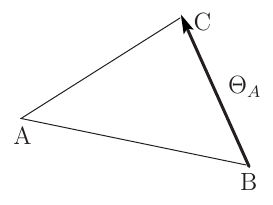}
    \subcaption{Angle vector at a vertex.}
    \label{fig:angle-vector}
  \end{minipage}
  \begin{minipage}{0.32\textwidth}
    \centering
    \includegraphics[width=\linewidth]{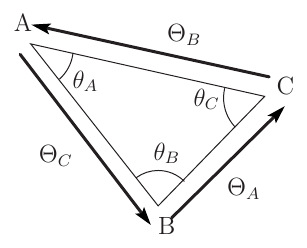}
    \subcaption{}
    \label{fig:triangle-identity}
  \end{minipage}
 \caption{Illustrations for \Cref{prop:exterior-theorem}, \Cref{def:angle-vector}, and \Cref{thm:triangle-identity}.}
\end{figure}

\begin{proposition}[Exterior Angle Theorem]\label{prop:exterior-theorem}
In $\triangle_{\mathcal P}ABC$, the exterior angle at any vertex equals the sum of the other two interior angles.
\end{proposition}

\begin{proof}
For the exterior angle $\phi_A$ at $A$, by definition $\phi_A=-\theta_A$.
From \Cref{thm:sum-of-angles-zero}, $\theta_A+\theta_B+\theta_C=0$, hence
$\phi_A=\theta_B+\theta_C$. The other vertices are analogous.
\end{proof}

\begin{definition}[Vector Notation for a DA Triangle]\label{def:angle-vector}
Let $\triangle_{\mathcal P}ABC$ be oriented counterclockwise.
Assign to each vertex the vector
\[
\Theta_A := \vec{BC}, \qquad
\Theta_B := \vec{CA}, \qquad
\Theta_C := \vec{AB}.
\]
\end{definition}

\begin{theorem}[Fundamental Identity for DA Triangles]\label{thm:triangle-identity}
With the above notation, one always has
\[
\Theta_A + \Theta_B + \Theta_C = 0.
\]
\end{theorem}

\begin{remark}[Role as a Basic Equation]
In Euclidean geometry,
\[
\vec{AB} + \vec{BC} + \vec{CA} = \vec{0},
\]
from which the cosine and sine laws follow by applying the inner and outer products.
In DA geometry, the identity
\[
\Theta_A + \Theta_B + \Theta_C = 0
\]
plays an analogous structural role: beyond yielding “sum of interior angles $=0$,”
it serves as the starting point for DA analogues of the cosine and sine laws.
\end{remark}

\begin{figure}[htbp]
  \centering
  \begin{minipage}{0.32\textwidth}
    \centering
    \includegraphics[width=\linewidth]{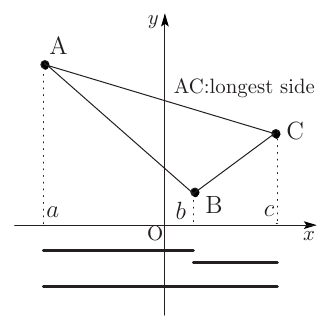}
    \subcaption{Triangle equation.}
    \label{fig:triangle-equation}
  \end{minipage}
  \begin{minipage}{0.32\textwidth}
    \centering
    \includegraphics[width=\linewidth]{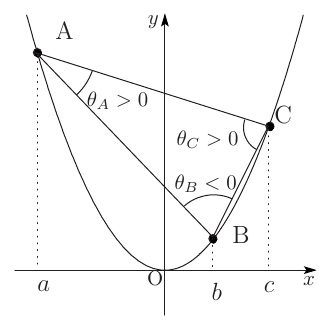}
    \subcaption{Concave down case ($a<b<c$): $\theta_B<0$.}
    \label{fig:concave-down}
  \end{minipage}
  \begin{minipage}{0.32\textwidth}
    \centering
    \includegraphics[width=\linewidth]{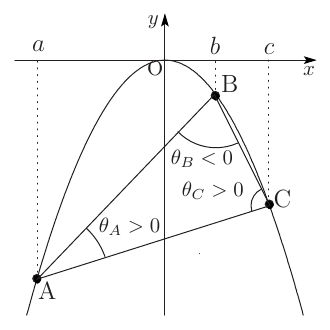}
    \subcaption{Concave up case ($a<b<c$): still $\theta_B<0$.}
    \label{fig:concave-up}
  \end{minipage}
  \caption{DA-triangle illustrations:
  (a) visualizing the triangle equation;
  (b) if the circumparabola is concave down, then $\theta_B<0$;
  (c) if it is concave up, the same conclusion holds.}
  \label{fig:triangle-equation-and-sign}
\end{figure}

\begin{maintheorem}[Triangle Equation]\label{mthm:triangle-equation}
For any DA triangle $\triangle_{\mathcal P}ABC$, the largest DA norm of the sides equals the sum of the other two.
\end{maintheorem}

\begin{proof}
Let the $x$–coordinates of the vertices satisfy $a<b<c$.
Then the longest side is $|AC|_{\mathcal P}$, and by definition of the DA norm,
\[
  |AC|_{\mathcal P}=c-a,\quad
  |AB|_{\mathcal P}=b-a,\quad
  |BC|_{\mathcal P}=c-b.
\]
Hence
\[
  |AC|_{\mathcal P}=c-a=(c-b)+(b-a)
  =|AB|_{\mathcal P}+|BC|_{\mathcal P}.
\]
\end{proof}

\begin{corollary}[Nonexistence of Equilateral DA Triangles]\label{cor:no-equilateral}
There is no DA triangle whose three side DA norms are all equal.
\end{corollary}

\begin{proof}
Such an equality would contradict the triangle equation \Cref{mthm:triangle-equation}.
\end{proof}

\begin{theorem}[Uniqueness of the Signs of Interior Angles]\label{thm:angle-sign-uniqueness}
In a DA triangle $\triangle_{\mathcal P}ABC$, exactly one interior angle is negative.
\end{theorem}

\begin{proof}
By \Cref{thm:sum-of-angles-zero}, at least one angle is negative.
Using \Cref{mthm:triangle-equation} together with
\Cref{prop:positive-diff-angle-equivalence}, two or more cannot be negative.
Hence exactly one is negative.
\end{proof}

\begin{corollary}
If $\triangle_{\mathcal{P}}ABC$ is isosceles in the DA sense, then the vertex angle is negative.
\end{corollary}

\begin{proof}
Project $\triangle_{\mathcal{P}}ABC$ along the projective direction $d$ onto the reference line $\ell$.
The image is the segment whose endpoints are the projections of the two base vertices.
Exactly one vertex projects onto the interior of the opposite side's projection. Let $D'$ be that interior projected point on the base; take $D$ on the original base with $d$–projection $D'$.
By definition, the corresponding vertex angle is a negative DA.
\end{proof}


\section{DA Bisectors, the Incenter, and Surrounding Results}

In this chapter we introduce a projective slope chart so as to derive the
angle–bisector theorem in the DA setting.

\begin{figure}[htbp]
  \centering
  \begin{minipage}{0.32\textwidth}
    \centering
    \includegraphics[width=\linewidth]{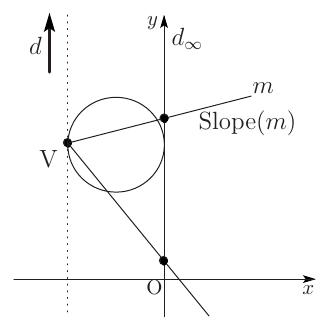}
    \subcaption{Cyclic order model on the slope circle.}
    \label{fig:proj-slope-with-dinf}
  \end{minipage}
  \begin{minipage}{0.32\textwidth}
    \centering
    \includegraphics[width=\linewidth]{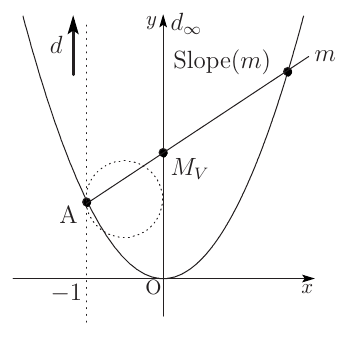}
    \subcaption{Example at $A=(-1,1)$ (projective slope chart).}
    \label{fig:proj-slope-with-dinf-example}
  \end{minipage}
  \begin{minipage}{0.32\textwidth}
    \centering
    \includegraphics[width=\linewidth]{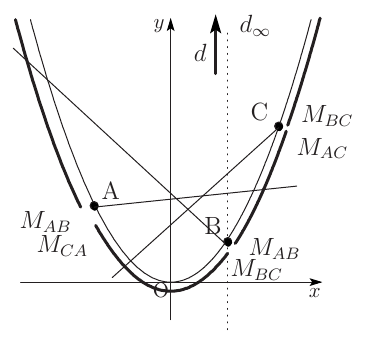}
    \subcaption{Inner arcs at $A$, $C$, and $B$ (the one at $B$ includes $d_\infty$).}
    \label{fig:inner-arcs}
  \end{minipage}
 \caption{(a) Cyclic order on the compactified slope line.  
  (b) Example of the projective slope chart.  
  (c) Inner arcs at the vertices; at $B$ the inner arc is defined as the one containing $d_\infty$.}
  \label{fig:slope-circle}
\end{figure}

\begin{definition}[Projective slope chart and singular slope]\label{def:proj-slope-with-dinf}
Fix a projective reference structure $(\ell,d)$. Consider the one–point compactification of the slope set
\[
\overline{\mathbb{R}}_{d}:=\mathbb{R}\cup\{\dinf\},
\]
where $\pm\infty$ are identified and endowed with the circular (oriented) order.
For each point $V$, let $\Dir(V)$ denote the set of oriented rays from $V$.
Define the map
\[
\mslope_V:\ \Dir(V)\longrightarrow \overline{\mathbb{R}}_{d}
\]
by
\[
\mslope_V(\ell)=
\begin{cases}
\slope(\ell) & (\ell \not\parallel d),\\[2pt]
\dinf & (\ell \parallel d).
\end{cases}
\]
Here $\Cut\in\Dir(V)$ denotes the singular ray parallel to $d$, and we set $\mslope_V(\Cut)=\dinf$.
On each connected component of $\Dir(V)\setminus\{\Cut\}$ the chart $\mslope_V$ is monotone and continuous;
the two ends of $\mathbb{R}$ are glued at the ideal point corresponding to $\dinf$ in the circular order.
\end{definition}

\begin{remark}
The map $\mslope_V$ is well defined and compatible with the circular order on $\overline{\mathbb{R}}_{d}$.
Note also that the phrase ``difference angle $=0$'' refers to the absorptive boundary rule for the angle measure and
is unrelated to the numeric slope value $0$. The symbol $\dinf$ denotes the singular slope corresponding to $d$
and should not be confused with the ordinary notion of $\infty$.
\end{remark}

\begin{definition}[Inner arcs in the compactified slope chart]\label{def:inner-arc-slope}
For each vertex $V\in\{A,B,C\}$, write by $M_V\in\overline{\mathbb R}_d$ the slope of a line through $V$,
and let $m_{AB},m_{BC},m_{CA}$ denote the slopes of the three sides.
Define the inner arc (endpoints excluded) by
\[
\arcintV{A}=\{\,M_A\mid m_{AB}\prec M_A\prec m_{AC}\,\},\qquad
\arcintV{C}=\{\,M_C\mid m_{CA}\prec M_C\prec m_{CB}\,\},
\]
\[
\arcintV{B}=\{\,M_B\mid M_B\prec m_{BA}\,\}\ \cup\ \{\,M_B\mid m_{BC}\prec M_B\,\},
\]
where $\prec$ is the circular order on $\overline{\mathbb R}_d$.
On the circle $\arcintV{B}$ is a single open arc and always contains $d_\infty$.
\end{definition}

\begin{remark}[Behavior at the vertex $B$]
Under the DA conventions (straight angle $=0$, angle to a singular line $=0$),
the inner arc $\arcintV{B}$ contains $d_\infty$, hence its bisecting direction collapses to $d_\infty$.
In the parabola–normalized coordinates $A(a,a^2)$, $B(b,b^2)$, $C(c,c^2)$, the point $d_\infty$ corresponds to the vertical line $x=b$.
\end{remark}

\begin{definition}[Bisector of a negative DA]\label{def:minus-angle-bisector}
If $\triangle_{\mathcal{P}}ABC$ has $\theta_B<0$, then the DA bisector of $\theta_B$ is the singular line through $B$.
\end{definition}

\begin{figure}[htbp]
  \centering
  \begin{minipage}{0.32\textwidth}
    \centering
    \includegraphics[width=\linewidth]{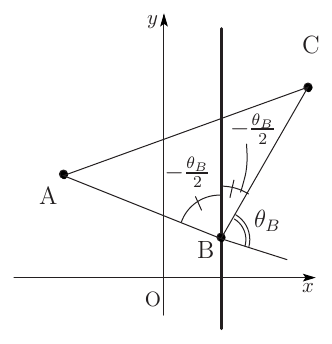}
    \subcaption{Bisector of a negative angle (singular line through $B$).}
    \label{fig:d-a-axi-fig-minus-angle-bisector}
  \end{minipage}
  \begin{minipage}{0.32\textwidth}
    \centering
    \includegraphics[width=\linewidth]{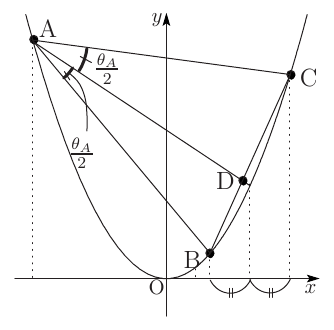}
    \subcaption{DA angle–bisector theorem.}
    \label{fig:proj-slope-with-dinf-example}
  \end{minipage}
  \begin{minipage}{0.32\textwidth}
    \centering
    \includegraphics[width=\linewidth]{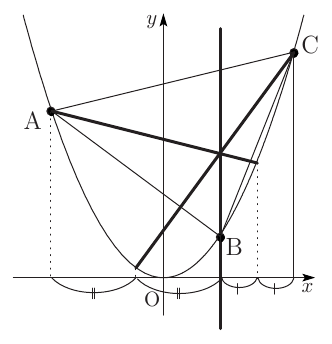}
    \subcaption{Incenter of a DA triangle.}
    \label{fig:inner-point}
  \end{minipage}
\caption{(a) Bisector of a negative DA. 
(b) DA angle–bisector theorem. 
(c) Incenter in the DA setting.}
\end{figure}

\begin{maintheorem}[DA Angle–Bisector Theorem]\label{mthm:diff-angle-bisector}
Let $\triangle_{\mathcal{P}}ABC$ have the positive interior DA at $A$, denoted $\theta_A$.
Let $\ell$ be the DA bisector of $\theta_A$, and let $D=\ell\cap BC$.
Then
\[
|AB|_{\mathcal{P}}\ :\ |AC|_{\mathcal{P}} \;=\; |BD|_{\mathcal{P}}\ :\ |DC|_{\mathcal{P}}.
\]
\end{maintheorem}

\begin{proof}
Ratios of DA norms are invariant under scaling along the projective direction,
so it suffices to treat the case where $\triangle_{\mathcal{P}}ABC$ is inscribed in $y=x^2$.
Let $A(a,a^2)$, $B(b,b^2)$, $C(c,c^2)$ with $a<b<c$. Then $\measuredangle_{\mathcal{P}}A>0$.
On $y=x^2$, take
\[
A'\Bigl(\tfrac{b+c}{2},\bigl(\tfrac{b+c}{2}\bigr)^2\Bigr),
\]
so that
\[
|BA'|_{\mathcal{P}}=|A'C|_{\mathcal{P}}=\tfrac{c-b}{2},
\]
and hence $AA'$ is the DA bisector of $\measuredangle_{\mathcal{P}}A$.
The lines
\[
AA':\ y=\Bigl(a+\tfrac{b+c}{2}\Bigr)x-\tfrac{a(b+c)}{2},\qquad
BC:\ y=(b+c)x-bc
\]
meet at $D$, whose $x$–coordinate is
\[
\Bigl(a+\tfrac{b+c}{2}\Bigr)x_D-\tfrac{a(b+c)}{2}=(b+c)x_D-bc
\quad\iff\quad
x_D=\frac{2bc-ab-ac}{\,b+c-2a\,}.
\]
Therefore
\begin{align*}
|BD|_{\mathcal{P}}\ :\ |DC|_{\mathcal{P}}
&=\Bigl(\tfrac{2bc-ab-ac}{\,b+c-2a\,}-b\Bigr)\ :\ \Bigl(c-\tfrac{2bc-ab-ac}{\,b+c-2a\,}\Bigr)\\
&=(b-a)(c-b)\ :\ (c-a)(c-b)\\
&=(b-a)\ :\ (c-a)
= |AB|_{\mathcal{P}}\ :\ |AC|_{\mathcal{P}}.
\end{align*}
For a negative angle, \cref{def:minus-angle-bisector} shows that the bisector is the singular line,
and one has $|BD|_{\mathcal{P}} = |BA|_{\mathcal{P}}$ and $|DC|_{\mathcal{P}} = |AC|_{\mathcal{P}}$,
so the same ratio holds. The zero–angle case is degenerate and trivial.
\end{proof}

Consequently, as in Euclidean geometry, the incenter is well defined in DA geometry.

\begin{corollary}\label{cor:exist_inner-point}
The three interior DA angle bisectors of $\triangle_{\mathcal{P}}ABC$ concur at a single finite point.
This point is called the incenter of $\triangle_{\mathcal{P}}ABC$.\footnote{%
If one defines ``angle bisectors'' symmetrically, the intersection pattern may at first look scattered.
In fact, the unique finite intersection is the incenter; the other intersections are the excenters
(one of which lies at an ideal point).}
\end{corollary}

\begin{figure}[htbp]
  \centering
  \begin{minipage}{0.32\textwidth}
    \centering
    \includegraphics[width=\linewidth]{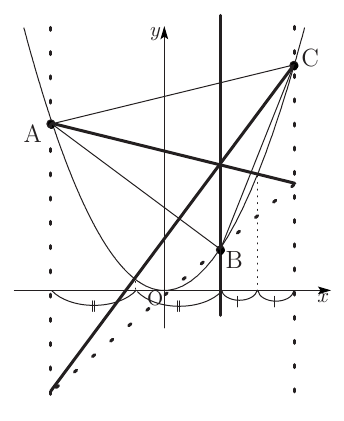}
    \subcaption{}
    \label{fig:excenter}
  \end{minipage}
  \begin{minipage}{0.32\textwidth}
    \centering
    \includegraphics[width=\linewidth]{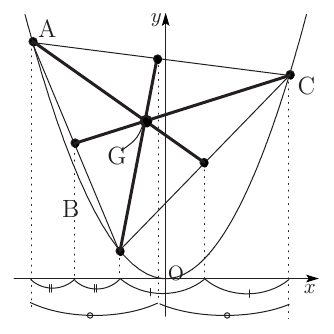}
    \subcaption{}
    \label{fig:centroid}
  \end{minipage}
\begin{minipage}{0.32\textwidth}
    \centering
    \includegraphics[width=\linewidth]{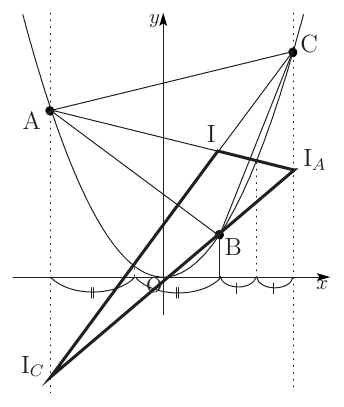}
    \subcaption{}
    \label{fig:bisector-triangle}
  \end{minipage}  
   \caption{(a) Excenter. (b) Centroid. (c) Bisector triangle.}

  \label{fig:excent-centroid-bisec-tri}
\end{figure}

\begin{proof}
In the proof of \Cref{mthm:diff-angle-bisector}, take the point
\(C'\!\left(\tfrac{c+a}{2},\bigl(\tfrac{c+a}{2}\bigr)^2\right)\).
Then the intersection of \(AA'\) and \(CC'\) has \(x\)-coordinate \(b\),
hence lies on the singular line through \(B\).
Therefore the three interior DA bisectors of \(\triangle_{\mathcal P}ABC\)
concur at a single finite point.
\end{proof}

\begin{remark}
The statements in this subsection are most naturally formulated under the
absorptive boundary policy. With the other policies (lift/divergent),
the behavior of straight angles and bisection changes, and the existence
and/or position of the incenter must be adjusted accordingly.
\end{remark}

\begin{corollary}[Existence of excenters]\label{cor:excenter}
For a DA triangle \(\triangle_{\mathcal P}ABC\), three excenters are defined
as the intersections of the interior bisector at one vertex with the
exterior bisectors at the other two vertices.
\end{corollary}

\begin{proof}
In the proof of \Cref{mthm:diff-angle-bisector}, take
\(B'\bigl(\tfrac{c+a}{2},(\tfrac{c+a}{2})^2\bigr)\).
Then \(BB'\) is the bisector of \(-\theta_B\).
The intersection \(I_A=AA'\cap BB'\) has \(x\)-coordinate \(c\),
so it lies on the finite line \(x=c\), which is the bisector of \(-\theta_C\).
The point \(I_C\) is constructed analogously.
On the other hand, the bisectors of \(-\theta_A\), \(-\theta_C\), and \(+\theta_B\)
are all singular lines, so their intersection is an ideal point at infinity.
\end{proof}

\begin{corollary}[Existence of the centroid]\label{cor:centroid}
In a DA triangle \(\triangle_{\mathcal P}ABC\), let \(D,E,F\) be the midpoints
of the sides \(BC,CA,AB\), respectively. Then the three medians
\(AD,BE,CF\) concur at a single point \(G\), called the DA centroid.
\end{corollary}

\begin{proof}
By the definition of the DA norm, the points \(D,E,F\) coincide with the
usual midpoints in Euclidean geometry, hence the existence of \(G\) is
well-defined.
\end{proof}

Assume \(\measuredangle_{\mathcal P}B<0\).
Let \(I\) be the incenter and \(I_A,I_C\) the two finite excenters of
\(\triangle_{\mathcal P}ABC\).
The triangle formed by these three points is called the bisector triangle,
and its centroid is denoted by \(G_I\).

\begin{definition}[Bisector triangle]\label{def:bisector-triangle}
For a DA triangle \(\triangle_{\mathcal P}ABC\) with \(\measuredangle_{\mathcal P}B<0\),
let \(I\) be its incenter and \(I_A,I_C\) the two finite excenters.
The triangle \(\triangle_{\mathcal P}II_AI_C\) is called the bisector triangle,
and its centroid is denoted by \(G_I\).
\end{definition}

\begin{remark}
Under the assumption \(\measuredangle_{\mathcal P}B<0\),
the points \(I, I_A, I_C\) are always non-collinear, so the bisector triangle
is indeed a (nondegenerate) triangle.
\end{remark}

\begin{figure}[htbp]
  \centering
  \begin{minipage}{0.40\textwidth}
    \centering
    \includegraphics[width=\linewidth]{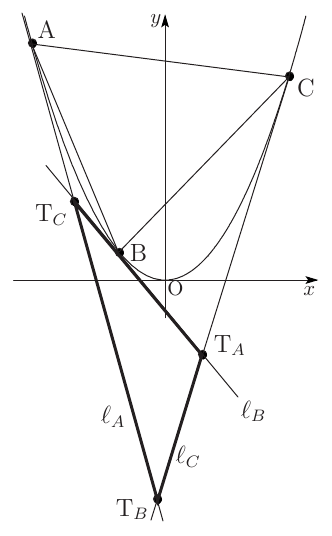}
    \subcaption{}
    \label{fig:tangent-triangle}
  \end{minipage}
  \begin{minipage}{0.40\textwidth}
    \centering
    \includegraphics[width=\linewidth]{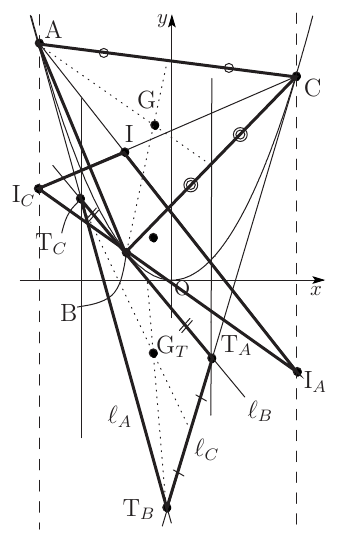}
    \subcaption{}
    \label{fig:bisector-centroid}
  \end{minipage}
  \caption{(a) Tangent triangle. (b) Centroid of the bisector triangle.}
  \label{fig:tangent-and-tree-G}
\end{figure}

\begin{definition}[Tangent triangle]\label{def:tangent-triangle}
Let \(\triangle_{\mathcal P}ABC\) be a DA triangle, and let \(\mathcal P\)
be its circumparabola. Denote by \(\ell_A,\ell_B,\ell_C\) the tangents to
\(\mathcal P\) at \(A,B,C\), respectively. Define the three pairwise intersections by
\[
T_A = \ell_B \cap \ell_C,\quad 
T_B = \ell_C \cap \ell_A,\quad 
T_C = \ell_A \cap \ell_B.
\]
Then \(\triangle_{\mathcal P}T_AT_BT_C\) is called the tangent triangle.
\end{definition}

\begin{proposition}[Centroid of the bisector triangle]\label{prop:bisector-centroid}
Let \(G\) be the centroid of \(\triangle_{\mathcal P}ABC\),
and \(G_T\) the centroid of its tangent triangle \(\triangle_{\mathcal P}T_AT_BT_C\).
Then the centroid \(G_I\) of the bisector triangle \(\triangle_{\mathcal P}II_AI_C\)
satisfies
\[
   G_I = \tfrac{1}{2}(G+G_T).
\]
\end{proposition}

\begin{proof}
Normalize to \(\mathcal P: y=x^2\).
With vertices \(A(a,a^2), B(b,b^2), C(c,c^2)\) (\(a<b<c\)),
\[
 G=\Bigl(\tfrac{a+b+c}{3}, \tfrac{a^2+b^2+c^2}{3}\Bigr).
\]
The tangents are \(\ell_A: y=2ax-a^2\), \(\ell_B: y=2bx-b^2\), \(\ell_C: y=2cx-c^2\).
Hence
\[
T_A=\Bigl(\tfrac{b+c}{2}, bc\Bigr),\ 
T_B=\Bigl(\tfrac{c+a}{2}, ca\Bigr),\ 
T_C=\Bigl(\tfrac{a+b}{2}, ab\Bigr),
\]
and the centroid of the tangent triangle is
\[
 G_T=\Bigl(\tfrac{a+b+c}{3}, \tfrac{ab+bc+ca}{3}\Bigr).
\]
From the computations in \Cref{mthm:diff-angle-bisector,cor:excenter},
the coordinates of \(I, I_A, I_C\) are rational in \(a,b,c\),
and their average yields
\[
 G_I=\Bigl(\tfrac{a+b+c}{3}, \tfrac{a^2+b^2+c^2+ab+bc+ca}{6}\Bigr).
\]
Therefore \(G_I\) is the midpoint of \(G\) and \(G_T\).
\end{proof}


\section{Applications of DA Geometry I: Correspondence with Classical Theorems and Miquel's Triangle Theorem}

\begin{quote}
In this chapter, we examine in detail the correspondence between classical theorems in Euclidean geometry and those in DA geometry, thereby clarifying the structural breadth of the latter.
Using parabolas inscribed in quadrilaterals and isosceles trapezoids as our main examples,
we derive new relations involving lengths and angles,
culminating in the establishment of the DA version of Miquel's Triangle Theorem.
These results reveal that DA geometry resonates deeply with Euclidean geometry while possessing its own autonomous structure.\\[1ex]
We begin by verifying, as the parabolic counterpart of the Inscribed Angle Theorem,
the constancy of the parabolic inscribed angle.
\end{quote}

\begin{proposition}[Parabolic Inscribed–Angle Constancy]\label{prop:parabolic-inscribed-angle}
Let four points $A,B,C,D$ lie on the parabola $y=x^{2}$, and let $X = AC \cap BD$.
Then
\[
\angle_{\mathcal P} BAX = \angle_{\mathcal P} BDC,
\qquad
\angle_{\mathcal P} ABX = \angle_{\mathcal P} ADC .
\]
\end{proposition}

\begin{proof}
In a parabola, the isoptic curve (the locus of points subtending a fixed angle at two points) is again a parabola.
Hence, the locus of points from which the segment $AB$ is seen under an equal DA–angle is parabolic,
and since both $C$ and $D$ lie on it, the angles coincide.
\end{proof}

\begin{definition}[DA Quadrilateral]
A DA quadrilateral $\square_{\mathcal P}ABCD$ is a quadrilateral none of whose sides or diagonals is parallel to the singular direction.
\end{definition}

\begin{proposition}[Characterization of Parabolically Inscribed Quadrilaterals]\label{prop:parabolic-quadrilateral}
For a DA quadrilateral $\square_{\mathcal P}ABCD$, the following conditions are equivalent:
\begin{enumerate}
\item $\square_{\mathcal P}ABCD$ is inscribed in a parabola whose axis is parallel to the projection direction.
\item The sum of the opposite DA angles is $0$.
\end{enumerate}
\end{proposition}

\begin{proof}
(1)$\Rightarrow$(2).
Normalize so that $\mathcal P\colon y=x^{2}$ and let
$A=(a,a^{2}),\ B=(b,b^{2}),\ C=(c,c^{2})$ with $a<b<c$.
For $D=(d,d^{2})$ ($c<d$), direct computation gives
$\theta_{B}=a-c$ in $\triangle_{\mathcal P}ABC$ and $\theta_{D}=c-a$ in $\triangle_{\mathcal P}ACD$;
thus $\theta_{B}+\theta_{D}=0$.

(2)$\Rightarrow$(1).
Let $D=(d,d')$ ($c<d$).
From the assumption,
\[
\theta_{D}
=\frac{c^{2}-d'}{c-d}-\frac{a^{2}-d'}{a-d}=c-a .
\]
Subtracting the right–hand side and simplifying gives
\[
\frac{(a-c)(d^{2}-d')}{(a-d)(c-d)}=0 .
\]
Since $a\neq c$ and $d\notin\{a,c\}$, we obtain $d'=d^{2}$;
hence $D$ lies on $y=x^{2}$.
Therefore $\square_{\mathcal P}ABCD$ is inscribed in the parabola $\mathcal P$.
\end{proof}

\begin{figure}[htbp]
  \centering
  \begin{minipage}{0.32\textwidth}
    \centering
    \includegraphics[width=\linewidth]{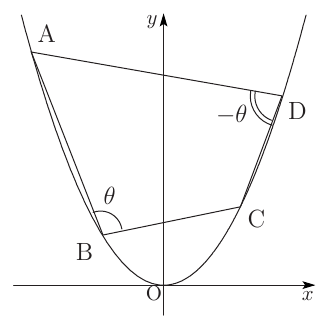}
    \subcaption{Vertical angles formed by two intersecting lines}
    \label{fig:parabolic-quadrilateral-condition}
  \end{minipage}
  \begin{minipage}{0.32\textwidth}
    \centering
    \includegraphics[width=\linewidth]{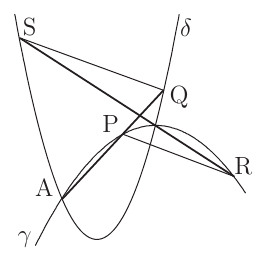}
    \subcaption{Vertical angles formed by two intersecting lines}
    \label{fig:parabolic-intersecting}
  \end{minipage}
  \begin{minipage}{0.32\textwidth}
    \centering
    \includegraphics[width=\linewidth]{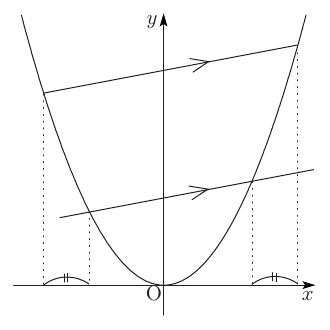}
    \subcaption{Convention: a straight angle is $0$.}
    \label{fig:parabolic-isosceles-trapezoid}
  \end{minipage}

  \caption{Vertical angles and straight–angle convention.}
\end{figure}

\begin{theorem}[Parabolic Analogue of the Intersecting–Circles Theorem]\label{thm:parabolic-intersecting}
Let two parabolas $\gamma,\delta$ intersect at distinct points $A,B$.
Let an arbitrary line through $A$ meet $\gamma,\delta$ again at $P,Q$,
and a line through $B$ meet them again at $R,S$.
Then $PR \parallel QS$.
\end{theorem}

\begin{proof}
Since $A,P,Q$ are collinear, $\angle QAB=\angle PAB$.\\
By \cref{mthm:parabola-locus}, we have
$\angle PAB=\angle PRB=\angle SRP$.
Hence the alternate angles are equal, and therefore $PR\parallel QS$.
\end{proof}

\begin{definition}[DA Isosceles Trapezoid]\label{def:parabolic-isosceles-trapezoid}
A DA quadrilateral $\square_{\mathcal P}ABCD$ is called a DA isosceles trapezoid
if exactly one pair of opposite sides is parallel
and the other pair of opposite sides are equal in the DA norm.
\end{definition}

\begin{proposition}[DA Isosceles Trapezoids and Parabolic Inscription]\label{prop:parabolic-isosceles-trapezoid}
Every DA isosceles trapezoid is inscribed in a parabola.
Conversely, two parallel lines intersecting a parabola at two points each determine a DA isosceles trapezoid.
\end{proposition}

\begin{proof}
(Forward direction)
Normalize $\mathcal P:y=x^{2}$ and set
$A=(a,a^{2}),\ B=(b,b^{2}),\ C=(c,c^{2})$ with $a<b<c$.
Let $D=(d,d')$ be a general point.
Assume $\square_{\mathcal P}ABCD$ is a DA isosceles trapezoid:
\[
AD\parallel BC,\qquad |AB|_{\mathcal P}=|CD|_{\mathcal P}.
\]
By the alternate–angle lemma for parallel lines,
\[
\angle_{\mathcal P}ACB=\angle_{\mathcal P}CAD .
\]
Let $D'=(\tilde d,\tilde d^{2})$ be the second intersection of line $AD$ with $\mathcal P$.
Equality of corresponding angles yields
\[
\angle_{\mathcal P}CAD'=\angle_{\mathcal P}ACB=b-a .
\]
On the parabola, $|XY|_{\mathcal P}=|x_{Y}-x_{X}|$, hence
$|CD'|_{\mathcal P}=|AB|_{\mathcal P}$.
Together with the hypothesis we have
\[
|CD|_{\mathcal P}=|AB|_{\mathcal P}=|CD'|_{\mathcal P}.
\]
Since $|CX|_{\mathcal P}=|x_{X}-c|$, it follows that $x_{D}=\tilde d$.
Both $D$ and $D'$ lie on $AD$ and share the same $x$–coordinate,
so $D=D'$ and consequently $d'=d^{2}$.
Therefore $\square_{\mathcal P}ABCD$ is inscribed in $\mathcal P$.

(Converse direction)
Let $A=(a,a^{2}),B=(b,b^{2}),C=(c,c^{2}),D=(d,d^{2})$ with $a<b<c<d$ on $\mathcal P:y=x^{2}$,
and assume $AD\parallel BC$.
Equality of slopes gives
\[
\frac{d^{2}-a^{2}}{d-a}=\frac{c^{2}-b^{2}}{c-b}
\Longleftrightarrow a+d=b+c .
\]
Hence $d-c=b-a$, i.e. $|CD|_{\mathcal P}=|AB|_{\mathcal P}$.
By definition, $\square_{\mathcal P}ABCD$ is a DA isosceles trapezoid.
\end{proof}

\begin{figure}[htbp]
  \centering
  \begin{minipage}{0.32\textwidth}
    \centering
    \includegraphics[width=\linewidth]{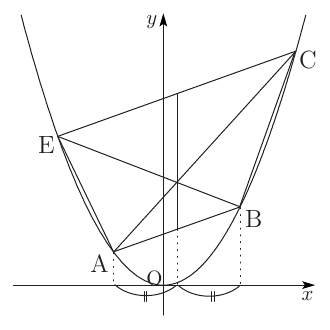}
    \subcaption{Parabolic analogue of Brahmagupta's theorem}
    \label{fig:brahmagupta-parabolic}
  \end{minipage}
 \begin{minipage}{0.32\textwidth}
    \centering
    \includegraphics[width=\linewidth]{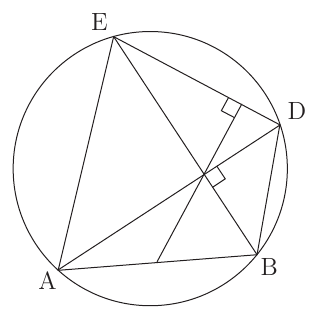}
    \subcaption{Classical Brahmagupta's theorem in Euclidean geometry}
    \label{fig:brahmagupta-euclid}
  \end{minipage}
\begin{minipage}{0.32\textwidth}
    \centering
    \includegraphics[width=\linewidth]{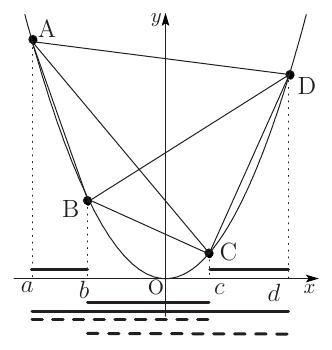}
    \subcaption{Parabolic version of Ptolemy's theorem.}
    \label{fig:parabolic-isosceles-trapezoid}
  \end{minipage}

\caption{Parabolic analogues of classical theorems: Brahmagupta and Ptolemy.}

\end{figure}

\begin{theorem}[DA Version of Brahmagupta's Theorem]\label{thm:brahmagupta-parabolic}
Let $\mathcal P: y=x^2$ be a parabola, and take four points $E,A,B,D$ on $\mathcal P$
in this order with $x$–coordinates satisfying $e<a<b<d$.
Assume $DE \parallel AB$.
Then the singular line passing through $C = AD \cap EB$
bisects the oriented DA–length $|AB|_{\mathcal P}$.
\end{theorem}

\begin{proof}
From the computation of DA–angles,
\[
\angle_{\mathcal P}(ADE)=\angle_{\mathcal P}(ABE)=a-e, \qquad
\angle_{\mathcal P}(BED)=\angle_{\mathcal P}(BAD)=d-b.
\]
Hence $a-e=d-b$, i.e., the quadrilateral $\square_{\mathcal P}ABCD$ is a DA isosceles trapezoid.
By \Cref{prop:parabolic-isosceles-trapezoid}, $ABCD$ is inscribed in the parabola $\mathcal P$.
Moreover, by the theorem on DA isosceles triangles, the $x$–coordinate of
$C=AD\cap EB$ satisfies
\[
x_C=\frac{a+b}{2}.
\]
Therefore, the singular line through $C$ bisects the DA–length of $AB$.
\end{proof}

\begin{theorem}[Brahmagupta's Theorem (Euclidean version)]\label{thm:brahmagupta-euclid}
In a cyclic quadrilateral $EABD$ whose diagonals are perpendicular,
the perpendicular dropped from the intersection of the diagonals to the side $DE$
bisects the opposite side $AB$.
\end{theorem}

\begin{remark}
Although this classical theorem is elementary, it provided the motivation
for interpreting the perpendicular line as a singular line in DA geometry.
The condition $AB\parallel DE$ corresponds to the “orthogonality of diagonals” in the Euclidean theorem,
and its necessity will later be justified through a Simson–type theorem.%
\footnote{Indeed, if perpendicularity is represented by a singular line in DA geometry,
then $\triangle_{\mathcal P}ABC$ becomes a DA isosceles triangle,
from which $AB\parallel DE$ necessarily follows.}
\end{remark}

\begin{theorem}[Ptolemy's Theorem (DA version)\footnote{
This theorem itself is not new; it has already been recognized globally in the context of DA geometry.}]
\label{thm:ptolemy-law}
Let $A,B,C,D$ be four points on a parabola whose axis is parallel to the projection direction.
Then the following holds for oriented DA–lengths:
\[
|AB|_{\mathcal P}|CD|_{\mathcal P}+|AD|_{\mathcal P}|BC|_{\mathcal P}
= |AC|_{\mathcal P}|BD|_{\mathcal P}.
\]
(In particular, for $\mathcal P:y=x^2$, one has $|PQ|_{\mathcal P}=x_Q-x_P$.)
\end{theorem}

\begin{remark}[Ptolemy and Additivity of Segments]
In a circle, Ptolemy's theorem is equivalent to the one–dimensional identity
$AB=AC+CB$ for any point $C$ on segment $AB$
(which follows by inversion centered at one vertex).
Hence, a Ptolemy–type equality does not depend on the specific nature of a circle or parabola,
but arises universally from one–dimensional projective additivity.
From this viewpoint, an algebraic expansion provides the most natural proof.
\end{remark}

\begin{proof}
Let $A=(a,a^2),B=(b,b^2),C=(c,c^2),D=(d,d^2)$.
Then a direct computation gives
\[
(b-a)(d-c)+(d-a)(c-b)=(c-a)(d-b).
\]
\end{proof}

\begin{remark}
This result has already been mentioned by Seimiya%
in the context of DA geometry.
A complete proof is included here for completeness.
\end{remark}

\begin{figure}[htbp]
  \centering
  \begin{minipage}{0.32\textwidth}
    \centering
    \includegraphics[width=\linewidth]{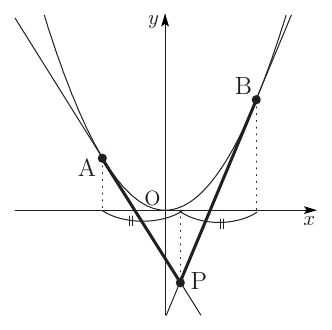}
    \subcaption{}
    \label{fig:tangent-lengths}
  \end{minipage}
  \begin{minipage}{0.32\textwidth}
    \centering
    \includegraphics[width=\linewidth]{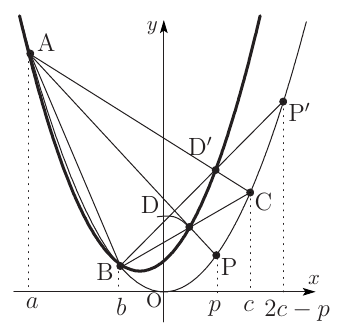}
    \subcaption{}
    \label{fig:arc-symmetric-conparabolic}
  \end{minipage}
  \caption{Equal tangent lengths and arc symmetry.}
\end{figure}

\begin{proposition}[Equality of Tangent Lengths]\label{prop:equal-tangent-lengths}
For two points $A(a,a^2)$ and $B(b,b^2)$ on the parabola $y=x^2$,
the two tangents $PA,PB$ drawn from a common external point $P=(x_P,y_P)$
have equal DA–norm lengths.
\end{proposition}

\begin{proof}
The tangent equation is $y=2tx-t^2$.
The contact parameters $t_1,t_2$ of the tangents from $P$ satisfy
$y_P=2t_i x_P-t_i^2$, hence $t_1+t_2=2x_P$.
Since the DA–norm of a $P$–tangent is $|t_i-x_P|$
(a special case of $|PQ|_{\mathcal P}=|x_Q-x_P|$),
we have $|t_1-x_P|=|t_2-x_P|$.
\end{proof}

\begin{proposition}[Arc Symmetry and Conparabolicity]\label{prop:arc-symmetric-conparabolic}
Let $\triangle_{\mathcal P}ABC$ be a DA triangle on $\mathcal P:\ y=x^2$,
and let $P$ be a point on the arc $BC$ ($a<b<p<c$).
Define $P'=(2c-p,(2c-p)^2)$ as the reflection of $P$ with respect to $x=c$,
and set $D=BC\cap AP$, $D'=CA\cap BP'$.
Then the four points $A,B,D,D'$ lie on the same parabola.
\end{proposition}

\begin{proof}
By definition of reflection, $\angle(APB)=\angle(AP'B)$.
Thus, the sum of opposite DA–angles in quadrilateral $ABDD'$ is $0$.
By \Cref{prop:parabolic-quadrilateral}, the four points $A,B,D,D'$ are inscribed in the same parabola.
\end{proof}

\begin{maintheorem}[DA Version of Miquel's Triangle Theorem]\label{mthm:miquel-triangle}
Let $ABC$ be a DA triangle, and take points $D\in BC$, $E\in CA$, $F\in AB$.
Then the three parabolas passing through $A,E,F$;
through $B,F,D$; and through $C,D,E$
intersect at a common point $M$ (finite or at infinity if their axes are parallel).
\end{maintheorem}

\begin{figure}[htbp]
  \centering
  \begin{minipage}{0.35\textwidth}
    \centering
    \includegraphics[width=\linewidth]{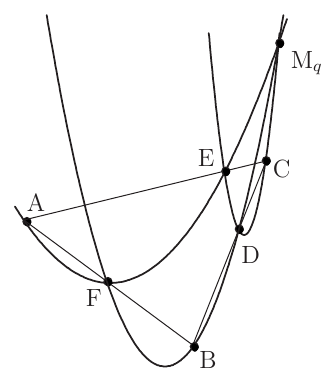}
    \label{fig:miquel-triangle}
  \end{minipage}
  \caption{The parabolic Miquel triangle.}
\end{figure}

\begin{proof}[Proof]
Let $M$ be the intersection of the parabolas through $A,E,F$ and $C,F,D$.
By \Cref{prop:parabolic-inscribed-angle},
\[
\angle_{\mathcal P}FAE=\angle_{\mathcal P}FME, \quad
\angle_{\mathcal P}ECD=\angle_{\mathcal P}EMD .
\]
Hence
\[
\angle_{\mathcal P}FMD
=\angle_{\mathcal P}FME+\angle_{\mathcal P}EMD
=\angle_{\mathcal P}FAE+\angle_{\mathcal P}ECD
=-\angle_{\mathcal P}ABC .
\]
By \Cref{prop:parabolic-quadrilateral}, the point $M$ also lies on the parabola through $F,B,D$.
Therefore, the three parabolas meet at a common point $M=M_q$.
\end{proof}


\section{Applications of DA Geometry II: Singular Lines and Miquel's Quadrilateral Theorem}
\begin{quote}
One of the striking features of DA geometry is that classical theorems are reconstructed in new forms through the presence of singular lines.
The DA version of Miquel's triangle theorem in the previous chapter was a prototypical instance; when one treats complete quadrilaterals, the constraints become even more manifest.
In this chapter, we define complete quadrilaterals intrinsically within DA geometry while avoiding singular lines, and we verify the Ceva–Menelaus theorems.
We then derive a DA version of Miquel's quadrilateral theorem, showing that DA geometry can reconstitute classical structures from a new vantage point.
Throughout, we adopt the convention of directed ratios, that is, ratios measured with respect to the oriented DA–norm along each side.
\end{quote}

\begin{figure}[htbp]
  \centering
  \begin{minipage}{0.32\textwidth}
    \centering
    \includegraphics[width=\linewidth]{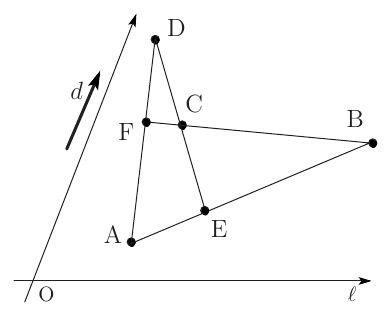}
    \subcaption{}
    \label{fig:nonsingular-quad}
  \end{minipage}
  \begin{minipage}{0.32\textwidth}
    \centering
    \includegraphics[width=\linewidth]{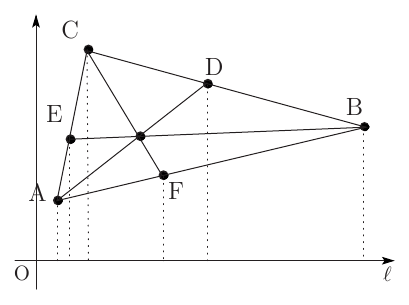}
    \subcaption{}
    \label{fig:ceva_diff_angle}
  \end{minipage}
  \begin{minipage}{0.32\textwidth}
    \centering
    \includegraphics[width=\linewidth]{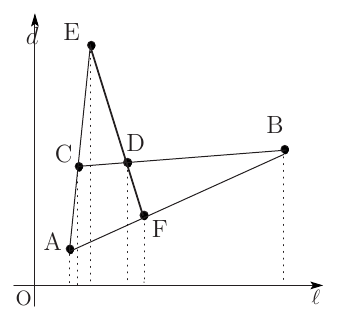}
    \subcaption{}
    \label{fig:menelaus_diff_angle}
  \end{minipage}
  \caption{Nonsingular complete quadrilateral and the Ceva–Menelaus configuration in the DA setting.}
\end{figure}

\begin{assumption}[Nonsingular configuration]\label{asp:nonsingular-quad}
Consider a complete quadrilateral formed by four lines in general position, avoiding any singular line (that is, any line parallel to the reference direction $d$).
Label the six intersection points by $A,B,C,D$ (in cyclic order of the quadrilateral) and $E=AB\cap CD$, $F=BC\cap AD$.
When considering DA circumparabolas associated with the constituent triangles, assume that no triple degenerates into a singular line
(this can be avoided, if necessary, by switching affine charts, i.e., by an appropriate gauge transformation; see \Cref{rem:singular}).
\end{assumption}

\begin{theorem}[DA Version of Ceva's Theorem]\label{thm:ceva_diff_angle}
Let $D,E,F$ lie on the sides of $\triangle_{\mathcal P}ABC$.
Then $AD,BE,CF$ are concurrent if and only if
\[
\frac{|BD|_{\mathcal P}}{|DC|_{\mathcal P}}
\cdot \frac{|CE|_{\mathcal P}}{|EA|_{\mathcal P}}
\cdot \frac{|AF|_{\mathcal P}}{|FB|_{\mathcal P}} \;=\; 1 .
\]
Here, segment ratios taken on a singular line are also measured by the DA norm.
\end{theorem}

\begin{theorem}[DA Version of Menelaus' Theorem]\label{thm:menelaus_diff_angle}
For $\triangle_{\mathcal P}ABC$, suppose points $D,E,F$ lie on the lines $BC,CA,AB$ respectively (possibly on their extensions).
Then $D,E,F$ are collinear if and only if
\[
\frac{|BD|_{\mathcal P}}{|DC|_{\mathcal P}}
\cdot \frac{|CE|_{\mathcal P}}{|EA|_{\mathcal P}}
\cdot \frac{|AF|_{\mathcal P}}{|FB|_{\mathcal P}} \;=\; -1 .
\]
\end{theorem}

\begin{proof}
In Euclidean geometry, Menelaus' theorem follows from the projective invariance of the cross–ratio.
In DA geometry, each side length in the DA norm is represented simply by the difference of $x$–coordinates, which remains invariant under projective transformations.
Hence, the same projective argument carries over directly.
Therefore, the collinearity of $D,E,F$ is equivalent to the above product being $-1$.
Since Ceva and Menelaus are projectively dual, they hold as a pair in DA geometry as well.
\end{proof}

\begin{remark}
In both \Cref{thm:ceva_diff_angle} and \Cref{thm:menelaus_diff_angle}, special care is required when the configuration involves a singular line.
By the definition of the DA (projected) length — the difference of projected $x$–coordinates — 
external division ratios are consistently measured by the DA norm as well.
\end{remark}

\begin{quote}
These results confirm that the projective nature of Ceva and Menelaus remains intact in DA geometry,
forming the algebraic foundation for the forthcoming DA version of Miquel's quadrilateral theorem.
\end{quote}

\begin{figure}[htbp]
  \centering
  \begin{minipage}{0.48\textwidth}
    \centering
    \includegraphics[width=\linewidth]{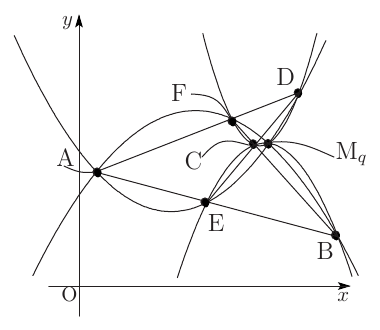}
    \label{fig:miquel-quadrilateral}
  \end{minipage}
  \caption{Parabolic Miquel quadrilateral: four circumparabolas meet at a common point.}
\end{figure}

One of the principal consequences of the nonsingular configuration
is the following theorem, which serves as the parabolic analogue of
Miquel's quadrilateral theorem.

\begin{maintheorem}[DA Version of Miquel's Quadrilateral Theorem]\label{mthm:miquel-quadrilateral}
Under the assumption of \Cref{asp:nonsingular-quad}, define
\begin{align*}
   &\mathcal C_{ABF}: \text{the circumparabola passing through } (A,B,F),\\
   &\mathcal C_{BCE}: \text{the circumparabola passing through } (B,C,E),\\
   &\mathcal C_{CDF}: \text{the circumparabola passing through } (C,D,F),\\
   &\mathcal C_{DAE}: \text{the circumparabola passing through } (D,A,E).
\end{align*}
Then the four parabolas $\mathcal C_{ABF},\ \mathcal C_{BCE},\ \mathcal C_{CDF},\ \mathcal C_{DAE}$
intersect at a common point $M$, which may lie in the finite plane or at infinity.
\end{maintheorem}

\begin{proof}
By the cyclic symmetry of the configuration, it suffices to show that any three of the four parabolas are concurrent.
We fix $M$ as one of the intersection points of $\mathcal C_{ABF}$ and $\mathcal C_{BCE}$,
and show that the remaining two also pass through $M$.

First, let $M$ be a common point of $\mathcal C_{ABF}$ and $\mathcal C_{BCE}$.
Since the four points $A,B,F,M$ lie on $\mathcal C_{ABF}$
and $B,C,E,M$ lie on $\mathcal C_{BCE}$,
the parabolic–cyclic property (\Cref{prop:parabolic-quadrilateral}, stating that the sum of opposite DA–angles in an inscribed quadrilateral is zero) gives
\begin{equation}\label{eq:paracyclic1}
\measuredangle_{\mathcal P} AFM=-\measuredangle_{\mathcal P} ABM,\qquad
\measuredangle_{\mathcal P} ECM=-\measuredangle_{\mathcal P} EBM.
\end{equation}
Using $E\in CD$ and $F\in AD$, together with the orientation convention for DA–angles,
which states that the slope $\Slp{}$ remains the same for opposite rays,
the equality of vertical angles (\Cref{prop:vertical-angle}),
and the fact that a straight angle equals zero (\Cref{lem:straight-angle-zero}),
we obtain
\begin{equation}\label{eq:line-replace}
\measuredangle_{\mathcal P} DCM = -\,\measuredangle_{\mathcal P} ECM,\qquad
\measuredangle_{\mathcal P} DFM = -\,\measuredangle_{\mathcal P} AFM.
\end{equation}
Combining \eqref{eq:paracyclic1} and \eqref{eq:line-replace}, we have
\[
\measuredangle_{\mathcal P} DCM= \measuredangle_{\mathcal P} DFM.
\]
By the constancy of the parabolic inscribed angle
(\Cref{prop:parabolic-inscribed-angle}),
it follows that $M$ lies on $\mathcal C_{CDF}$.
By exchanging the roles of $(C,D,E)$ and $(A,B,F)$, the same reasoning shows that $M\in\mathcal C_{DAE}$.
Hence the four circumparabolas
$\mathcal C_{ABF},\ \mathcal C_{BCE},\ \mathcal C_{CDF},\ \mathcal C_{DAE}$
concur at the common point $M$.
(If $M$ degenerates to a point at infinity, the result still holds under projective extension.)
\end{proof}

\begin{remark}[Treatment of configurations involving singular lines]\label{rem:singular}
If any of the triples of defining points includes a singular line
(i.e., a line parallel to the reference direction $d$),
the corresponding circumparabola degenerates into a straight line containing a point at infinity.
Except for the degenerate case where $M$ itself lies at infinity,
the nonsingular condition of \Cref{asp:nonsingular-quad}
can always be restored by a suitable change of affine chart,
that is, by a gauge transformation avoiding the direction $d$.
\end{remark}

\begin{remark}[Resolution of the Weiss--Odehnal Conjecture]\label{rem:weiss-odehnal-conj}
The conjectural statement of Weiss--Odehnal~\cite{WeissOdehnal2024}
suggests that a Miquel configuration should retain its concurrency
even in the parabolic (isotropic) limit.

Our main result, \Cref{mthm:miquel-quadrilateral},
provides the first explicit and rigorous construction of this
“parabolic Miquel point,” thereby giving a positive resolution of
the Weiss--Odehnal Conjecture.

Furthermore, we are also able to address the Brocard point as a topic related to 
the Miquel point, and in doing so we reinforce the discussion in~\cite{WeissOdehnal2024} by establishing the Brocard configuration in DA geometry together with its fundamental properties
(see Base2 for details).
\end{remark}

\begin{quote}
This parabolic version of Miquel's quadrilateral theorem completes the correspondence
between DA geometry and classical cyclic configurations,
and provides the foundation for the study of parabolic powers and inner products in the next chapter.
\end{quote}


\section{DA Perpendiculars and the DA Bisector Collinearity Theorem}
\begin{quote}
In DA geometry, singular lines play an essential role: they not only collapse distance to zero
but also seem, at first sight, to complicate perpendicularity.
Under the convention that right angles ($90^\circ$) and straight angles ($180^\circ$) are both
absorbed into the zero difference angle, it is necessary to justify that the definition
of a DA perpendicular is still valid.
In this chapter we first define the DA perpendicular and justify its consistency through a Simson-type theorem.
Then, via the structures of the circumcenter and orthocenter, we derive the DA Bisector Collinearity Theorem.
\end{quote}

\subsection{Foundation of the DA Perpendicular}

\begin{figure}[htbp]
  \centering
  \begin{minipage}{0.32\textwidth}
    \centering
    \includegraphics[width=\linewidth]{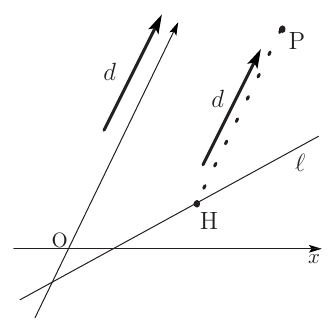}
    \subcaption{DA perpendicular: definition and projection.}
    \label{fig:diff-perp}
  \end{minipage}
  \begin{minipage}{0.32\textwidth}
    \centering
    \includegraphics[width=\linewidth]{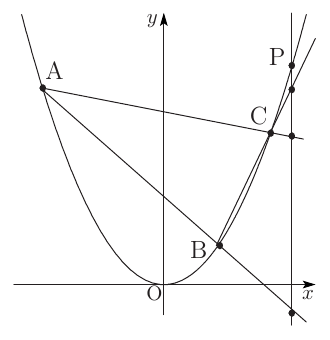}
    \subcaption{Naive Simson translation: collinearity only.}
    \label{fig:simson-diff-incorrect}
  \end{minipage}
  \begin{minipage}{0.32\textwidth}
    \centering
    \includegraphics[width=\linewidth]{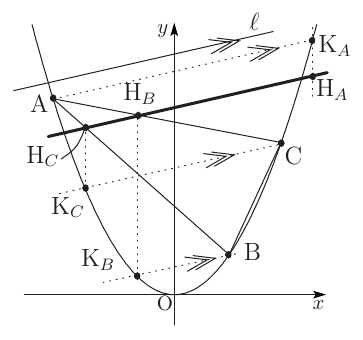}
    \subcaption{Full Simson configuration in the DA setting: concurrency at infinity and the Simson line.}
    \label{fig:simson-diff}
  \end{minipage}
  \caption{DA perpendiculars and Simson-type configurations.}
\end{figure}

\begin{definition}[Orthogonal projection]\label{def:orth-projection}
For a line $\ell$ and a point $P\notin \ell$, the foot of the perpendicular from $P$ to $\ell$
is called the \emph{orthogonal projection} of $P$ onto $\ell$.
\end{definition}

Using orthogonal projection, we define perpendiculars in DA geometry.

\begin{definition}[DA perpendicular]\label{def:diff-perp}
Let $\ell$ be a line with slope, and let $P\notin \ell$.
Denote by $H$ the orthogonal projection of $P$ onto $\ell$.
The segment $PH$ is called the DA–perpendicular (the singular line through $P$ perpendicular to $\ell$),
and $H$ is called the \emph{foot of the perpendicular}.
\end{definition}

The statement that validates this definition is Simson’s theorem.
However, if one naively translates the classical Simson theorem into DA geometry,
the result reduces to a mere collinearity statement, missing its essential geometric content.

\begin{proposition}[Naive Simson translation (collinearity only)]\label{prop:simson-diff-incorrect}
Let $\triangle_{\mathcal P}ABC$ lie on $y=\kappa x^2$, and let $P=(p,\kappa p^2)\notin\{A,B,C\}$ also lie on $\mathcal P$.
Then the feet of the DA–perpendiculars from $P$ to each side (or its extension) are collinear.
\end{proposition}

\begin{proof}
Let the equations of $AB,BC,CA$ be
$y=\kappa (a+b)x-\kappa ab$, $y=\kappa (b+c)x-\kappa bc$, $y=\kappa (c+a)x-\kappa ca$, respectively.
Since each DA–perpendicular is a vertical line $x=\mathrm{const}$, the feet are
\[
H_{AB}=\bigl(p,\,\kappa (a+b)p-\kappa ab\bigr),\quad
H_{BC}=\bigl(p,\,\kappa (b+c)p-\kappa bc\bigr),\quad
H_{CA}=\bigl(p,\,\kappa (c+a)p-\kappa ca\bigr),
\]
so all satisfy $x=p$ and hence are collinear on the vertical line $x=p$.
\end{proof}

Nevertheless, the genuine DA version of Simson’s theorem goes beyond collinearity:
via a prescribed \emph{Simson direction} it simultaneously asserts concurrency at a point at infinity
and collinearity on a finite line.
That is, the concurrency that held at a finite point in Euclidean geometry degenerates to infinity in DA geometry,
while collinearity remains finite.
This reveals that the Euclidean Simson theorem implicitly encodes two assertions—concurrency and collinearity—
which DA geometry makes explicit.

\begin{theorem}[DA Simson theorem]\label{prop:simson-diff}
Let $\triangle_{\mathcal P}ABC$ be a DA triangle inscribed in $y=x^2$,
and let $\ell$ be a line of slope $m$ (the \emph{Simson direction}).
Through each of $A,B,C$, draw a line parallel to $\ell$ meeting the parabola again at $K_A,K_B,K_C$.
From $K_A,K_B,K_C$, drop DA–perpendiculars to the opposite sides $BC,CA,AB$,
and denote their feet by $H_A,H_B,H_C$.
\begin{itemize}
\item[(a)] Each of $K_AH_A, K_BH_B, K_CH_C$ is parallel to the axis of the parabola;
hence they meet at the same point at infinity in the projective plane.
This point corresponds to the unique projective point shared by all tangent lines of the parabola.
\item[(b)] The points $H_A,H_B,H_C$ are collinear, and the direction of this line coincides with the Simson direction.
\end{itemize}
\end{theorem}

\begin{proof}
We argue algebraically to keep track of directions.
(a) follows directly from the definition of the DA–perpendicular: each $K_*H_*$ is axis-parallel, hence concurrent at infinity.

(b) Let $A=(a,a^2)$, $B=(b,b^2)$, $C=(c,c^2)$.
The line through $A$ parallel to $\ell$ is $y=m(x-a)+a^2$,
whose second intersection with $y=x^2$ has $x$–coordinate given by the quadratic $x^2-mx+ma-a^2=0$:
besides $x=a$, we have $x_{K_A}=m-a$. Similarly, $x_{K_B}=m-b$, $x_{K_C}=m-c$.

The sides are
\[
BC:\ y=(b+c)x-bc,\qquad
CA:\ y=(c+a)x-ca,\qquad
AB:\ y=(a+b)x-ab.
\]
Since DA–perpendiculars are vertical, we obtain
\[
\begin{aligned}
H_A&=\bigl(m-a,\ (b+c)(m-a)-bc\bigr),\\
H_B&=\bigl(m-b,\ (c+a)(m-b)-ca\bigr),\\
H_C&=\bigl(m-c,\ (a+b)(m-c)-ab\bigr).
\end{aligned}
\]
Each satisfies
\[
y - m x = m(a+b+c) - m^2 - (ab+bc+ca),
\]
so $H_A,H_B,H_C$ lie on the same line
$y=mx+\bigl[m(a+b+c)-m^2-(ab+bc+ca)\bigr]$, which is parallel to $\ell$.
\end{proof}

\begin{figure}[htbp]
  \centering
  \begin{minipage}{0.50\textwidth}
    \centering
    \includegraphics[width=\linewidth]{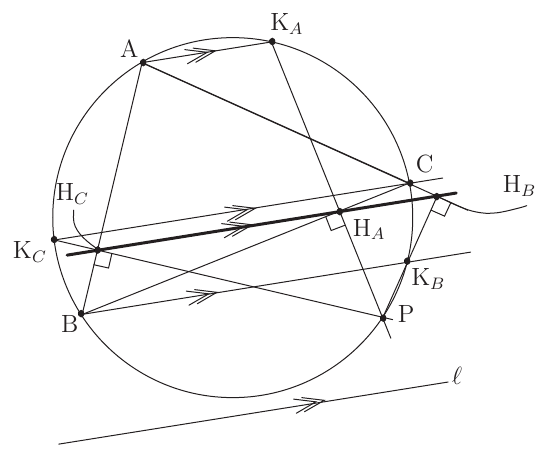}
    \label{fig:simson-euclid}
  \end{minipage}
  \caption{Simson’s theorem in Euclidean geometry traced from the DA setting.}
\end{figure}

\begin{theorem}[Simson’s theorem (Euclidean form)]
Let $\triangle ABC$ have circumcircle $\Gamma$, and fix a line $\ell$ (the \emph{Simson direction}).
Through $A,B,C$ draw lines parallel to $\ell$, meeting $\Gamma$ again at $K_A,K_B,K_C$.
From $K_A,K_B,K_C$, drop perpendiculars to $BC,CA,AB$, and denote their feet by $H_A,H_B,H_C$.
\begin{itemize}
\item[(a)] The three lines $K_AH_A, K_BH_B, K_CH_C$ are concurrent at a point $P\in\Gamma$.
\item[(b)] The points $H_A,H_B,H_C$ are collinear, and their line is parallel to $\ell$.
\end{itemize}
\end{theorem}

\begin{remark}[Points at infinity and their double structure]
The DA version shows that Euclidean concurrency appears as a degeneration to a point at infinity:
$K_AH_A, K_BH_B, K_CH_C$ do not meet at a finite point but converge to the same ideal point in the DA sense.

Moreover, this “point at infinity’’ has a twofold behavior:
with respect to the projective direction it resembles the “circle at infinity’’ of hyperbolic geometry,
while with respect to the reference direction it behaves like the “line at infinity’’ of affine geometry.
In other words, DA geometry contains two kinds of infinity simultaneously, and the choice emerges as a boundary policy.
This duality stems from the fact that $\delta=0$ encodes two conceptually distinct cases—
parallelism at infinity (straight angle) and the singular direction itself—despite sharing the same numerical value.
A detailed discussion and its connection to calibration are deferred to \emph{Base~2}.
\end{remark}

Finally, we record two immediate (and essentially tautological) center results.

\begin{proposition}[DA circumcenter]\label{prop:diff-circumcenter}
The perpendicular bisectors of the sides of $\triangle_{\mathcal P}ABC$
intersect at the point at infinity corresponding to the direction $d$.
\end{proposition}

\begin{proof}
Let $A=(a,\kappa a^2)$, $B=(b,\kappa b^2)$, $C=(c,\kappa c^2)$.
The midpoint of $BC$ has $x$–coordinate $(b+c)/2$.
Each DA–perpendicular is vertical (parallel to $d$), so the perpendicular bisectors are
$x=(b+c)/2$, $x=(c+a)/2$, $x=(a+b)/2$, hence mutually parallel and concurrent at the ideal point in direction $d$.
\end{proof}

\begin{proposition}[DA orthocenter]\label{prop:diff-orthocenter}
The DA–perpendiculars from the vertices of $\triangle_{\mathcal P}ABC$ to the opposite sides
meet at a single point at infinity.
\end{proposition}

\begin{proof}
By \Cref{def:diff-perp} and the parallel postulate, for any side there exists a unique singular line through the vertex perpendicular to it. Hence the claim.
\end{proof}

\subsection{Examples of DA–Perpendicular Phenomena without Euclidean Counterparts}
\begin{quote}
As the foregoing suggests, the natural perpendicular in DA geometry is a singular line.
Observing that this singular line is the external bisector corresponding to a positive interior DA angle,
we obtain geometric properties that hold in DA geometry but cannot hold in Euclidean geometry.
This highlights DA geometry as an autonomous system beyond the classical framework.
\end{quote}

\begin{figure}[htbp]
  \centering
  \begin{minipage}{0.32\textwidth}
    \centering
    \includegraphics[width=\linewidth]{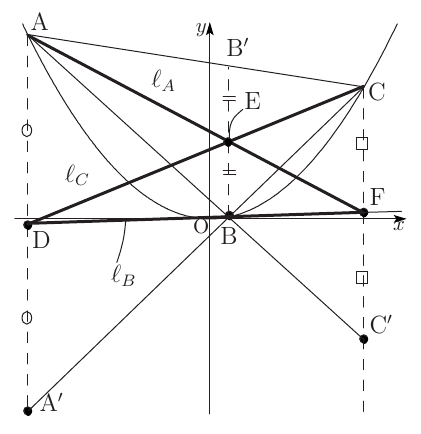}
    \subcaption{}
    \label{fig:bis-midpoint-safe}
  \end{minipage}
  \begin{minipage}{0.32\textwidth}
    \centering
    \includegraphics[width=\linewidth]{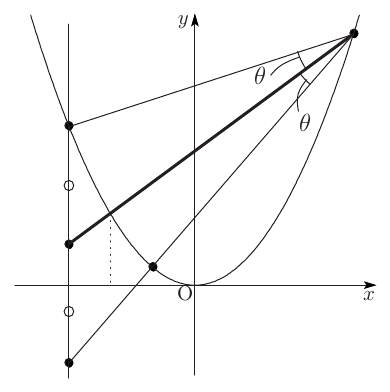}
    \subcaption{}
    \label{fig:singular-mid-point}
  \end{minipage}
  \caption{(a) The midpoint associated with the bisector.
           (b) The midpoint in the presence of a singular line.}
\end{figure}

We next give a lemma that fails in Euclidean geometry and serves as a foundation for isogonality in DA geometry.

\begin{lemma}[Bisectors and midpoints of the feet]\label{lem:bis-midpoint-safe}
The key point is that, in DA geometry, the positional relation between angle bisectors and the feet of DA–perpendiculars
reproduces, in a DA–geometric form, the Euclidean midpoint–connection phenomenon.
The proof is based on the algebraic description of DA–perpendiculars, directly showing that each intersection coincides with
the midpoint of the corresponding segment to the foot.\\
Let $\triangle_{\mathcal{P}}ABC$ lie on $\mathcal{P}\colon y=\kappa x^{2}$.
At each vertex, let $\ell_A,\ell_B,\ell_C$ denote the bisectors of the \emph{positive} DA angles, and set
\[
D = \ell_B\cap \ell_C,\quad
E = \ell_C\cap \ell_A,\quad
F = \ell_A\cap \ell_B .
\]
From $A,B,C$, drop DA–perpendiculars to the opposite sides (or their extensions), and denote the feet by
$A',B',C'$, respectively.
Then $D$ is the midpoint of $AA'$, $E$ is the midpoint of $BB'$, and $F$ is the midpoint of $CC'$.
\end{lemma}

\begin{proof}
Normalize $A=(a,\kappa a^{2})$, $B=(b,\kappa b^{2})$, $C=(c,\kappa c^{2})$ with $a<b<c$.
The sides are
$BC:\ y=\kappa (b+c)x-\kappa bc$, $CA:\ y=\kappa (c+a)x-\kappa ca$, $AB:\ y=\kappa (a+b)x-\kappa ab$.
The feet are
$A'=(a,\kappa (ab-bc+ca))$, $B'=(b,\kappa (ab+bc-ca))$, $C'=(c,\kappa (bc+ca-ab))$.

The bisectors of the \emph{positive} DA angles are
\begin{align*}
 &\ell_A:\ y=\kappa \!\left(a+\tfrac{b+c}{2}\right)x-\tfrac{\kappa\, a(b+c)}{2},\\
 &\ell_B:\ y=\kappa \!\left(b+\tfrac{c+a}{2}\right)x-\tfrac{\kappa\, b(c+a)}{2},\\
 &\ell_C:\ y=\kappa \!\left(c+\tfrac{a+b}{2}\right)x-\tfrac{\kappa\, c(a+b)}{2}.
\end{align*}
At a vertex with a negative DA angle, the bisector is a singular (vertical) line and the corresponding pairwise intersection degenerates to a point at infinity.
Thus, in the \emph{finite} intersection cases, solving for $D=\ell_B\cap\ell_C$ yields
$x_D=a$, $y_D=\tfrac{\kappa}{2}(a^{2}+ab-bc+ca)$.
The midpoint of $AA'$ is $\bigl(a,\tfrac{\kappa}{2}(a^{2}+ab-bc+ca)\bigr)$, hence $D$ coincides with it.
The cases of $E$ and $F$ are analogous.
\end{proof}

\begin{remark}[Singular lines and midpoints]\label{rem:singular-mid-point}
In the statement above we excluded cases where a side or a Ceva line lies on a singular line.
However, for singular triangles one obtains similar relations via the angle–bisector theorem,
and conversely this can be used to assign a unique midpoint even to segments lying on a singular line.
Thus the algebraic treatment is not merely formal: it underpins midpoint constructions on singular lines.
\end{remark}

\begin{proposition}[No Euclidean counterpart: midpoint alignment generally fails]\label{prop:no-euclid-trace-midpoint}
Let $ABC$ be a Euclidean triangle.
Let $\ell_A$ be the internal bisector of $\angle A$, $\ell_B$ the \emph{external} bisector of $\angle B$, and $\ell_C$ the internal bisector of $\angle C$.
Set
\[
D=\ell_B\cap\ell_C,\quad E=\ell_C\cap\ell_A,\quad F=\ell_A\cap\ell_B,
\]
and define
\[
A':=AD\cap BC,\qquad B':=BE\cap CA,\qquad C':=CF\cap AB.
\]
Then, except for special degenerate situations, there is no configuration in which
$D$ is the midpoint of $AA'$, $E$ is the midpoint of $BB'$, and $F$ is the midpoint of $CC'$ simultaneously.
\end{proposition}

\begin{proof}
Here $D=I_C$ (the $C$–excenter), $E=I$ (the incenter), and $F=I_A$ (the $A$–excenter).
Let $A':=AD\cap BC$. By the external bisector property,
\[
\frac{BA'}{A'C}=\frac{AB}{AC}=\frac{c}{b}
\]
in directed lengths.
For $D$ to be the midpoint of $AA'$, the projective ratio convention requires $AH_A=H_AA'$ for the relevant foot $H_A$,
and reconciling this with triangle coordinates gives
\[
\frac{c}{b}=-\,\frac{BH_A}{H_AC},
\]
forcing $AC=CB$.
Likewise, the midpoint conditions for $E,B'$ and for $F,C'$ force $BA=AC$ and $CB=BA$, respectively.
Cyclically, $AB=BC=CA$ follows; but even then, prescribing $\ell_B$ to be an external bisector is incompatible,
since in an equilateral triangle internal and external bisectors do not coincide.
Hence the assertion generally fails.
\end{proof}

\begin{remark}[Isogonality as an intrinsic structure in DA geometry]\label{rem:isogonal-structure}
While the collinearity phenomena above occur as projective invariants in other geometries,
the midpoint alignment linking angle bisectors and feet of perpendiculars is specific to DA geometry.
This structure underlies the isogonal concept (equal–angle conjugation) in DA geometry,
making explicit a new symmetry arising from treating angles as primary quantities.
\end{remark}

\subsection{Examples Reducible to Euclidean Geometry and the Main Theorem}
\begin{quote}
In the previous subsection we confirmed, as a property intrinsic to DA geometry, a surprising coincidence between midpoints and angle bisectors.
On the other hand, for related concurrency–collinearity configurations, computations in DA geometry reduce to Euclidean geometry.
A representative case is the following DA Bisector Collinearity Theorem, which bridges DA and classical geometry.
A striking feature is that an elementary computation internal to DA geometry—namely \Cref{mthm:DABCT}—translates directly into
a new collinearity theorem in Euclidean geometry (\Cref{thm:Angle-Bisector-Collinearity-Theorem}).
This “export” shows that DA geometry not only stands alone but also contributes genuinely new theorems to the classical theory.
\end{quote}

\begin{figure}[htbp]
  \centering
  \begin{minipage}{0.40\textwidth}
    \centering
    \includegraphics[width=\linewidth]{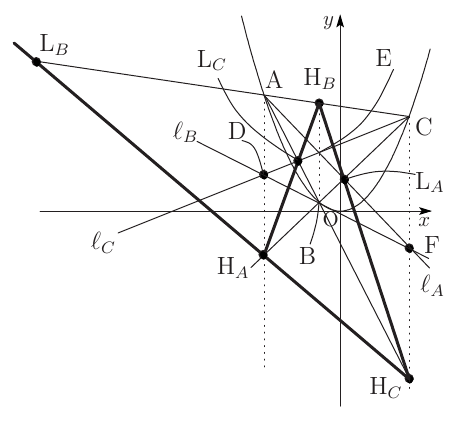}
    \subcaption{}
    \label{fig:bis-mthm-DABCT}
  \end{minipage}
  \begin{minipage}{0.40\textwidth}
    \centering
    \includegraphics[width=\linewidth]{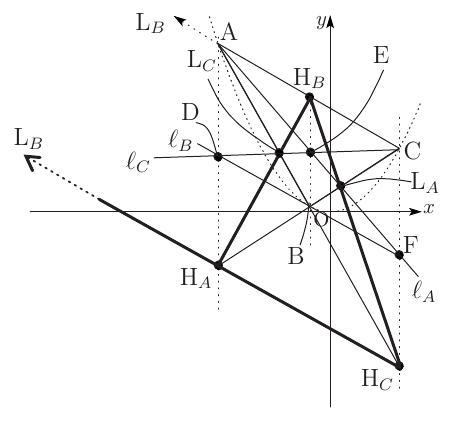}
    \subcaption{}
    \label{fig:singular-mid-point}
  \end{minipage}
  \caption{DA bisectors and concurrency/collinearity.}
\end{figure}

\begin{maintheorem}[DA Bisector Collinearity Theorem]\label{mthm:DABCT}
Consider a DA triangle $\Ptri ABC$ on $y=\kappa x^{2}$ with $\kappa>0$.
Let $\ell_A,\ell_B,\ell_C$ be the DA–angle bisectors at the vertices, and set
\[
D=\ell_B\cap\ell_C,\quad E=\ell_C\cap\ell_A,\quad F=\ell_A\cap\ell_B.
\]
Let $H_A,H_B,H_C$ be the feet of the DA–perpendiculars from $A,B,C$ to the opposite sides
(or their extensions), respectively.
Define also
\[
L_A:=BC\cap\ell_A,\qquad L_B:=CA\cap\ell_B,\qquad L_C:=AB\cap\ell_C.
\]
Then:
\begin{itemize}
\item[(a)] If $\Ptri ABC$ is not DA–isosceles, then
\begin{align*}
  &AB,\ \ell_C,\ H_AH_B \ \text{are concurrent at } L_C,\\
  &BC,\ \ell_A,\ H_BH_C \ \text{are concurrent at } L_A,\\
  &CA,\ \ell_B,\ H_CH_A \ \text{are concurrent at } L_B.
\end{align*}
\item[(b)] The three points $L_A,L_B,L_C$ are collinear.
\end{itemize}
\end{maintheorem}

\begin{proof}
We outline the strategy.
Non–isoscelesness ensures that $AH_AH_CC$ does not form a parallelogram and that
$L_B=BC\cap\ell_B$ appears as a finite point.
Assuming $L_A,L_B,L_C$ are finite,
we determine the DA–angle bisectors from the midpoint property of feet and deduce the required concurrencies.

\medskip\noindent
(a)\; By \Cref{lem:bis-midpoint-safe}, $D$ and $E$ are the midpoints of $AH_A$ and $BH_B$, respectively, and
\[
  |AL_C|_{\mathcal P}:|L_CH_A|_{\mathcal P}=1:1,\qquad
  |L_CH_B|_{\mathcal P}:|L_CB|_{\mathcal P}=1:1.
\]
Hence, by consistency with the DA–angle–bisector rule (\Cref{mthm:diff-angle-bisector}),
$L_CD$ and $L_CE$ bisect $\angle_{\mathcal P}AL_CH_A$ and $\angle_{\mathcal P}BL_CH_B$, respectively.
Applying \Cref{lem:bis-midpoint-safe} again yields that $AB,\ \ell_C,\ H_AH_B$ concur at $L_C$ (and cyclically for the others).

\medskip\noindent
(b)\; Using the projective invariance of DA–length ratios and applying \Cref{thm:ceva_diff_angle},
and observing that no singular line occurs, we get
\[
\frac{|CA|_{\mathcal P}}{|AL_B|_{\mathcal P}}\cdot
\frac{|L_BD|_{\mathcal P}}{|DB|_{\mathcal P}}\cdot
\frac{|BL_A|_{\mathcal P}}{|L_AC|_{\mathcal P}}
=1,
\]
which implies that $L_A,L_B,L_C$ are collinear.
\end{proof}

\begin{remark}
Each concurrency above arises from “bisecting the same DA angle,”
but there are two situations:
\begin{itemize}
  \item For $L_A$ and $L_C$: a singular line induces a folding, so the bisector theorem appears in the form of vertical angles, leading to concurrency.
  \item For $L_B$: the singular line occurs in a single direction, so both $L_BD$ and $L_BF$ bisect $\angle_{\mathcal P}CL_BH_C$, and concurrency follows.
\end{itemize}
\end{remark}

The theorem intertwines the nature of difference angles with collinearity in a particularly neat way.
As a consequence, we obtain the following result in Euclidean geometry.

\begin{figure}[htbp]
  \centering
  \begin{minipage}{0.55\textwidth}
    \centering
    \includegraphics[width=\linewidth]{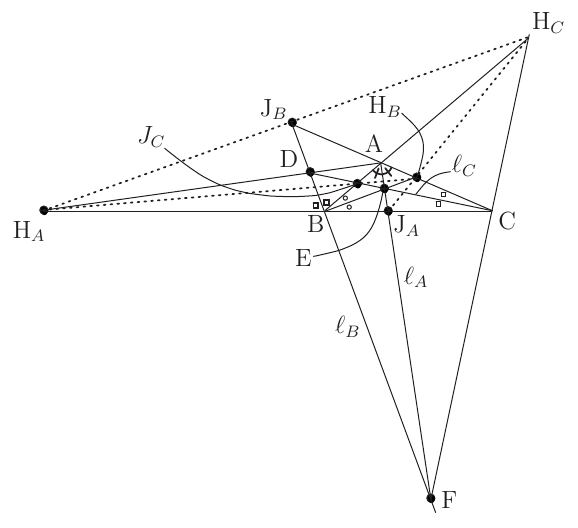}
    \label{fig:bis-ABC-Thm}
  \end{minipage}
  \caption{Angle–bisector collinearity (\Cref{thm:Angle-Bisector-Collinearity-Theorem}).}
\end{figure}

\begin{theorem}[Angle–Bisector Collinearity Theorem]\label{thm:Angle-Bisector-Collinearity-Theorem}
Let $ABC$ be a Euclidean triangle.
Let $\ell_A$ be the internal bisector of $\angle A$, $\ell_B$ the external bisector of $\angle B$, and $\ell_C$ the internal bisector of $\angle C$, and set
\[
\ell_B\cap\ell_C=D,\qquad \ell_C\cap\ell_A=E,\qquad \ell_A\cap\ell_B=F.
\]
Define
\[
H_A:=BC\cap AD,\qquad H_B:=CA\cap BE,\qquad H_C:=AB\cap CF,
\]
and
\[
J_C:=AB\cap\ell_C,\qquad
J_A:=BC\cap\ell_A,\qquad
J_B:=CA\cap\ell_B.
\]
Then:
\begin{itemize}
\item[(a)] (Concurrency)
\begin{align*}
    &AB,\ \ell_C,\ H_AH_B\ \text{are concurrent at }J_C,\\
    &BC,\ \ell_A,\ H_BH_C\ \text{are concurrent at }J_A,\\
    &CA,\ \ell_B,\ H_CH_A\ \text{are concurrent at }J_B.
\end{align*}
\item[(b)] The points $J_A,J_B,J_C$ are collinear.
\end{itemize}
\end{theorem}

\begin{remark}
To the best of the author’s knowledge, this collinearity theorem does not appear in the classical triangle–geometry literature%
\footnote{See also Clark Kimberling’s \emph{Encyclopedia of Triangle Centers},
\url{http://faculty.evansville.edu/ck6/encyclopedia/ETC.html}, accessed Oct.\ 2025.},
although related configurations are discussed in sources such as the ETC.
\end{remark}

\begin{proof}
We prove (a); the other cases follow cyclically.
By the angle–bisector theorem,
\[
\frac{AJ_C}{J_CB}=\frac{AC}{CB}=\frac{b}{a}.
\]
Since $AD$ is the external bisector at $A$,
\[
\frac{BH_A}{H_AC}=-\frac{c}{b},
\]
and since $BE$ is the internal bisector at $B$,
\[
\frac{CH_B}{H_BA}=\frac{a}{c}.
\]
Therefore,
\[
\frac{BH_A}{H_AC}\cdot \frac{CH_B}{H_BA}\cdot \frac{AJ_C}{J_CB}
=\Bigl(-\frac{c}{b}\Bigr)\!\cdot\!\frac{a}{c}\!\cdot\!\frac{b}{a}=-1,
\]
so by Menelaus’ theorem the points $H_A,H_B,J_C$ are collinear.
Hence $AB,\ \ell_C,\ H_AH_B$ are concurrent at $J_C$.

For (b), in trilinear coordinates with respect to $\triangle ABC$ we have
\[
J_A=(0:1:1),\quad J_B=(1:0:-1),\quad J_C=(1:1:0),
\]
which satisfy $x-y+z=0$, proving collinearity.
\end{proof}

\begin{remark}[Summary: contrasting the presence/absence of Euclidean counterparts]
Based on the perpendicular structure in DA geometry,
we first established the “bisectors and midpoint of feet’’ theorem, which has no Euclidean counterpart.
In contrast, the DA Bisector Collinearity Theorem arises from elementary computations internal to DA geometry
and simultaneously reduces to Euclidean geometry, yielding a new concurrency/collinearity statement there.

This contrast shows that DA geometry is not a mere imitation of the classical theory:
it both produces genuinely new phenomena on its own and, at the same time,
exports new propositions back to classical geometry.
\end{remark}

\subsection{Isogonal conjugacy in the DA setting}

In this subsection, as an application of \Cref{lem:bis-midpoint-safe}, we show that
an $m\!:\!n$ division of an angle realizes the \emph{same} ratio $m\!:\!n$ on the singular line at each vertex,
and that the DA version of isogonality preserves concurrency.
(The ratios are independent of the parabola parameter $\kappa>0$.\footnote{Henceforth we take $y=x^2$.
All statements remain valid for $y=\kappa x^2$, since the relevant ratios are invariant in $\kappa$.})

We begin with a lemma directly related to \Cref{lem:bis-midpoint-safe}.

\begin{figure}[htbp]
  \centering
  \begin{minipage}{0.32\textwidth}
    \centering
    \includegraphics[width=\linewidth]{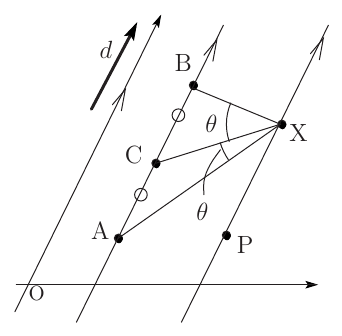}
    \subcaption{}
    \label{fig:singular-midpoint-angle-general}
  \end{minipage}
  \begin{minipage}{0.32\textwidth}
    \centering
    \includegraphics[width=\linewidth]{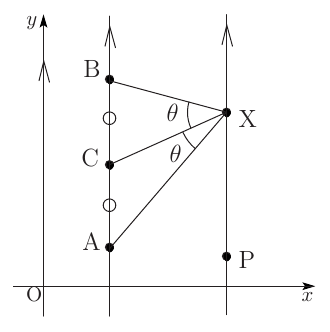}
    \subcaption{}
    \label{fig:singular-midpoint-angle-normalized}
  \end{minipage}
  \begin{minipage}{0.32\textwidth}
    \centering
    \includegraphics[width=\linewidth]{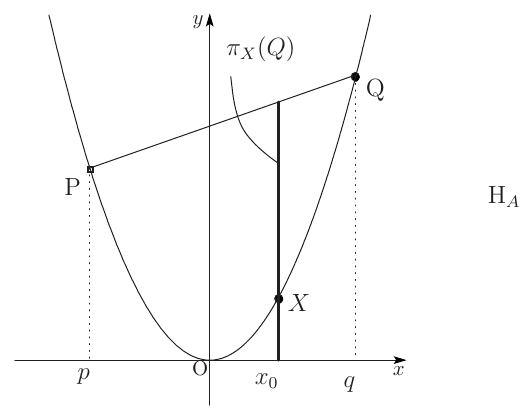}
    \subcaption{}
    \label{fig:singular-proj-length}
  \end{minipage}
  \caption{(a) General projective direction $d$.
           (b) Normalized case (taking $\ell$ as the vertical axis).
           (c) Definition of the singular projective length $\pi_X(Q)$.}
\end{figure}

\begin{lemma}[Angle symmetry via a midpoint on a singular line]\label{lem:singular-midpoint-angle}
Let $\ell$ be a singular line and take $A,B\in\ell$ with midpoint $C$.
For any point $P\notin\ell$, let $m$ be the singular line through $P$; then for every $X\in m$,
\[
  \angle_{\mathcal P}AXC \;=\; \angle_{\mathcal P}CXB .
\]
\end{lemma}

\begin{proof}
This follows immediately from the definition of the difference angle:
the slope differences are independent of the projective direction, hence the claim is normalization–free.
\end{proof}

\begin{definition}[Singular projective length]\label{def:singular-proj-length}
Fix $P=(p,p^2)$ and let $X=(x_0,x_0^2)$ be a vertex point.
For $Q=(q,q^2)$ on the parabola, let $X_Q$ denote the intersection of the chord $PQ$ with the vertical line $x=x_0$.
Define the \emph{singular projective length} with respect to $X$ by
\[
\pi_X(Q):=(p+q)x_0-pq .
\]
\end{definition}

\begin{lemma}[Linearity]\label{lem:singular-linearity}
As a function of $q$, $\pi_X(Q)$ is affine linear; hence
\[
\pi_X(\lambda q_1+(1-\lambda)q_2)=\lambda\pi_X(q_1)+(1-\lambda)\pi_X(q_2).
\]
\end{lemma}

\begin{remark}
The only property we require is precisely the linearity
\[
\pi_X(\lambda q_1+(1-\lambda)q_2)
= \lambda \pi_X(q_1) + (1-\lambda)\pi_X(q_2),
\]
which permits a direct application to the DA–Ceva condition and enables concurrency arguments.
\end{remark}

\begin{figure}[htbp]
  \centering
  \begin{minipage}{0.45\textwidth}
    \centering
    \includegraphics[width=\linewidth]{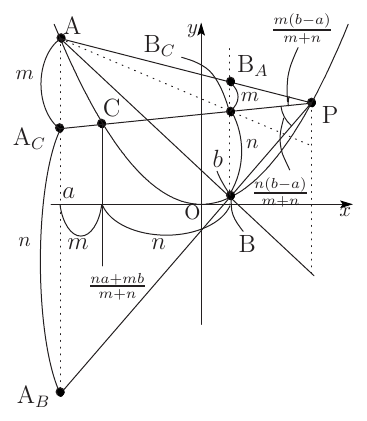}
    \label{fig:mn-bisector-singular-division}
  \end{minipage}
  \caption{Illustration of \Cref{thm:mn-bisector-singular-division}.}
\end{figure}

By \Cref{lem:singular-linearity}, an $m\!:\!n$ angle division can be encoded
as an $m\!:\!n$ internal division on the corresponding singular line.
In Euclidean geometry, only the angle bisector links naturally to edge ratios;
for a general $m\!:\!n$ division no such correspondence holds.
In DA geometry, however, this correspondence extends to arbitrary $m:n$.
We now state the theorem.
(Here, an “$m\!:\!n$ internal bisector of an angle’’ is a line that splits the angle in the ratio $m\!:\!n$.)

\begin{theorem}[Angle $m\!:\!n$ division and $m\!:\!n$ division on a singular line]\label{thm:mn-bisector-singular-division}
Let $A=(a,a^2)$, $B=(b,b^2)$, $P=(p,p^2)$ be distinct points on the parabola.
Let $PC$ be the $m\!:\!n$ internal bisector of $\angle APB$, and write
\[
c=\tfrac{na+mb}{m+n},\qquad C=(c,c^2)
\]
(so $C$ lies in the same connected component and satisfies $\angle APC:\angle CPB=m:n$).
For $X\in\{A,B\}$, let $X_Y$ denote the foot of the DA–perpendicular from $X$ to the line $PY$.
Then
\[
AA_C:A_CA_B=m:n,\qquad BB_C:B_CB_A=n:m.
\]
In particular, for $m=n$ one has: $A_C$ is the midpoint of $AA_B$, and $B_C$ is the midpoint of $BB_A$.
\end{theorem}

\begin{proof}
For $X=A$, $\pi_A(q)=(p+q)a-pq=(a-p)q+pa$ is affine linear in $q$.
Since $c=\dfrac{na+mb}{m+n}$, we get
$\pi_A(c)=\dfrac{n}{m+n}\pi_A(a)+\dfrac{m}{m+n}\pi_A(b)$.
Hence on the vertical line $x=a$ we have $AA_C:A_CA_B=m:n$.
The case $X=B$ is analogous, yielding $BB_C:B_CB_A=n:m$.
\end{proof}

\begin{definition}[Side–based angle division and DA isogonal]\label{def:isog}
For a vertex angle $\theta_A$, we make explicit the side from which the angle division is measured:
denote by $\ell_A^{(m:n;BA)}$ the $m\!:\!n$ divider measured from side $BA$,
and by $\ell_A^{(m:n;CA)}$ the one measured from side $CA$.
Let $L_A$ be the DA bisector of $\theta_A$.
For any ray $\ell$ at $A$, define its \emph{DA isogonal} $\ell^\star$ as the ray symmetric to $\ell$ with respect to $L_A$ in the DA sense.
(When the angle is negative or the line is singular, we take the external bisector—i.e., the singular line—as the axis of symmetry;
that line is then fixed by isogonality.)
\end{definition}

\begin{figure}[htbp]
  \centering
  \begin{minipage}{0.45\textwidth}
    \centering
    \includegraphics[width=\linewidth]{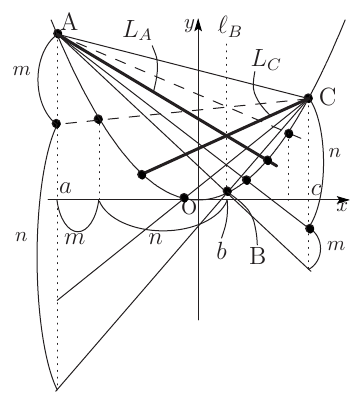}
    \label{fig:isog-mn-mixed}
  \end{minipage}
  \caption{Mixed $m\!:\!n$ angle division:
  $\ell_A,\ell_C$ are measured from sides $BA,CB$, and $\ell_B$ is the singular line at $B$.}
\end{figure}

\begin{proposition}[Mixed side–bases for $m\!:\!n$ divisions and preservation of concurrency]\label{prop:isog-mn-mixed}
Assume $\theta_B<0$.
Let $\ell_A=\ell_A^{(m:n;BA)}$, $\ell_C=\ell_C^{(m:n;CB)}$, and let $\ell_B$ be the singular line through $B$.
Then $\ell_A,\ell_B,\ell_C$ are concurrent.
Moreover, the isogonals satisfy $\ell_B^\star=\ell_B$, and
$\ell_A^\star,\ell_B^\star,\ell_C^\star$ are also concurrent.
\end{proposition}

\begin{proposition}[Common side–base ($\alpha,\beta,\gamma$) and preservation of concurrency]\label{prop:isog-alpha}
Suppose $\ell_A=\ell_A^{(\alpha:1-\alpha;BA)}$, $\ell_B=\ell_B^{(\beta:1-\beta;CB)}$,
and $\ell_C=\ell_C^{(\gamma:1-\gamma;CA)}$ are concurrent.
Then their DA isogonals $\ell_A^\star,\ell_B^\star,\ell_C^\star$ are also concurrent.
\end{proposition}

\begin{proof}[Sketch of proof]
The key is the DA–Ceva condition and the observation that the isogonal map in this setting inverts the relevant ratios.
Let $T_A=\ell_A\cap BC$ (and similarly $U_B,V_C$).
The DA–Ceva concurrency condition reads
\[
\frac{\overrightarrow{BT_A}}{\overrightarrow{T_AC}}\cdot
\frac{\overrightarrow{CU_B}}{\overrightarrow{U_BA}}\cdot
\frac{\overrightarrow{AV_C}}{\overrightarrow{V_CB}}=1 .
\]
By definition, the ratio $BT_A:T_AC$ is a rational function of $(m,n)$ (or $\alpha,\beta,\gamma$),
and under the DA isogonal it transforms by inversion $r\mapsto r^{-1}$.
Hence the product remains $1$, so concurrency is preserved.
When $\theta_B<0$, the bisector at $B$ is the singular line and is fixed by isogonality.
\end{proof}

\begin{remark}
All ratio formulas above follow from the mere linearity of the singular projective length.
If one introduces a DA area, an alternative derivation parallel to the usual Euclidean area–ratio arguments is also possible (details omitted).
\end{remark}


\section{Hierarchy of Similarity and Congruence}

\begin{quote}
In DA geometry a norm (the DA–length) is available, but it is not a priori clear how far this quantity should be built into a notion of “congruence.”
It is therefore natural to introduce a \emph{hierarchy} of strengths for “similarity” and “congruence.”
Below we define similarity in DA geometry in several tiers.\\
First, by \Cref{mthm:triangle-equation}, ratios of corresponding three sides suggest a notion of similarity; however, since this originates purely from the angle definition, it constitutes the \emph{weakest} form of similarity.
\end{quote}

\subsection{Definitions of Similarity and Basic Relations}

\begin{definition}[Norm–similarity (SSS)]\label{def:sim-SSS}
If
\[
\dnorm{AB}:\dnorm{BC}:\dnorm{CA}=\dnorm{DE}:\dnorm{EF}:\dnorm{FD},
\]
we write $\Ptri ABC \simSSS \Ptri DEF$.
\end{definition}

Since the sum of oriented DA angles of a triangle is $0$, two angles determine the third.
This motivates an angle–based similarity.

\begin{definition}[Similarity = DA–angle similarity (AA)]\label{def:sim-AA}
For DA triangles $\Ptri ABC$ and $\Ptri DEF$, if two pairs of corresponding oriented DA angles are equal, we write
$\Ptri ABC \simAA \Ptri DEF$ (also denoted $\sim_{\mathcal P}$).%
\footnote{According to Florian Cajori, \textit{A History of Mathematical Notations}, Vol.~II (1929), G.~W.~Leibniz used $a\sim b$ for similarity in his unpublished 1679 manuscript “Characteristica Geometrica.” While symbols close to $\mathrel{\backsim}$ also appear in his usage, the modern convention predominantly employs $\sim$. In Japanese literature, $\mathrel{\backsim}$ is still frequently used.}
\end{definition}

\begin{definition}[Signed SAS similarity]\label{def:Signed-SAS similarity}
If the ratios of two corresponding sides are equal and the included oriented DA angles (with sign) are equal, we say the triangles are in \emph{signed SAS similarity} (abbrev.\ SAS$\pm$).
\end{definition}

\begin{proposition}[Limitation of norm–similarity]\label{prop:SSS-not-AA}
There exist pairs of DA triangles that are norm–similar ($\simSSS$) but not DA–angle similar ($\simAA$).
\end{proposition}

\begin{proof}
Let the circumparabola be $y=x^2$. For DA triangles $\triangle_{\mathcal P}ABC$ and $\triangle_{\mathcal P}A'B'C'$, take
$x$–coordinates $a<b<c$ for $A,B,C$ and $ka<kb<kc$ for $A',B',C'$ with $k>0$.
Then the two triangles are norm–similar, but
\[
\measuredangle_{\mathcal P}B=a-c,\qquad
\measuredangle_{\mathcal P}B'=ka-kc=k(a-c),
\]
so the angles are not equal in general; hence not $\simAA$.
\end{proof}

\begin{proposition}[SAS$\pm$ implies SSS and AA]\label{prop:SAS-to-SSS-AA}
Let $\triangle ABC$ and $\triangle A'B'C'$ be DA triangles with
$x$–coordinates $A(a),B(b),C(c)$ ($a<b<c$) and $A'(a'),B'(b'),C'(c')$.
Assume the DA–norm ratios of two corresponding sides agree:
\[
\frac{b-a}{\,b'-a'\,}=\frac{c-a}{\,c'-a'\,}=\frac1k \qquad (k>0),
\]
and the included oriented DA angles are equal (SAS$\pm$).
Then (i) the remaining side ratio also agrees (hence SSS), and (ii) all angles agree (hence AA).
\end{proposition}

\begin{proof}
\textbf{(i) SAS$+$ (included angles of the same sign).}
WLOG take the included angle at $A$ and write
\[
\measuredangle A=\measuredangle A' = c-b
\]
(by the definition of oriented DA angles).
From the side–ratio hypothesis,
\[
c'-a'=k(c-a),\qquad b'-a'=k(b-a).
\]
Hence
\[
c'-b'=(c'-a')-(b'-a')=k\{(c-a)-(b-a)\}=k(c-b),
\]
so
\[
\frac{c-b}{\,c'-b'\,}=\frac1k,
\]
which is the third side ratio—thus SSS holds.

For the angles: under scaling of the $x$–axis by $k$, the quadratic coefficient of the model parabola rescales reciprocally; the oriented DA angle values are preserved after matching the scale (cf.\ the DA–parabola scaling rule).
Thus
\[
\measuredangle C'=\frac{b'-a'}{k}=\frac{k(b-a)}{k}=b-a=\measuredangle C,
\]
and by the DA angle sum $=0$, the remaining angle also matches, giving AA.

\medskip
\textbf{(ii) SAS$-$ (included angles of opposite sign).}
Similarly, take the included (negative) angle at $B$:
\[
\measuredangle B=\measuredangle B' = a-c .
\]
With
\[
a'-c'=k(a-c),\qquad b'-c'=k(b-c),
\]
we obtain
\[
a'-b'=(a'-c')-(b'-c')=k\{(a-c)-(b-c)\}=k(a-b),
\]
so SSS holds. Matching the parabola scaling as above gives equality of the remaining angles, hence AA.
\end{proof}

\begin{proposition}[Equivalence of SAS$\pm$ and AA]\label{prop:SAS-AA-equiv}
For DA triangles, $\Ptri ABC \simSAS \Ptri A'B'C'$ if and only if $\Ptri ABC \simAA \Ptri A'B'C'$.
\end{proposition}

\begin{proof}
($\Rightarrow$) Under SAS$\pm$, two side ratios are fixed and the included angle agrees.
Using the DA angle sum $=0$ and the circumparabola scaling rule (coefficient rescales reciprocally to the similarity factor), the remaining angle also agrees, giving AA.
($\Leftarrow$) Under AA, the included angle equality is immediate; by \Cref{prop:AA-implies-SSS} all side ratios agree, hence SAS$\pm$.
\end{proof}

\begin{proposition}[AA implies SSS]\label{prop:AA-implies-SSS}
If $\Ptri ABC \simAA \Ptri A'B'C'$, then $\Ptri ABC \simSSS \Ptri A'B'C'$.
\end{proposition}

\begin{proof}
WLOG take the circumparabola coefficient for $\Ptri ABC$ as $1$ and for $\Ptri A'B'C'$ as $k\neq0$.
With the coordinate matching $a'=ka$, $b'=kb$, $c'=kc$, each side's DA–norm ratio coincides, so SSS follows.
\end{proof}

\begin{figure}[htbp]
  \centering
  \begin{minipage}{0.32\textwidth}
    \centering
    \includegraphics[width=\linewidth]{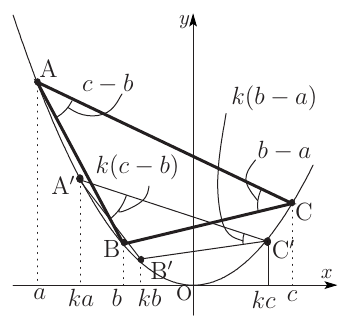}
    \subcaption{}
    \label{fig:sss-not-to-aa}
  \end{minipage}
  \begin{minipage}{0.32\textwidth}
    \centering
    \includegraphics[width=\linewidth]{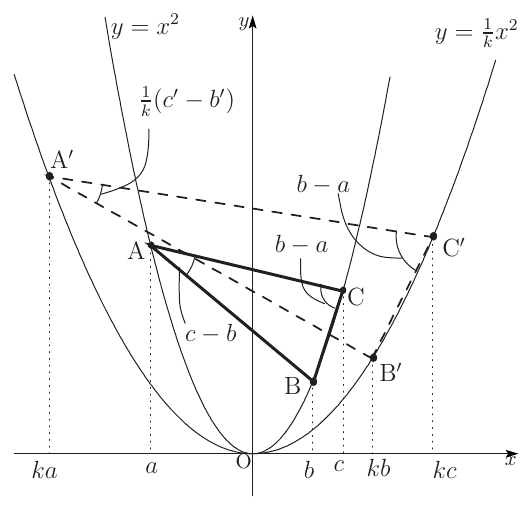}
    \subcaption{}
    \label{fig:sas-plus-to-aa}
  \end{minipage}
  \begin{minipage}{0.32\textwidth}
    \centering
    \includegraphics[width=\linewidth]{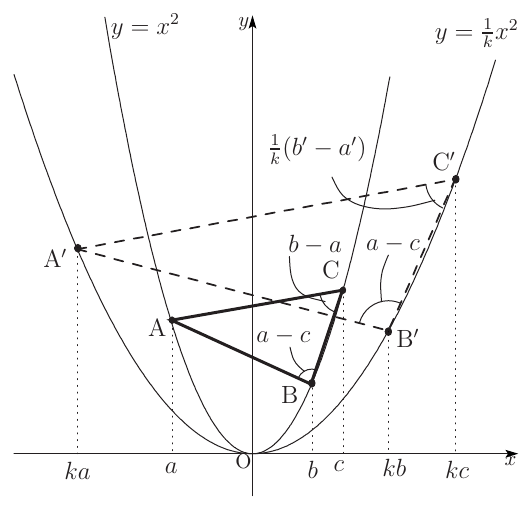}
    \subcaption{}
    \label{fig:sas-minus-to-aa}
  \end{minipage}
  \caption{(a) An example of $\simSSS$ without $\simAA$.
           (b) Positive SAS: $\simSAS$ implies both $\simSSS$ and $\simAA$.
           (c) Negative SAS: the same implication holds with a negative included angle.}
\end{figure}

\begin{corollary}[Chain of strengths]\label{cor:sim-chain}
\[
\boxed{\ \Ptri ABC \simSAS \Ptri A'B'C' \ \Longleftrightarrow\ \Ptri ABC \simAA \Ptri A'B'C' \ \subset\ \Ptri ABC \simSSS \Ptri A'B'C'\ }.
\]
The inclusion follows from \Cref{prop:AA-implies-SSS}, and the strictness from \Cref{prop:SSS-not-AA}.
\end{corollary}

\begin{proposition}[Norm–similarity of the tangent triangle]\label{prop:tangent-similarity}
Let $\Ptri ABC$ be a DA triangle and $\Ptri A'B'C'$ its tangent triangle.
Then $\Ptri A'B'C' \simSSS \Ptri ABC$.
\end{proposition}

\begin{proof}
Normalize $\mathcal P: y=x^2$ and write
$A(a,a^2), B(b,b^2), C(c,c^2)$ with $a<b<c$.
As computed in \Cref{prop:bisector-centroid}, the vertices of the tangent triangle are
\[
A'\Bigl(\tfrac{b+c}{2},\, bc\Bigr),\quad
B'\Bigl(\tfrac{c+a}{2},\, ca\Bigr),\quad
C'\Bigl(\tfrac{a+b}{2},\, ab\Bigr).
\]
Since DA–norm side lengths are proportional to $x$–differences,
\[
|A'B'|_{\mathcal P}:|B'C'|_{\mathcal P}:|C'A'|_{\mathcal P}
=\tfrac{b-a}{2}:\tfrac{c-b}{2}:\tfrac{c-a}{2}
=|AB|_{\mathcal P}:|BC|_{\mathcal P}:|CA|_{\mathcal P}.
\]
Hence $\Ptri A'B'C' \simSSS \Ptri ABC$.
\end{proof}

\begin{remark}
In general the DA angles do not match (AA fails).
For instance, the angle at $A'=\ell_B\cap\ell_C$ equals the slope difference of the tangents:
\[
\measuredangle_{\mathcal P} B'A'C'= \Slp(\ell_C)-\Slp(\ell_B)=2c-2b,
\]
whereas the corresponding angle of $\Ptri ABC$ is $c-b$, so the factors disagree (except in special configurations).
\end{remark}

\begin{proposition}[DA–angle similarity via diagonal section]\label{prop:diag-AA}
Let $A,B,C,D$ lie in order on the parabola $y=x^2$, forming an inscribed quadrilateral.
If $X=AC\cap BD$, then
\[
\triangle XAB \simAA \triangle XCD,\qquad
\triangle XBC \simAA \triangle XAD.
\]
\end{proposition}

\begin{proof}
By the constancy of parabolic inscribed angles,
\[
\measuredangle BAX = \measuredangle BDC = \measuredangle XDC.
\]
Moreover, $\angle AXB=\angle CXD$ are vertical angles.
Thus $\triangle XAB$ and $\triangle XCD$ have two equal angles (AA).
The second claim follows analogously.
\end{proof}

\subsection{A Hierarchy of Congruence in DA Geometry}
\begin{quote}
We now turn to congruence of DA triangles. As with similarity, it is natural to stratify congruence into tiers.
As a DA–specific, more permissive notion (a “quasi–congruence”), we first introduce \emph{DA–norm congruence}.
When the context is clear, we simply say “norm congruent.”
\end{quote}

\begin{definition}[DA–norm congruence]
Two DA triangles $\triangle_{\mathcal P} ABC$ and $\triangle_{\mathcal P} DEF$ are \emph{DA–norm congruent}
if all corresponding DA–norms of sides agree. We write
\[
\triangle_{\mathcal P} ABC \equiv^{\mathrm{norm}}_{\mathcal P} \triangle_{\mathcal P} DEF .
\]
\end{definition}

Finally, we define the strongest relation between two triangles, namely \emph{DA congruence}.
When the context is clear, we simply say “congruent.”

\begin{definition}[DA congruence]
Two DA triangles $\triangle_{\mathcal P} ABC$ and $\triangle_{\mathcal P} DEF$ are \emph{DA congruent}
if all corresponding DA–norms of sides agree and all corresponding oriented DA angles agree.
In this case we write
\[
\triangle_{\mathcal P} ABC \equiv_{\mathcal P} \triangle_{\mathcal P} DEF .
\]
\end{definition}

\begin{proposition}[Congruence and the circumparabola coefficients]
If $\triangle_{\mathcal P} ABC \equiv_{\mathcal P} \triangle_{\mathcal P} DEF$, then the absolute values of the quadratic coefficients
of the two circumparabolas are equal.
\end{proposition}

\begin{proof}
Suppose $\triangle_{\mathcal P} ABC \equiv_{\mathcal P} \triangle_{\mathcal P} DEF$ and let the quadratic coefficients of the circumparabolas
for $\triangle_{\mathcal P} ABC$ and $\triangle_{\mathcal P} DEF$ be $\kappa$ and $\kappa'$ respectively.
Write the $x$–coordinates of $A,B,C,D,E,F$ as $a,b,c,d,e,f$.
The magnitude of the oriented DA angle at $B$ is $|\kappa(c-a)|$ for $ABC$ and $|\kappa'(f-d)|$ for $DEF$.
Under DA congruence, the corresponding side norms satisfy $|c-a|=|f-d|$ and the angles are equal, hence
$|\kappa(c-a)|=|\kappa'(f-d)|$, which forces $|\kappa|=|\kappa'|$.
\end{proof}

\begin{proposition}[Limitation of norm congruence]\label{prop:norm-cong-limit}
Even if two triangles are DA–norm congruent, their corresponding angles need not agree.
Hence norm congruence is strictly weaker than DA congruence.
\end{proposition}

\begin{proof}
Consider $\triangle_{\mathcal P}ABC$ on $y=x^2$ and $\triangle_{\mathcal P}A'B'C'$ whose $x$–coordinates are $\kappa a,\kappa b,\kappa c$ with $\kappa>0$.
Then all side DA–norms agree, but
\[
\measuredangle_{\mathcal P} B=a-c,\qquad
\measuredangle_{\mathcal P} B'=\kappa(a-c),
\]
so the angles differ unless $\kappa=1$.
\end{proof}

Thus, even when DA–norms match, the DA angles may differ; this occurs precisely when the absolute values of the circumparabola coefficients differ.
Equivalently, we have:

\begin{proposition}[When norm congruence equals DA congruence]\label{prop:norm-to-cong}
Let $\Ptri ABC$ and $\Ptri DEF$ be DA–norm congruent.
Then the following are equivalent:
\begin{enumerate}
  \item $\Ptri ABC \equiv_{\mathcal P} \Ptri DEF$ (DA–congruent).
  \item The circumparabolas of $\Ptri ABC$ and $\Ptri DEF$ have equal absolute values of their quadratic coefficients.
\end{enumerate}
\end{proposition}

\begin{proof}
($1 \Rightarrow 2$) Under DA congruence, all corresponding DA angles agree.
Since the angle magnitudes scale with the circumparabola coefficient, the absolute values of the coefficients must be equal.

($2 \Rightarrow 1$) DA–norm congruence gives equality of all side DA–norms.
If, in addition, the absolute values of the quadratic coefficients coincide, the angle–scaling agrees as well; hence all corresponding DA angles are equal.
Therefore the triangles are DA–congruent.
\end{proof}

\begin{theorem}[Parabolic perimetric congruence]\label{thm:parabolic-perimetric-cong}
Let $\Ptri ABC$ be a triangle on a circumparabola $\Gamma$.
Let $A',B',C'$ be the points obtained by shifting each vertex along $\Gamma$ by the same difference angle $\theta$.
Then
\[
   \Ptri ABC \equiv_{\mathcal P} \Ptri A'B'C'.
\]
\end{theorem}

\begin{proof}
WLOG take $\Gamma:y=x^2$ and write $A(a,a^2)$, $B(b,b^2)$, $C(c,c^2)$ with $a<b<c$.
By hypothesis, the $x$–coordinates of $A',B',C'$ are $a+\theta$, $b+\theta$, $c+\theta$.
Hence
\[
 |A'B'|_{\mathcal P} = (b+\theta)-(a+\theta) = b-a = |AB|_{\mathcal P},
\]
and similarly $|B'C'|_{\mathcal P}=|BC|_{\mathcal P}$ and $|C'A'|_{\mathcal P}=|CA|_{\mathcal P}$.
Thus
\[
   \Ptri ABC \equiv^{\mathrm{norm}}_{\mathcal P} \Ptri A'B'C'.
\]
Both triangles lie on the same parabola $\Gamma$, so the quadratic coefficients coincide.
Therefore the corresponding DA angles also agree, and
\[
   \Ptri ABC \equiv_{\mathcal P} \Ptri A'B'C'
\]
follows.
\end{proof}

\begin{corollary}[Parabolic shift group]\label{cor:parabolic-shift-group}
For a fixed circumparabola $\Gamma$, consider the DA–shift
\[
   T_\theta : (x,y) \mapsto (x+\theta,\ (x+\theta)^2).
\]
Then:
\begin{enumerate}
  \item $T_\theta$ maps $\Gamma$ to itself.
  \item $T_\theta$ preserves parabolic perimetric (DA) congruence of triangles on $\Gamma$.
  \item $T_{\theta_1}\circ T_{\theta_2}=T_{\theta_1+\theta_2}$, hence $\{T_\theta \mid \theta\in\mathbb R\}$ is isomorphic to the additive group $(\mathbb R,+)$.
\end{enumerate}
\end{corollary}

Finally, we present a theorem that symbolically crowns the theory of congruence in DA geometry.  
This result captures the geometric beauty and depth of the system and serves as a fitting conclusion to the main body of this paper.

\begin{maintheorem}\label{mthm:parabolic-final-theorem}
Let $\gamma$ and $\delta$ be two parabolas whose axes are parallel to the same
projective direction~$d$ and which have the same quadratic coefficient~$\kappa$.
Let $\Ptri ABC$ be a triangle inscribed in $\gamma$, and
$\Ptri A'B'C'$ a triangle inscribed in~$\delta$,
satisfying
\[
\Ptri ABC \equiv_{\mathcal P} \Ptri C'B'A'
\qquad
\text{(orientation-reversing $\mathcal P$--congruence).}
\]
By drawing the $\mathcal P$--perpendiculars from
$A,B,C$ to $B'C',\,C'A',\,A'B'$, respectively,
let $H_A,H_B,H_C$ denote their feet.
Then the points $H_A,H_B,H_C$ are collinear on a line~$L$.
\end{maintheorem}

\begin{figure}[H]
  \centering
  \includegraphics[width=0.45\linewidth]{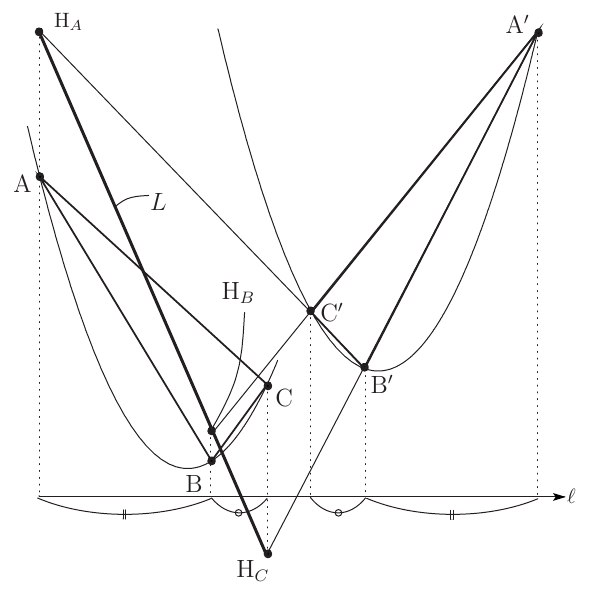}
\caption{The three feet $H_A,H_B,H_C$ of the difference--angle perpendiculars 
from $A,B,C$ to $B'C',C'A',A'B'$ lie on the same line $L$.}
\label{fig:mthm-diffangle_congruence_up}
\end{figure}

\begin{proof}
Assume $\triangle ABC \equiv_{\mathcal P} \triangle C'B'A'$ (on the same parabola, $\mathrm{norm\text{-}congruence}\Leftrightarrow \mathrm{congruence}$ for a common $k>0$).
Let
\[
AB=B'A'=a,\quad BC=C'B'=b,\quad CC'=\lambda.
\]
Take DA–norms as directed quantities, assigning a negative sign to those corresponding to external division.
Then
\[
\frac{\overrightarrow{AB}}{\overrightarrow{BH_C}}
\cdot
\frac{\overrightarrow{H_CH_A}}{\overrightarrow{H_AH_B}}
\cdot
\frac{\overrightarrow{H_BC}}{\overrightarrow{CA}}
=
(-1)\cdot\frac{a}{b+\lambda}\cdot\frac{a+b}{a}\cdot\frac{b+\lambda}{a+b}
= -1.
\]
By Menelaus' theorem, $H_A,H_B,H_C$ are collinear.
\end{proof}

\begin{remark}
The statement of Main Theorem~\ref{mthm:parabolic-final-theorem} remains valid 
even when the opening direction of the circumparabola is reversed (upward or downward). 
That is, the $\mathcal P$–congruence depends only on the projective direction~$d$ 
and on the magnitude of the quadratic coefficient~$\kappa$, 
not on its sign.
\end{remark}

\begin{figure}[H]
  \centering
\includegraphics[width=0.45\linewidth]{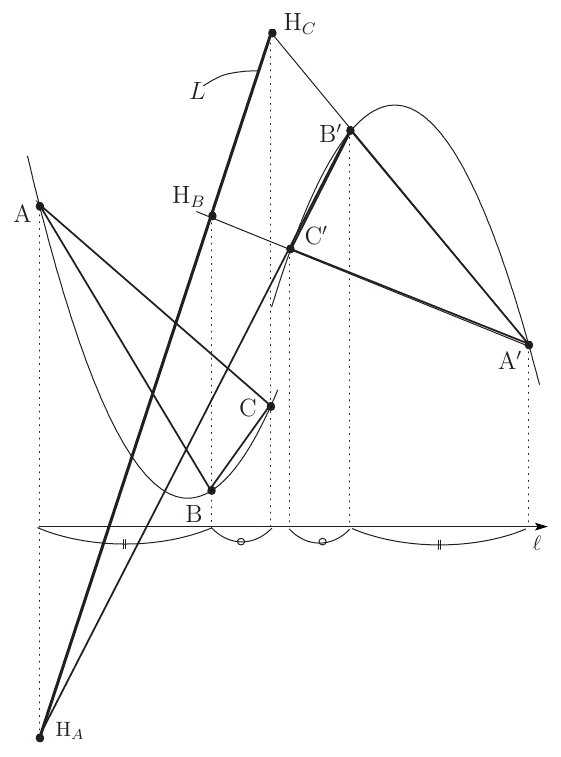}
\caption{ congruence for a downward-opening circumparabola. 
The same collinearity persists under the reversal of the parabola's orientation.}
\end{figure}

\begin{remark}
Although this theorem is presented as an application of 
$\mathcal P$–congruence in  geometry,
its Euclidean counterpart reveals a comparable geometric elegance
and a deep connection with the classical Simson theorem.
A detailed discussion of this relationship is omitted here,
as it would lead us away from the main line of the paper.
\end{remark}

\begin{remark}[Summary of this chapter]
In this chapter, we formalized the hierarchy of congruence in DA geometry and clarified the relationship between norm congruence and full congruence.  
Through parabolic perimetric congruence and its group structure, we revealed the rich diversity of the congruence theory and culminated with
\Cref{mthm:parabolic-final-theorem}, which concretely illustrates its depth.  
This result marks only the beginning—its generalization to polygons and to other conic curves remains an open and promising direction.
\end{remark}

\subsection{Equivalence of Figures in DA Geometry}

The notion of geometric equivalence in DA geometry exhibits a hierarchy distinct from that in Euclidean geometry.  
Between similarity and congruence naturally appears an intermediate notion called \emph{norm congruence}.
This stems from the independence between the DA norm and the DA angle, and from the role of the quadratic coefficient of the circumparabola
as a new invariant.

\Cref{tab:DA-equivalence} summarizes the three fundamental levels of equivalence in DA geometry.
DA similarity preserves the ratio of DA–norms and angles but changes the area and the coefficient of the circumparabola.
Norm congruence matches both norms and angles but may differ in the coefficient.
Full DA congruence requires equality of all four: norm, angles, area, and circumparabola coefficient.

\begin{table}[htbp]
\centering
\small
\begin{tabularx}{\linewidth}{@{}%
  >{\raggedright\arraybackslash}m{2.3cm}%
  >{\centering\arraybackslash}m{2.1cm}%
  >{\centering\arraybackslash}m{1.8cm}%
  >{\centering\arraybackslash}m{1.9cm}%
  >{\raggedright\arraybackslash}m{3.0cm}%
  >{\raggedright\arraybackslash}X@{}}
\toprule
Type & DA Norm & DA Angles & Area & Quadratic Coefficient of Circumparabola & Characteristic \\ 
\midrule
Similarity
  & Ratio fixed ($a\!:\!b$)
  & Preserved
  & $a^2\!:\!b^2$
  & $1\!:\!k^2$
  & Weakest equivalence \\ 
Norm congruence
  & Equal ($1\!:\!1$)
  & Preserved
  & May differ
  & Distinct
  & Reflects density difference \\ 
Full congruence
  & Equal ($1\!:\!1$)
  & Preserved
  & Equal
  & Equal in absolute value
  & Complete equivalence \\ 
\bottomrule
\end{tabularx}
\caption{Comparison of equivalence relations in DA geometry}
\label{tab:DA-equivalence}
\end{table}

This stratification demonstrates that DA geometry is not merely an extension of Euclidean geometry
but a framework with its own intrinsic invariants and refined hierarchy of equivalence.
It also consolidates the structural essence of the preceding theory
and provides a natural bridge toward the concluding reflections of the paper.


\section{Conclusion and Outlook}

In this paper, we have clarified the hierarchical structure of similarity and congruence in DA geometry.  
In contrast to Euclidean geometry, it was shown that SSS similarity does not necessarily imply AA similarity,  
and that this phenomenon arises as a structural consequence of the quadratic coefficient of the circumparabola.  
This fact demonstrates that DA geometry is not merely a variant of distance geometry,  
but rather an independent system grounded on the axioms of angles.

The significance of this observation can be summarized in two points.  
First, it shows that angles and norms can be mathematically separated in a consistent manner,  
providing a new geometric perspective centered on the axioms of angles.  
Second, it situates the framework of DA geometry naturally within the context of Hilbert’s Fourth Problem.  
In particular, the DA norm can be interpreted as a pre--Finsler structure,  
and thus DA geometry can be understood as an intermediate system  
between Euclidean and Finsler geometries.

As future work, we intend to systematically organize the intrinsic properties of DA geometry,  
such as the area formula, the median theorem, and the inner--product structure.  
Through these investigations, we aim to establish DA geometry as a consistent axiomatic system  
and to present a new geometric vision from the standpoint of first--order angular geometry.


\appendix
\section{Vertical Angles and Boundary Policies}\label{app:vertical-angle-boundary}

This appendix supplements the logical framework of the angle axioms \cref{ax:A1}–\cref{ax:A5},
clarifying the treatment of singular directions and the consistency
of angle definitions under the absorption-type boundary policy.
It ensures that the axiomatic structure of DA geometry
remains logically closed even across singular boundaries.

In the proof of \Cref{prop:vertical-angle} (equality of DA vertical angles),
the ambiguity arising from crossing a singular direction $\Cut$
was implicitly handled by adopting the absorption-type boundary policy.
Although the computation appears to involve only simple slope differences,
the definition of the difference angle can become ambiguous
when the intersection point $P$ lies on or crosses $\Cut$.
Here we give a formal treatment of this issue.

\begin{definition}[Two sides of the cut and boundary--zero pairing]\label{def:cut-pairing}
Let $\Cut$ denote the singular direction at a point $P$,
and let $D^+, D^-$ be the two connected components of $\Dir{P}\setminus\{\Cut\}$.
Assign boundary tags $\epsilon^+$ and $\epsilon^-$ to each side,
and declare a local equivalence $\epsilon^+ \sim_D \epsilon^-$.
For directions along $PX$ and $PY$, define
\[
\measuredangle_{\mathcal P}(PX,PY)
:=(\phi_*(\overrightarrow{PY})+\epsilon^*)
-(\phi_\star(\overrightarrow{PX})+\epsilon^\star)
\quad (*,\star\in\{+,-\}),
\]
and finally reduce modulo~$\sim_D$.
The equivalence $\sim_D$ corresponds to modding out the relation
$\epsilon^+ - \epsilon^- = 0$, forming a quotient structure.
\end{definition}

\begin{remark}
Geometrically, this pairing identifies the two infinitesimal sides
of the singular direction as a single ``absorbed'' zero boundary.
It is analogous to treating the two sides of a folded tangent at infinity
as one boundary point.
\end{remark}

\begin{lemma}[Independence of cutting side]\label{lem:cut-indep}
By \Cref{def:cut-pairing}, the ambiguity of whether or not one crosses $\Cut$
is eliminated, and the angle is well-defined.
\end{lemma}

\begin{proposition}[Sign inversion of negative difference angles (no A' required)]\label{prop:neg-flip}
Let $P\in AA'$, $B\notin AA'$, and $PB\not\parallel d$.
If $\angle APB=\theta$ and $PB\in D^+$, $PA\in D^-$,
then $\angle BPA=-\theta$.
Hence the definition of the bisector of a negative difference angle
is well-defined, independent of the cutting side.
\end{proposition}

\begin{proof}
By definition,
\[
\angle APB=(\phi_+(\overrightarrow{PB})+\epsilon^+)
-(\phi_-(\overrightarrow{PA})+\epsilon^-),\qquad
\angle BPA=(\phi_-(\overrightarrow{PA})+\epsilon^-)
-(\phi_+(\overrightarrow{PB})+\epsilon^+).
\]
Reducing under $\epsilon^+\sim_D\epsilon^-$
gives the sign-reversed relation between the two angles.
\footnote{
Identifying the zero on the singular direction independently of $d^+$ and $d^-$
corresponds to the absorption-type calibration.
If one adopts the lift-type calibration,
the ``origin'' of~$0$ can be retained as an infinitesimal quantity
in a nonstandard analytic model,
and taking the standard part yields the same result.
For brevity, the absorption-type is adopted in this paper.
}
\end{proof}

\begin{remark}
Consequently, the results of \Cref{prop:vertical-angle}
and \Cref{lem:straight-angle-zero}
remain well-defined even within this formal framework
including boundary policies.
\end{remark}

\section{Logical Equivalence between the Vertical--Angle Law and the Flat--Angle Axiom}\label{app:vertical-flat-equiv}

\begin{lemma}[Equivalence of the vertical--angle law and the flat--angle axiom under the absorption policy]\label{lem:vertical-flat-equiv}
Fix the absorption-type boundary policy in A6 (continuity and boundary conditions),
and assume A1 (antisymmetry) and A2 (additivity).
Then the following two statements are equivalent:
\begin{enumerate}
  \item[(V±)] (\emph{Oriented vertical--angle law})  
  If the lines $AB$ and $CD$ intersect at $P$, then $\angle APC = -\angle DPB$.
  \item[(F)] (\emph{Flat--angle axiom})  
  If $A, P, B$ are collinear, then $\angle APB = 0$.
\end{enumerate}
\end{lemma}

\begin{proof}
\textbf{(F $\Rightarrow$ V)}\;
By A2 (additivity),
$\angle APC+\angle CPB=\angle APB=0$,
so $\angle APC=-\angle CPB$.
Similarly, $\angle DPB=-\angle APD$,
hence $\angle APC=\angle DPB$.

\smallskip
\textbf{(V $\Rightarrow$ F)}\;
Take collinear points $A,P,B$ and any line $CD$ through $P$.
Let $C\to B$ and $D\to A$ along the line $AB$.
By additivity,
$\angle APB=\angle APC+\angle CPB=\angle DPB+\angle APD$ (using~V).
Under the absorption policy,
as $C\to B$ and $D\to A$, each term is absorbed into zero,
hence $\angle APB=0$.
\end{proof}

\begin{remark}[Choice of primitive depends on the system]
In DA geometry, the oriented vertical–angle law (V)
arises directly from the definition via slope difference.
In classical Euclidean axiomatics,
the flat–angle constancy (F) is usually adopted as a normalization,
from which (V) follows as a theorem.
Within the calibration framework, however,
under the absorption policy the two statements are mutually derivable
by \Cref{lem:vertical-flat-equiv},
and are therefore essentially equivalent.
\end{remark}


\section{Pseudo--Metric Nature of the DA  norm}\label{app:pseudo-metric}

In DA geometry, the difference--angle norm $|AB|_{\mathcal P}$ defined in
\Cref{def:diff-angle-norm}
satisfies two of the axioms of distance by definition:
nonnegativity and symmetry.
Moreover, by \Cref{mthm:triangle-equation},
the triangle inequality always holds as an equality.
On the other hand, because of the presence of the projective direction
(the singular line),
the norm does not satisfy positive definiteness.
Hence the vector space endowed with the DA  norm
is an extremely singular pseudo--metric space.

A simple example of such a singular space can be given
by a one--dimensional injective map $f$
through
\[
d(x,y)=|f(x)-f(y)|.
\]
This satisfies the same set of properties—nonnegativity, symmetry,
and equality in the triangle inequality—but fails to be positive definite
if $f$ is constant along some direction.

In other words, the DA norm can be regarded
as an extension of this low--dimensional pseudo--metric behavior
to a higher--dimensional setting by introducing a singular direction.
Through this viewpoint,
the study of geodesics in DA geometry
becomes entirely natural:
since the geometry itself admits any continuous function graph as a geodesic,
one may in particular select all linear functions---that is,
straight lines---as geodesics of the system.

This observation provides the foundation for interpreting DA geometry as a degenerate Finsler-type structure, whose geodesics coincide with affine functions.


\section{Choice of Axioms and the Parabolic Extension of Hilbert's System}\label{app:choice-of-axioms}

\begin{quote}
In this paper, while several possible axiom systems are taken into consideration,  
special attention is given to the axioms of congruence, which are formulated  
so that the angle itself becomes a \textit{primary geometric quantity}.  
Whereas the Hilbert system treats the angle as subordinate to the metric structure,  
the present framework emancipates the angle as a primary quantity,  
thereby removing Axiom~III$_4$ (angle transfer) and adopting instead  
the ASA-type congruence principle as a new foundation.  
Following Hilbert's standpoint in his \textit{Foundations of Geometry},  
the clarification of the notion of congruence is ultimately grounded  
in the axioms of motion in Euclidean geometry  
and is reduced to the congruence of triangles.

From this viewpoint, geometry as a discipline  
admits multiple coherent systems depending on the chosen set of axioms.  
Among them, the following three are the most fundamental in determining  
the nature of a geometric framework:
\[
\begin{array}{lll}
\text{(I)} & \text{Axiom of Congruence},\\
\text{(II)} & \text{Axiom of Parallel Lines},\\
\text{(III)} & \text{Axiom of Continuity (or Divisibility)}.
\end{array}
\]

In this appendix, we focus on the first two---the axioms of congruence and of parallels---  
and summarize how existing geometric systems can be classified  
according to their respective choices.  
The discussion on the \textit{Axiom of Continuity} will be addressed separately in another paper.
\end{quote}

\subsection{Fundamental Congruence Axioms (SAS, ASA, and SSS)}

\begin{quote}
This appendix reformulates the Hilbert group~III (axioms of congruence)  
in a parabolic---that is, Difference--Angle---extension,
and organizes the three fundamental forms of triangle congruence:  
SAS, ASA, and SSS.  
Each form reflects which quantity---angle or distance---is regarded as primary,  
and each generates a distinct ``language'' of geometry.

Hilbert himself did not postulate triangle congruence directly,  
but derived the familiar SAS-type congruence as a theorem  
by combining the axioms of angle transfer (III$_4$)  
and the relational axiom (III$_5$).  
In this paper, we refer to this standpoint as the ``SAS type'' for convenience,  
and contrast it with the ``ASA type,''  
in which angles are treated as primary quantities,  
and the ``SSS type,'' in which only distances are primary.  
This comparison clarifies that the choice of axioms  
determines not only the structure of geometry but also its fundamental language.
\end{quote}

\begin{axiombox}{CONG--SAS(Side--Angle--Side : Hilbert Type)}{CONG-SAS}
{}For any two triangles $\triangle ABC$ and $\triangle DEF$,  
if
\[
AB = DE, \quad AC = DF, \quad \angle BAC = \angle EDF,
\]
then
\[
\triangle ABC \equiv \triangle DEF.
\]
This form characterizes a rigid geometry that preserves  
the symmetry between distances and angles,  
corresponding to Hilbert's axiomatic geometry.
\end{axiombox}

\begin{axiombox}{CONG--ASA(Angle--Side--Angle : Angle--Primary Type)}{CONG-ASA}
{}For any two triangles $\triangle ABC$ and $\triangle DEF$,  
if
\[
AB = DE, \quad 
\angle CAB = \angle FDE, \quad 
\angle CBA = \angle FED,
\]
then
\[
\triangle ABC \equiv \triangle DEF.
\]
This form characterizes a geometry in which the angle is adopted as the sole primary quantity,  
and all secondary quantities such as distance and area are reconstructed from angular relations.  
It corresponds to the difference--angle geometry and its parabolic extensions.
\end{axiombox}

\begin{axiombox}{CONG--SSS(Side--Side--Side : Distance--Primary Type)}{CONG-SSS}
{}For any two triangles $\triangle ABC$ and $\triangle DEF$,  
if
\[
AB = DE, \quad BC = EF, \quad CA = FD,
\]
then
\[
\triangle ABC \equiv \triangle DEF.
\]
This form characterizes a distance geometry  
in which only distances are primary quantities,  
and angles are reconstructed secondarily from metric relations.
\end{axiombox}

\begin{remark}[Adopted Framework]\label{rem:adoption}
In this paper, we adopt \Cref{ax:CONG-ASA} (ASA type)  
as the foundational axiom,  
establishing a system in which the angle is the unique primary geometric quantity  
from which all other quantities are derived.  
The SAS and SSS forms are presented here for comparison and reference.
\end{remark}


\subsection{Axiom of Angle Transfer and Its Necessity}

\begin{quote}
The preceding discussion has focused primarily on CONG5 (\Cref{ax:CONG5}),  
but in systems based on SSS congruence, or in axiomatic frameworks that do not introduce angles as primitive quantities,  
the adoption of CONG4 becomes relatively important.
\end{quote}

\begin{remark}[Omission and Significance of CONG4]\label{rem:cong4-role}
In Hilbert's system, angles were subordinate to the distance structure,  
and therefore their transfer (Axiom~III$_4$) was required as an explicit axiom.  
In the present framework, however, angles are independent first-order quantities,  
defined by Axioms~A1--A5.  
Hence, the transfer of angles is no longer necessary as an independent axiom.  
Accordingly, CONG4 is left vacant,  
and the distributive relation connecting angles and lengths (the ASA congruence) is provided by CONG5.
\end{remark}

\begin{table}[h]
\centering
\caption{Correspondence between Congruence Bases and the Necessity of CONG4}
\label{tab:cong-bases-summary}
\begin{tabularx}{\textwidth}{
  >{\raggedright\arraybackslash}p{2.8cm}
  >{\raggedright\arraybackslash}p{3.0cm}
  >{\centering\arraybackslash}p{2.8cm}
  >{\raggedright\arraybackslash}X
}
\toprule
Basis type & Primary quantities & Need for CONG4/5 & Corresponding geometry \\
\midrule
SAS (\Cref{ax:CONG-SAS}) & Distance and angle (symmetric) & Required (including III$_4$) &
Hilbert / Cayley--Klein / Euclidean \\[1ex]
ASA (\Cref{ax:CONG-ASA}) & Angle and scale & Not required (angle autonomous) &
Difference--Angle geometry (parabolic limit) \\[1ex]
SSS (\Cref{ax:CONG-SSS}) & Distance only &
\textbf{Not required in principle\footnote{Unless angles are introduced as primary quantities.}} &
Distance geometry / CAT(0)--type structure \\
\bottomrule
\end{tabularx}
\end{table}

\begin{quote}
Through this classification, the ASA-type axiomatic system developed in this paper  
can be positioned as a parabolic and angle-first extension of Hilbert's system.  
Even without the angle-transfer axiom,  
it retains a consistent principle of congruence,  
representing the essence of parabolic angular quantity in its purest form.  
Therefore, this axiomatic structure is not merely a subsystem of Hilbert's geometry,  
but rather the \emph{canonical form of first-order angular geometry} that emerges only in the parabolic limit.
\end{quote}

\subsection{Axiom of Parallel Lines and the Branching of Geometry}

Among the following three forms, exactly one shall be adopted.  
They are mutually exclusive, and the choice determines the curvature type of the geometry.

\begin{axiom}[PAR--Euc (Euclidean type)]
For any line $\ell$ and any point $P$ not lying on $\ell$,  
there exists exactly one line through $P$ that is parallel to $\ell$.
\end{axiom}

\begin{axiom}[PAR--Hyp (Hyperbolic type)]
For any line $\ell$ and any point $P$ not lying on $\ell$,  
there exist at least two distinct lines through $P$ that are parallel to $\ell$.
\end{axiom}

\begin{axiom}[PAR--Elp (Elliptic type)]
For any line $\ell$ and any point $P$ not lying on $\ell$,  
there exists no line through $P$ that is parallel to $\ell$.
\end{axiom}

\begin{table}[H]
\centering
\caption{Axiom of Parallel Lines and the Classification of Geometries}
\begin{tabular}{lcl}
\toprule
Adopted form & Curvature & Corresponding geometry \\
\midrule
PAR--Elp (no parallels) & positive & Elliptic geometry \\
PAR--Euc (one parallel) & zero & Euclidean geometry, DA geometry \\
PAR--Hyp (multiple parallels) & negative & Hyperbolic geometry \\
\bottomrule
\end{tabular}
\end{table}

\begin{remark}[Summary of the Appendix]
The results of this appendix supplement the consistency of angles under the boundary policy,
the equivalence between the constancy of vertical and straight angles,
and the selection principles of congruence and parallelism.
These confirm that the axiomatic system of DA geometry
is self-contained, including its behavior on the boundary.
\end{remark}

\bibliographystyle{plain}
\bibliography{diff_angle_ref}

\begin{thebibliography}{1}

\bibitem{Hilbert1899}
David Hilbert.
\newblock {\em Grundlagen der Geometrie}.
\newblock Teubner, Leipzig, 1899.
\newblock English translation: \emph{Foundations of Geometry}, Springer, 1971.
  Japanese translation by K. Nakamura, Chikuma Gakugei Bunko, 1996.

\bibitem{WeissOdehnal2024}
Gunter Weiss and Boris Odehnal.
\newblock Miquel's theorem and its elementary geometric relatives.
\newblock {\em KoG}, 28(28):11--24, 2024.

\end{thebibliography}


\section*{Acknowledgments}

The author expresses his deepest gratitude to his wife, Junko, for her constant support throughout the preparation of this work.
He is also indebted to his friend from his graduate days at Tohoku University, Kanato Takeshita, whose frank opinions and constructive advice helped clarify the direction of the research.

I am deeply indebted to Prof.\ Satoshi Ishiwata (Yamagata University),
who kindly arranged an opportunity to discuss the work with me after the
submission of v1 and offered many valuable pieces of advice.
I express my profound gratitude here.

As a high-school teacher with limited time for research, the author gratefully acknowledges the assistance of AI-based tools, which supported the comparative study of existing geometrical frameworks, the construction of the TeX environment, and editorial organization.
All conceptual developments and mathematical results presented in this paper are entirely due to the author, yet without this assistance the completion of the manuscript within about three months would not have been possible.

The author further thanks Professors Athanase Papadopoulos and Charalampos Charitos for their kind advice and encouragement during the arXiv submission process, and Professor Jean-Marc Schlenker, whose endorsement made the public release of this work possible.
Their warm support gave the author the courage to bring independent research into the world.

Finally, the author expresses deep respect and gratitude to the great mathematicians who inspired many of the ideas developed here—Felix Klein, David Hilbert, and all those who offered thoughtful advice during the endorsement process.
It is the author’s hope that this geometry of difference angles may, in some quiet way, reach those great minds who left us a century ago.

Although still far from a complete understanding of geometry itself, the author has been taught through this journey the importance of viewing mathematics from many perspectives.

\end{document}